\documentclass[11pt]{amsart}
\makeatletter
\@ifundefined{pdfsuppressptexinfo}{}{\pdfsuppressptexinfo=-1 }
\@ifundefined{pdfinfoomitdate}{}{\pdfinfoomitdate=1 }
\@ifundefined{pdftrailerid}{}{\pdftrailerid{}}
\makeatother

\usepackage{amsmath,amssymb,amsthm,mathrsfs}
\usepackage[margin=1in]{geometry}
\usepackage[expansion=false]{microtype}
\usepackage{url}
\providecommand{\doi}[1]{\textsc{doi:}\,\url{https://doi.org/#1}}
\usepackage{enumitem}

\theoremstyle{plain}
\newtheorem{theorem}{Theorem}[section]
\newtheorem{proposition}[theorem]{Proposition}
\newtheorem{lemma}[theorem]{Lemma}
\newtheorem{corollary}[theorem]{Corollary}
\newtheorem{conjecture}[theorem]{Conjecture}
\newtheorem{hypothesis}[theorem]{Hypothesis}
\theoremstyle{definition}
\newtheorem{remark}[theorem]{Remark}
\newtheorem{example}[theorem]{Example}

\newcommand{\Q}{\mathbb{Q}}
\newcommand{\Z}{\mathbb{Z}}
\newcommand{\F}{\mathbb{F}}
\newcommand{\C}{\mathbb{C}}
\newcommand{\R}{\mathbb{R}}
\newcommand{\GU}{\mathrm{GU}}
\newcommand{\GSp}{\mathrm{GSp}}
\newcommand{\Sp}{\mathrm{Sp}}

\newcommand{\Aut}{\mathrm{Aut}}
\newcommand{\Tr}{\mathrm{Tr}}
\newcommand{\Res}{\mathrm{Res}}
\newcommand{\Upr}{U_{\mathrm{pr}}}
\newcommand{\Va}{\mathrm{Va}}
\newcommand{\snew}{S^{\mathrm{new}}}
\newcommand{\eps}{\varepsilon}
\newcommand{\Onorm}{\mathcal{O}}
\newcommand{\one}{\mathbf{1}}
\DeclareMathOperator{\charp}{charpoly}
\DeclareMathOperator{\Jac}{Jac}
\DeclareMathOperator{\diag}{diag}

\usepackage{booktabs}
\usepackage{graphicx}
\usepackage{pgfplots}
\pgfplotsset{compat=1.16}

\providecommand{\kro}[2]{\left(\frac{#1}{#2}\right)}

\providecommand{\Upr}{U_{\mathrm{pr}}}
\providecommand{\Va}{\mathrm{Va}}
\providecommand{\snew}{S^{\mathrm{new}}}
\providecommand{\charp}{\operatorname{charpoly}}

\usepackage[colorlinks=true,linkcolor=black,citecolor=black,urlcolor=blue]{hyperref}
\hypersetup{
	pdftitle={The Richelot isogeny graph and paramodular forms of weight 3: Atkin-Lehner signs and certified spectra},
	pdfauthor={Hung Dang},
	pdfsubject={},
	pdfkeywords={},
	pdfcreator={},
	pdfproducer={}}
\IfFileExists{orcidlink.sty}{\usepackage{orcidlink}}{\providecommand{\orcidlink}[1]{}}
\usepackage{cleveref}

\title[The Richelot graph and paramodular forms of weight 3]{The Richelot isogeny graph and paramodular forms of weight $3$: Atkin--Lehner signs and certified spectra}
\author{Hung Dang\,$^{\orcidlink{0009-0006-3272-0573}}$}
\address{Department of Mathematics, University of Phuong Dong, 171 Trung Kinh
	Street, Yen Hoa, Hanoi, Vietnam}
\email{hung.dt@phuongdong.edu.vn}
\date{October 15, 2026}

\begin{document}
	
	\begin{abstract}
		The Richelot $(2,2)$-isogeny graph $G_p$ of superspecial abelian
		surfaces in characteristic $p$ is the genus-$2$ analogue of the
		supersingular isogeny graph: its weighted adjacency matrix is a Brandt
		matrix for a compact inner form of $\GSp_4$, and it commutes with an
		Atkin--Lehner involution $R(\pi)$. Refining Ibukiyama's conjectures,
		we conjecture the automorphic origin and the $R(\pi)$-sign of every
		eigenvalue: Eisenstein, Saito--Kurokawa, Yoshida and
		paramodular eigenvalues, and pairs of type-$\Va$ eigensystems that only
		$R(\pi)$ separates. We compute $G_p$ at $33$ primes, with up to $2176$
		vertices, and certify the predicted eigenvalues and signs in exact
		arithmetic; at
		$p=167$ and $173$, where the sign prediction first becomes
		non-trivial, the graph carries the predicted eigenvalue $-8$ with sign
		$-1$. For every prime we prove that each type-$\Va$ parameter
		contributes exactly one eigenline of each $R(\pi)$-sign, a closed
		formula for the trace of $R(\pi)$, and a Fricke sign bias. If no further CAP representations occur, the non-trivial
		spectrum lies in $[-8\sqrt2,6+4\sqrt2]$, both ends being sharp under
		the conjecture, and only Saito--Kurokawa eigenvalues exceed $8\sqrt2$;
		at all $104$ primes $p\le601$ the extreme eigenvalue computed by
		Florit and Smith agrees with the largest Saito--Kurokawa one. Finally,
		the irrational roots of the residual square factor form one Galois
		orbit with full symmetric Galois group at each certified prime
		$p\ge29$, and
		the spectrum appears to follow a Plancherel law.
	\end{abstract}
	
	\maketitle
	
	\section{Introduction}\label{sec:intro}
	
	\subsection{The problem}\label{ss:problem}
	
	Mestre's \emph{m\'ethode des graphes} realizes Hecke operators on modular
	forms through isogeny graphs of supersingular elliptic curves
	\cite{Mestre1986}. In dimension two, the corresponding arithmetic objects
	are the Brandt matrices of binary quaternion Hermitian forms
	\cite{Hashimoto1980,HashimotoIbukiyama}, and their geometric realization is
	the Richelot $(2,2)$-isogeny graph on superspecial principally polarized
	abelian surfaces. The degree-$2$ Brandt operator is the weighted adjacency
	operator of this graph \cite[Thm.~39]{JordanZaytman}.

	Richelot isogenies are also central to isogeny-based cryptography, from
	the key-recovery attack on SIDH \cite{CastryckDecru} to hash functions
	and algorithms built on random walks, whose analysis depends on how
	fast walks on the graph mix. The graph is in general not Ramanujan
	\cite{JordanZaytman}, and its spectral gap has been studied numerically
	\cite{FloritSmith} and bounded through the automorphic spectrum
	\cite{AikawaTanakaYamauchi,CKTW}. Two questions therefore meet in the
	Richelot graph: which automorphic forms produce its eigenvalues, and
	where its spectrum lies. This paper answers the first eigenvalue by
	eigenvalue, and deduces the second.
	
	Let $B=B_{p,\infty}$ be the definite quaternion algebra ramified exactly
	at $p$ and $\infty$, and let $\Upr(p)$ denote the principal maximal
	parahoric genus. The space $M_{0,0}(\Upr(p))$ of functions on its classes
	has two compatible descriptions. On the geometric side, it is the weighted
	vertex space of the Richelot graph $G_p$. On the automorphic side, its Brandt
	operators realize Hecke operators on a compact inner form of $\GSp_4$.
	The two-sided ideal of reduced norm $p$ induces a commuting involution
	$R(\pi)$.
	
	The principal genus is more subtle than the non-principal genus. The
	non-principal trivial-weight space is related to weight-$3$ paramodular
	forms by the correspondence of \cite[\S9]{DPRT}. For the principal genus, Ibukiyama's
	conjectures describe how the same spectrum should be redistributed among
	Arthur types. Their dimension-level consequences are known, but dimensions
	and Hecke eigenvalues away from $p$ do not separate the two members of a
	type-$\mathrm{Va}$ pair. This paper uses the $R(\pi)$-sign to
	separate them, and computes the graph explicitly to test the resulting
	predictions eigenvalue by eigenvalue.

	Our results are of four kinds: a conjectural refinement of Ibukiyama's
	principal-genus conjectures (Conjecture~\ref{conj:refined-intro}); a
	computer-assisted theorem establishing it for every prime
	$11\le p\le149$ and for $p=167,173$ (Theorem~\ref{thm:certified-intro});
	theorems valid for every prime that prove parts of it or are suggested
	by the data (Theorems~\ref{thm:occurrence-intro}
	and~\ref{thm:structural-trace-intro}, and Theorem~\ref{thm:bias}); and
	consequences for the spectrum of the graph
	(Theorem~\ref{thm:spectrum-intro}). Table~\ref{tab:logical-status}
	records which is which.

	\subsection{The conjecture and the computation}\label{ss:results-status}

	An occurring Eisenstein, Saito--Kurokawa, Yoshida, or general-type
	parameter contributes to the spectrum of $B_2(2)$ the eigenvalue
	\[
	15,\qquad a_2(g)+6,\qquad a_2(g)+2a_2(f),\qquad a_2(F)
	\]
	respectively, by the transfer to $\GSp_4$ and the spin normalization at
	$2$ (Theorem~\ref{thm:transfer-framework}). Which parameters occur in the
	principal genus, with which multiplicity and which $R(\pi)$-sign, is the
	content of the following conjecture.

	\begin{conjecture}[Eigenvalue--sign refinement; verified for $11\le p\le149$ and $p=167,173$]\label{conj:refined-intro}
		For every prime $p\ge11$, the characteristic polynomial of $B_2(2)$
		decomposes into an Eisenstein factor, Saito--Kurokawa factors,
		opposite-sign Yoshida factors, an exact square
		$\mathrm{Va}_p(x)^2$, and a general-type factor $N_p(x)$. The involution
		$R(\pi)$ has the predicted signs on the Eisenstein, Saito--Kurokawa,
		Yoshida, and general-type components, and each type-$\mathrm{Va}$ pair
		splits into one $+1$ and one $-1$ eigenline.
	\end{conjecture}
	The full statement is Conjecture~\ref{conj:eigenvalue-sign} in
	Section~\ref{sec:statements}. The accompanying decomposition of the total
	dimension is equivalent to Ibukiyama's unconditional dimension identity;
	the new content is the eigenvalue-by-eigenvalue assignment and the
	$R(\pi)$-signs.

	The conjecture is suggested and tested by computing the graph.
	At $p=19$ the matrix $B_2(2)$ has the double eigenvalue $-2$, whose
	eigenspace is spanned by a vector fixed by $R(\pi)$ and a vector negated
	by it, as predicted for a type-$\Va$ pair, whose two eigensystems agree
	away from $19$ and are separated only by their Atkin--Lehner signs; the
	certificate itself confirms the agreement at $T(2)$ (\S\ref{ss:ihara19}). At
	$p=61$, the least prime level carrying a weight-$3$ paramodular non-lift,
	the $128\times128$ matrix has the simple eigenvalue $-7$; this is the
	eigenvalue at $2$ of that non-lift, computed by Poor and Yuen
	\cite[\S8, Ex.~1]{PoorYuen2015} (\S\ref{ss:p61}). At $p=167$ and
	$173$, the first prime levels with a Fricke-plus non-lift, that non-lift
	has $a_2=-8$, and the conjecture predicts the eigenvalue $-8$ with
	$R(\pi)$-sign $-1$; the computed graphs, with $1978$ and $2176$ vertices,
	confirm both (\S\ref{ss:fricke-plus}).

	\begin{theorem}[Certified spectra]\label{thm:certified-intro}
		For every prime $11\le p\le149$ and for $p=167,173$, the
		exact-arithmetic certificates in the
		archived Richelot--Brandt artifact verify the factorization and the
		$R(\pi)$-sign decomposition predicted by
		Conjecture~\ref{conj:refined-intro}. The general-type factor equals
		the characteristic polynomial of the spin $T(2)$ on the weight-$3$
		paramodular non-lifts of level $p$ in the database of
		Assaf--Ladd--Rama--Tornar\'ia--Voight.
	\end{theorem}

	The full statement is given as Theorem~\ref{thm:certified-spectra}, and
	its proof is in Section~\ref{sec:certificates}. A geometric engine
	produces the Richelot matrix, its weights and the $R(\pi)$-permutation
	from finite-field arithmetic alone; separate pipelines supply the
	weight-$2$ and weight-$4$ data. Exact polynomial division and isotypic
	projector traces then verify the predicted factorization and signs
	(Proposition~\ref{prop:reduction}). The certificates are replayed from
	frozen records by scripts that do not import the generating engines.
	Table~\ref{tab:ext} lists the block degrees at all $33$
	primes.

	\subsection{Results for every prime}

	The square and sign assertions of the type-Va clause hold for every
	prime.

	\begin{theorem}[Occurrence and splitting of the type-Va pairs]\label{thm:occurrence-intro}
		For every prime $p\ge11$ and every general-type parameter whose local
		component at $p$ is of type $\Va$ and which contributes to
		$M_{0,0}(\Upr(p))$, both members of the local
		$L$-packet at $p$ occur in $M_{0,0}(\Upr(p))$, each with automorphic
		multiplicity one and each contributing a single line, on which
		$R(\pi)$ acts by $+1$ for one member and by $-1$ for the other. In
		particular the type-$\Va$ contribution to the characteristic
		polynomial of $B_2(2)$ is an exact square split evenly by $R(\pi)$.
		Conversely, every parameter on the split side satisfying the
		conditions of Theorem~\ref{thm:occurrence}\textup{(d)} occurs.
	\end{theorem}
	The full statement is Theorem~\ref{thm:occurrence}. It follows from the
	global lifting theorem of R\"osner and Weissauer
	\cite[Thm.~11.4]{RoesnerWeissauer2021}, the local Langlands
	correspondence for $\GU_2(B_p)$ (Remark~\ref{rem:localpacket}), and
	Schmidt's computation of parahoric fixed vectors for this group
	(Remark~\ref{rem:local-constant}). More generally, the general-type
	multiplicities are proved (Remark~\ref{rem:other-rows}); what remains
	conjectural in Conjecture~\ref{conj:refined-intro} is the occurrence
	pattern of the Saito--Kurokawa and Yoshida rows and the $R(\pi)$-sign
	on $N_p$.

	The signed total of the predicted spectrum is also a theorem, for every
	prime $p\ge7$.

	\begin{theorem}[Structural trace identity]\label{thm:structural-trace-intro}
		For every prime $p\ge7$,
		\[
		\Tr\!\left(R(\pi)\mid M_{0,0}(\Upr(p))\right)
		=
		1+\dim\snew_4
		+\dim S_2^-\dim S_4^+
		-\dim S_2^+\dim S_4^-
		+d(p),
		\]
		where
		\[
		d(p)=\bigl(\dim S_3(K(p))-\dim S_4^-\bigr)-2\dim S_3^+(K(p)).
		\]
	\end{theorem}
	The proof (Section~\ref{sec:trace}) compares the closed trace formulas
	of the two genera \cite{IbukiyamaKatsura,Ibukiyama2019quinary}, whose
	difference reduces to a single class-number term
	(Proposition~\ref{prop:genus-trace-comparison}), and combines the
	comparison with the non-principal-genus correspondence. Moreover
	$d(p)\ge0$ for every prime (Theorem~\ref{thm:bias}): Fricke-minus
	non-lifts are never outnumbered.

	\subsection{The spectrum of the graph}\label{ss:spectrum-intro}

	The factorization locates the eigenvalues of $G_p$ that govern the
	mixing of random walks. Hypothesis~\ref{hyp:types} below states that no
	CAP representations other than the trivial one and the Saito--Kurokawa
	ones with trivial character contribute to the principal genus.

	\begin{theorem}[Spectrum of the Richelot graph]\label{thm:spectrum-intro}
		Assume Hypothesis~\ref{hyp:types}. For every prime $p\ge11$, every
		eigenvalue $\lambda\ne15$ of $B_2(2)$ satisfies
		$-8\sqrt2\le\lambda\le6+4\sqrt2$, and $|\lambda|>8\sqrt2$ only for
		Saito--Kurokawa eigenvalues $\lambda=a_2(g)+6$. If moreover the
		Saito--Kurokawa and Yoshida clauses of
		Conjecture~\ref{conj:refined-intro} hold, then the second-largest
		eigenvalue tends to $6+4\sqrt2$ and the smallest to $-8\sqrt2$ as
		$p\to\infty$.
	\end{theorem}
	The statements are Theorem~\ref{thm:spectral} and
	Proposition~\ref{prop:lambda2-limit}. The upper bound was obtained
	before by Corte-Real Santos, Kir, Ta\"ibi and Wesolowski \cite{CKTW},
	conditionally on the twisted weighted fundamental lemma; the lower
	bound, the identification of the eigenvalues above $8\sqrt2$, and the
	sharpness at both ends are new, and the result is the principal-genus
	counterpart of a theorem of Katsura and Yamauchi for the non-principal
	genus \cite{KatsuraYamauchi}. Because the extreme eigenvalues are
	predicted by elliptic newforms alone, they can be tested far beyond the
	certified range: at all $104$ primes $17\le p\le601$ the extreme
	eigenvalue computed from the graph by Florit and Smith
	\cite{FloritSmith} equals the largest Saito--Kurokawa eigenvalue to
	the printed precision (Table~\ref{tab:fs}). The bulk of the spectrum is
	the residual square, predicted to consist of type-$\Va$ pairs, and the
	distribution of the spectrum appears to approach
	the Plancherel law of $\mathrm{PGSp}_4(\Q_2)$, the analogue of the
	Kesten--McKay law (Conjecture~\ref{conj:plancherel},
	Figure~\ref{fig:plancherel}). At every certified prime $p\ge29$ the irrational
	roots of the residual factor $\Va_p$ form a single Galois orbit with
	full symmetric Galois group (Proposition~\ref{prop:va-galois}).

	\subsection{Contributions}\label{ss:what-is-new}

	\begin{enumerate}[label=\textup{(\roman*)},leftmargin=2.5em]
		\item \emph{Signs.} The $R(\pi)$-signs of
		Conjecture~\ref{conj:refined-intro}, which separate the type-$\Va$
		pairs that neither dimensions nor the Hecke algebra away from $p$
		can separate, and the proof of its type-$\Va$ clause for every prime:
		each type-$\Va$ parameter occurs and contributes one eigenline of
		each sign.
		\item \emph{Certified computation.} The graph, its weights and
		$R(\pi)$ at $33$ primes, computed without input from the automorphic
		side, with the full spectrum certified, the first non-trivial sign
		test at $p=167,173$, and agreement with the independent quinary data
		of \cite{ALRTV}.
		\item \emph{Trace identity and sign bias.} The trace of $R(\pi)$ in
		elliptic and paramodular dimension data, obtained by comparing known
		closed trace formulas through a genus comparison pointed out to the
		author by T.~Ibukiyama, and $d(p)\ge0$, which
		sharpens Ibukiyama's bound $d(p)\ge-\dim S_4^-$
		\cite[Thm.~7.1]{Ibukiyama-dim} by his own method once the lifts are
		removed.
		\item \emph{Spectrum.} The sharp spectral interval and the source of
		its extreme eigenvalues, tested at $104$ primes against independent
		graph spectra; the Plancherel law; and, at the certified primes, the
		Galois structure of the residual factor $\Va_p$
		(Proposition~\ref{prop:va-galois}), a Maeda-type phenomenon.
	\end{enumerate}

	\subsection{Relation to prior work}\label{ss:relation-prior}
	
	The transfer and Arthur-type framework used here rests on published work on
	the compact inner form and its transfer to $\GSp_4$
	\cite{RoesnerWeissauer2021,vanHoften,DPRT,Schmidt2018}. The parahoric data
	come from \cite[Table~A.15]{RobertsSchmidt} on the split side and from
	Schmidt's table in \cite{Narita2017,NaritaOkazaki2013} on the inner
	form. Ibukiyama's dimension identity supplies
	the unconditional count underlying the degree bookkeeping
	\cite[Thm.~3.1]{Ibukiyama2018conj}; the multiplicity assertion for the
	general-type rows is proved in Section~\ref{sec:occurrence} via
	\cite[Thm.~11.4, Cor.~11.6]{RoesnerWeissauer2021} and Schmidt's table,
	while the Saito--Kurokawa and Yoshida rows remain conjectural; in the
	certified range the factorization they predict is verified.

	The occurrence pattern in Conjecture~\ref{conj:refined-intro} is
	Ibukiyama's at $(k,j)=(3,0)$. The Saito--Kurokawa row over all of
	$\snew_4$ is \cite[Conj.~3.2(1)]{Ibukiyama2018conj}, the Yoshida row over
	the pairs of opposite Atkin--Lehner signs is
	\cite[Conj.~3.2(2)]{Ibukiyama2018conj}, and the doubling for local type
	$\Va$ ($c_0=2$) and the correspondence with paramodular non-lifts
	($c_0=1$) are \cite[Conj.~3.3]{Ibukiyama2018conj}. Ibukiyama uses the
	Atkin--Lehner involutions only on elliptic and Siegel modular forms; no
	involution of the principal genus appears in \cite{Ibukiyama2018conj}.
	The new content of Conjecture~\ref{conj:refined-intro} is the
	assignment of $R(\pi)$-signs: clauses (iii) and (v) of
	Conjecture~\ref{conj:eigenvalue-sign}. Ibukiyama also asks whether the
	two forms attached to one Siegel eigenform of local type $\Va$ generate
	the same automorphic representation \cite[\S3.2]{Ibukiyama2018conj}. At
	trivial weight and $p\ge11$, Theorem~\ref{thm:occurrence-intro} answers
	the analogous question: the block of an occurring type-$\Va$ parameter
	is the sum of the $\Upr(p)$-fixed vectors of two distinct automorphic
	representations $\Pi^+\not\cong\Pi^-$, each of multiplicity one, with
	opposite $R(\pi)$-spectra, and since each contributes a single line,
	$R(\pi)$ separates them. In particular the doubling for local type
	$\Va$ in \cite[Conj.~3.3]{Ibukiyama2018conj} holds at $(k,j)=(3,0)$
	for every prime. Ibukiyama's
	later dimension formulas with
	involutions \cite{Ibukiyama-dim} concern the non-principal genus.
	
	On the graph side, spectral bounds for superspecial isogeny graphs are
	obtained from the automorphic spectrum in aggregate
	\cite{AikawaTanakaYamauchi,KatsuraYamauchi,CKTW}. The predicted
	factorization assigns individual eigenvalues to automorphic types
	instead: under Hypothesis~\ref{hyp:types} it shows which type can
	attain each end of the spectrum, and it turns the spectral bound into a
	statement that can be tested numerically eigenvalue by eigenvalue.
	
	\smallskip
	\noindent\emph{Organization.}
	Section~\ref{sec:statements} states the results in full.
	Section~\ref{sec:prelim} fixes the arithmetic and geometric settings and
	all normalizations, and Section~\ref{sec:transfer} the transfer framework.
	Section~\ref{sec:occurrence} proves the type-$\Va$ theorem and
	Section~\ref{sec:trace} the trace identity. Section~\ref{sec:certificates}
	describes the computation, the certificate format and the proof of the
	certified theorem. Section~\ref{sec:consequences} treats Ihara pairs,
	the Galois structure of the type-$\Va$ factor, the spectral bound with
	its test up to $p=601$, and the spectral distribution;
	Section~\ref{sec:signbias} proves the Fricke-sign bias. Appendix~\ref{app:matrices} prints the smallest
	matrices.

	\section{Statement of results}
	\label{sec:statements}

	For $k\in\{2,4\}$, let $\snew_k=\snew_k(\Gamma_0(p))$ denote the space of
	elliptic newforms of weight $k$ and prime level $p$, and write
	$S_k^{\pm}$ for its Atkin--Lehner eigenspaces. We use the following
	combinations throughout:
	\begin{align}
		\mathrm{cross}(p)
		&:=\dim S_2^-\dim S_4^+ + \dim S_2^+\dim S_4^-,
		\label{eq:cross-def}\\
		\delta(p)
		&:=\dim S_3(K(p))-\dim S_4^-,
		\label{eq:delta-def}\\
		d(p)
		&:=\delta(p)-2\dim S_3^+(K(p)),
		\label{eq:d-def}\\
		L_0(p)
		&:=1+\dim\snew_4+\dim S_2^-\dim S_4^+-\dim S_2^+\dim S_4^-.
		\label{eq:L0-def}
	\end{align}
	Here $\delta(p)$ counts the weight-$3$ paramodular non-lifts, while $d(p)$
	counts them with Fricke signs. The lift part $L_0(p)$ is the portion of the
	trace accounted for by Eisenstein, Saito--Kurokawa, and Yoshida classes.
	
	\subsection{The structural trace identity}
	\begin{theorem}[Structural trace identity; full form of \Cref{thm:structural-trace-intro}]\label{thm:structural-trace}
		For every prime $p\ge7$,
		\begin{equation}\label{eq:structural-trace}
			\Tr\!\left(R(\pi)\mid M_{0,0}(\Upr(p))\right)=L_0(p)+d(p).
		\end{equation}
	\end{theorem}

	The key intermediate identity is the genus comparison
	$\Tr R(\pi)-\Tr R_{\mathrm{npg}}(\pi)=\nu(p)(p-e_3)/12$ of
	Proposition~\ref{prop:genus-trace-comparison}, where
	$R_{\mathrm{npg}}(\pi)$ is the involution of the non-principal genus;
	the comparison is global (Remark~\ref{rem:global-orbital}).

	\subsection{The eigenvalue--sign refinement}
	
	The trace theorem fixes the signed total of the anticipated spectrum, but
	not its individual constituents. The next statement is conjectural in
	general.
	
	\begin{conjecture}[C1: eigenvalue--sign refinement; verified for $11\le p\le149$ and $p=167,173$]
		\label{conj:eigenvalue-sign} \label{conj:refined}
		For every prime $p\ge11$,
		\begin{equation}\label{eq:C1-factorization}
			\begin{split}
				\charp B_2(2)(x)
				={}&(x-15)
				\prod_{g\subset\snew_4}m_g(x-6)
				\prod_{\substack{(f,g):\\ \eps_p(f)\eps_p(g)=-1}}
				\Res_y\bigl(t_f(y),m_g(x-y)\bigr)\\
				&\qquad\qquad\cdot \Va_p(x)^2N_p(x).
			\end{split}
		\end{equation}
		Here $g$ and $f$ run over Galois orbits of newforms, $m_g$ is the
		characteristic polynomial of $T_2$ on the orbit $g$ and $t_f$ that of
		$2T_2$ on the orbit $f$ \textup{(}\S\ref{ss:split}\textup{)}. Moreover:
		\begin{enumerate}[label=\textup{(\roman*)},leftmargin=2.5em]
			\item the Saito--Kurokawa product runs over all weight-$4$ newform orbits;
			\item the Yoshida product runs over the opposite-sign pairs;
			\item the type-$\Va$ block is an exact square, and each such pair splits
			under $R(\pi)$ into one $+1$ and one $-1$ eigenline;
			\item $N_p$ is monic of degree $\delta(p)$, with roots the spin
			$T(2)$-eigenvalues of the weight-$3$ paramodular non-lifts;
			\item the $R(\pi)$-signs are $+1$ on the Eisenstein and
			Saito--Kurokawa components, $-\eps_p(f)$ on the Yoshida component attached
			to $(f,g)$, and $-\eps_p(F)$ on the general-type component attached to a
			non-lift $F$.
		\end{enumerate}
	\end{conjecture}
	
	The clauses name each block after the automorphic type predicted to
	produce it, but as statements \eqref{eq:C1-factorization} and
	\textup{(i)}--\textup{(v)} concern only $\charp B_2(2)$ and the
	$R(\pi)$-signed multiplicities of its roots: which polynomials divide
	it and how often, that the residual is a square whose roots carry
	equal multiplicities of the two signs, and the signed multiplicities
	on the other blocks. We call this the \emph{spectral content} of the
	conjecture. Its \emph{automorphic content}, that the eigenvectors in
	each block come from representations of the named type with the
	stated multiplicities, is stronger: the joint spectrum of two
	operators does not determine the Hecke eigensystems away from $2$.

	The conjecture makes two logically distinct assertions: an occurrence and
	multiplicity assertion, namely that the residual factor is an exact square
	of the stated degree, and a sign assertion about $R(\pi)$ on each block.
	For the type-Va block both are proved for every prime in
	Section~\ref{sec:occurrence} (Theorem~\ref{thm:occurrence} and
	Corollary~\ref{cor:conj-iii}), up to the identification of the degree,
	which rests on the remaining rows of Table~\ref{tab:twocol}.
	Theorem~\ref{thm:certified-spectra} establishes the spectral content of
	the conjecture, the factorization \eqref{eq:C1-factorization} with
	clauses \textup{(i)}--\textup{(v)}, in the certified range; it does not
	establish the automorphic reading of the factors, which requires in
	addition Hypothesis~\ref{hyp:types} and the remaining principal-genus
	occurrence and multiplicity statements for the Saito--Kurokawa and
	Yoshida rows; the sign assertion on $N_p$ is separate.
	
	Unlike the Saito--Kurokawa and Yoshida blocks, $\Va_p$ is not given by
	elliptic data: it is the residual left by the bookkeeping identity, and
	its roots are eigenvalues of $B_2(2)$ that no weight-$2$ or weight-$4$
	newform predicts.
	
	\subsection{The certified finite verification}
	
	\begin{theorem}[Certified spectra]
		\label{thm:certified-spectra}
		\begin{enumerate}[label=\textup{(\alph*)},leftmargin=2.3em]
			\item For every prime $11\le p\le149$, the archived exact-arithmetic
			data and the accompanying verifier satisfy the certificate hypotheses
			of Proposition~\ref{prop:reduction}. Consequently clauses
			\textup{(i)}--\textup{(iii)} and \textup{(v)} of
			Conjecture~\ref{conj:eigenvalue-sign} hold at each of these
			thirty-one primes, in the spectral sense explained after
			Conjecture~\ref{conj:eigenvalue-sign}, with a residual factor $N_p$
			of degree $\delta(p)$.
			Moreover $N_p$ equals the characteristic polynomial of $T_{2,1}$ on
			the weight-$(3,0)$ paramodular non-lifts of level $p$ recorded in
			\cite{ALRTV}, so clause \textup{(iv)} holds as well.
			\item For $p=167$ and $p=173$, the hypotheses of
			Proposition~\ref{prop:reduction} hold with $N_p$ the characteristic
			polynomial of $T_{2,1}$ on the non-lifts of level $p$ in \cite{ALRTV}
			and with the Fricke signs recorded there. Consequently
			\eqref{eq:C1-factorization} and clauses
			\textup{(i)}--\textup{(v)} of Conjecture~\ref{conj:eigenvalue-sign}
			hold at these two primes, in the same sense. In
			particular the factor $x+8$ of the unique Fricke-plus non-lift occurs
			in $\charp B_2(2)$ with multiplicity one and $R(\pi)$-sign $-1$.
		\end{enumerate}
	\end{theorem}

	The certificates check the weighted graph relations for the Richelot
	matrix, the involution, exact division by the Eisenstein,
	Saito--Kurokawa, and Yoshida blocks, the exact square in the residual
	factor, and the signed isotypic multiplicities determined by projector
	traces. These checks determine $N_p$ and its degree, but not the
	identification of its roots with paramodular eigenvalues; for clause
	\textup{(iv)} the proof therefore uses, as an additional input, the
	comparison with the quinary database of \cite{ALRTV} in
	\S\ref{ss:external}. In part (b) the database enters from the start: its
	non-lift data are the prediction that the certificate tests, and the
	sign clause (v) is tested in a case where it is not forced, because for
	$p<167$ every non-lift is Fricke-minus and the predicted sign on $N_p$ is
	always $+1$. Like Conjecture~\ref{conj:eigenvalue-sign} itself, the
	theorem concerns the joint spectrum of $B_2(2)$ and $R(\pi)$: it does
	not identify individual eigenvectors with automorphic representations,
	and it does not determine Hecke eigenvalues away from $2$. The statement
	is finite and computer-assisted; it does not
	prove Conjecture~\ref{conj:eigenvalue-sign} beyond the certified range,
	by which we mean the primes $11\le p\le149$ together with $167$ and
	$173$.
	
	\subsection{Transfer and eigenvalue framework}
	
	The following theorem explains the eigenvalue laws used in the
	conjectural factorization.

	\begin{theorem}[Transfer and eigenvalue laws]\label{thm:transfer-framework}
		Let $p\ge11$.
		\begin{enumerate}[label=\textup{(\alph*)},leftmargin=2.3em]
			\item The spectrum of $B_2(2)$ on $M_{0,0}(\Upr(p))$ is real, $15$ is
			a simple eigenvalue of $R(\pi)$-sign $+1$, and every other eigenvalue
			lies in $(-15,15)$.
			\item Let $\Pi$ be an automorphic representation of $\GU_2(B)$
			contributing to $M_{0,0}(\Upr(p))$. If $\Pi$ is the trivial
			representation, of Saito--Kurokawa type attached to
			$g\subset\snew_4$, of Yoshida type attached to $(f,g)$, or of general
			type with transfer $F$, then in the normalization
			$Q_2(X)=1-\lambda_2X+\cdots+2^6X^4$ its $B_2(2)$-eigenvalue is
			respectively
			\[
			15,\qquad a_2(g)+6,\qquad a_2(g)+2a_2(f),\qquad a_2(F).
			\]
			In the general-type case the transfer $F$ exists, is cuspidal of
			weight $(3,0)$, and has the same unramified Hecke data at every
			$q\nmid p$.
		\end{enumerate}
	\end{theorem}

	\noindent Theorem~\ref{thm:transfer-framework} is proved in
	\S\ref{sec:prelim}--\S\ref{sec:transfer}: the spectral clauses are
	Proposition~\ref{prop:frame} together with Lemma~\ref{lem:eis}, and the
	eigenvalue formulas are Proposition~\ref{prop:arthur}. That no other
	types contribute to $M_{0,0}(\Upr(p))$ is not part of the theorem; it is
	Hypothesis~\ref{hyp:types} of \S\ref{ss:automorphic-transfer}, which
	underlies the automorphic reading of
	Conjecture~\ref{conj:eigenvalue-sign} but not its spectral content
	\eqref{eq:C1-factorization}, \textup{(i)}--\textup{(v)}, and is not used in the
	proofs of Theorems~\ref{thm:structural-trace},
	\ref{thm:certified-spectra}, \ref{thm:occurrence} and~\ref{thm:bias}.
	
	The associated block-degree bookkeeping is
	\begin{equation}\label{eq:dimension-bookkeeping}
		h_2(p)=1+\dim\snew_4+\mathrm{cross}(p)+2\deg\Va_p+\delta(p).
	\end{equation}
	Under the change of variables in \S\ref{sec:transfer}, this is equivalent
	to Ibukiyama's dimension theorem \cite[Thm.~3.1]{Ibukiyama2018conj}. It is
	therefore a reformulation of known dimension information, not a new
	class-number theorem.
	
	\subsection{Logical status}
	Table~\ref{tab:logical-status} summarizes the logical status of the
	principal assertions.
	
	\begin{table}[ht]
		\centering
		\begin{tabular}{@{}p{0.34\linewidth}p{0.17\linewidth}p{0.35\linewidth}@{}}
			\toprule
			Statement & Range & Status \\
			\midrule
			Theorem~\ref{thm:structural-trace} & $p\ge7$ & Unconditional theorem \\
			Equation~\eqref{eq:dimension-bookkeeping} & $p\ge11$ & Reformulation of a known dimension theorem \\
			Theorem~\ref{thm:transfer-framework} & $p\ge11$ & Derived from published transfer results; exhaustiveness of the four types is Hypothesis~\ref{hyp:types} \\
			Theorem~\ref{thm:occurrence} & $p\ge11$ & Unconditional theorem, via \cite[Thm.~11.4]{RoesnerWeissauer2021} and \cite[Appendix]{Narita2017} \\
			Conjecture~\ref{conj:eigenvalue-sign} & $p\ge11$ & General-type multiplicities proved (Theorem~\ref{thm:occurrence}, Remark~\ref{rem:other-rows}); Saito--Kurokawa and Yoshida rows and the sign on $N_p$ conjectural \\
			Theorem~\ref{thm:certified-spectra} & $11\le p\le149$; $p=167,173$ & Finite computer-assisted theorem with archived certificates; clause (iv), and for $p=167,173$ the predicted Fricke signs, use \cite{ALRTV} \\
			Theorem~\ref{thm:bias} & every prime & Unconditional theorem, via the exact formula for $d(p)$ \\
			Theorem~\ref{thm:spectral} & $p\ge11$ & Theorem under Hypothesis~\ref{hyp:types} \\
			Proposition~\ref{prop:lambda2-limit} & $p\to\infty$ & Under Hypothesis~\ref{hyp:types} and clauses (i), (ii) of Conjecture~\ref{conj:eigenvalue-sign} \\
			Table~\ref{tab:fs}; Conjecture~\ref{conj:plancherel} & $17\le p\le601$; $p\to\infty$ & Comparison with independent data; conjecture with numerical support \\
			Corollaries~\ref{cor:dp-one}, \ref{cor:spectral-certified}; Proposition~\ref{prop:va-galois} & certified primes & Computer-assisted, exact and ball arithmetic \\
			\bottomrule
		\end{tabular}
		\caption{Logical status of the principal assertions.}
		\label{tab:logical-status}
	\end{table}
	
	Theorem~\ref{thm:structural-trace} confirms
	Conjecture~\ref{conj:eigenvalue-sign} at the level of signed totals
	(Remark~\ref{rem:trace-consistency}) but does not determine the
	individual multiplicities. Apart from clause (iv) of
	Theorem~\ref{thm:certified-spectra}(a) and from part (b), where the
	database supplies the prediction, comparisons with paramodular databases
	are corroboration and not inputs to the proofs.
	
	\section{Richelot--Brandt setting and normalizations}
	\label{sec:prelim}
	
	This section fixes the arithmetic, geometric, and automorphic settings and
	all conventions used later. Unless stated otherwise, vector spaces are over
	$\C$ and $p$ is a prime.
	
	\subsection{The principal genus and its operators}
	\label{ss:principal-genus}
	
	Let $B=B_{p,\infty}$ be the definite quaternion algebra over $\Q$ ramified
	exactly at $p$ and $\infty$, with canonical involution $x\mapsto\bar x$,
	reduced trace $\mathrm{tr}\,x=x+\bar x$, and reduced norm
	$\mathrm{n}\,x=x\bar x$. Fix a maximal order $\Onorm\subset B$. By
	Deuring's correspondence, the supersingular elliptic curves over
	$\overline{\F}_p$ correspond to the left $\Onorm$-ideal classes, of number
	$h=h(p)$, and their masses satisfy Eichler's relation
	\begin{equation}\label{eq:eichler-mass}
		\sum_{E}\frac{1}{\#\Aut E}=\frac{p-1}{24}.
	\end{equation}
	
	On $B^2$ the standard positive definite quaternion Hermitian form \cite{Shimura1963}
	$H(x,y)=\bar x_1y_1+\bar x_2y_2$ has unitary similitude group
	\[
	G=\GU_2(B),
	\]
	an inner form of $\GSp_4$ over $\Q$ \cite{Ibukiyama2007compact}. The left
	$\Onorm$-lattices in $B^2$ fall into genera and classes. The principal
	genus of the standard lattice $\Onorm^2$ has class number $h_2=h_2(p)$,
	given in closed form by \cite{HashimotoIbukiyama}. Locally at $p$, the two
	genera of maximal type correspond to two conjugacy classes of maximal open
	compact subgroups of $G(\Q_p)$, of different types and not conjugate; we
	denote them $U_1$ (principal) and $U_2$ (non-principal).
	
	Let $L_1,\ldots,L_{h_2}$ be representatives of the principal-genus classes.
	We write
	\[
	M_{0,0}(\Upr(p))=\{f:\{L_1,\ldots,L_{h_2}\}\to\C\}
	\]
	with the mass inner product
	\begin{equation}\label{eq:mass-inner}
		\langle f,g\rangle
		=\sum_{i=1}^{h_2}\frac{f(L_i)\,\overline{g(L_i)}}{e_i},
		\qquad e_i=\#\Aut(L_i).
	\end{equation}
	
	The Brandt matrices $B_n(2)$ of \cite{Hashimoto1980} act on this space and
	realize the Hecke operators $T(n)$; we abbreviate $B_2(2)$ for the
	degree-$2$ operator. The unique maximal two-sided ideal $\pi$ of $\Onorm$
	of reduced norm $p$ satisfies $\pi^2=p\,\Onorm$, so the induced operator
	$R(\pi)$ is an involution,
	\[
	R(\pi)^2=\mathrm{id},
	\]
	commuting with every $T(n)$. It is the Atkin--Lehner involution of the
	principal genus. When the non-principal genus is used, the corresponding
	operator is written $R_{\mathrm{npg}}(\pi)$.
	
	\subsection{The Richelot graph}
	\label{ss:richelot-graph}
	
	A principally polarized abelian surface over $\overline{\F}_p$ is either
	the Jacobian of a smooth genus-$2$ curve with its theta polarization, or a
	polarized product of two elliptic curves. It is superspecial if, after
	forgetting the polarization, it is isomorphic to a product of supersingular
	elliptic curves. By \cite[\S1]{IKO1986}, the superspecial principally
	polarized abelian surfaces over $\overline{\F}_p$, with their
	automorphisms, are exactly the classes of the principal genus; hence the
	vertex set of the graph below is identified with the index set of
	\eqref{eq:mass-inner}.
	
	A Richelot $(2,2)$-isogeny is an isogeny of principally polarized abelian
	surfaces whose kernel is a maximal Weil-isotropic subgroup of the
	$2$-torsion. Each source has exactly
	\begin{equation}\label{eq:fifteen}
		15=(2^2+1)(2+1)=\frac{6!}{2^3\,3!}
	\end{equation}
	such kernels: the number of maximal isotropic subspaces of a
	$4$-dimensional symplectic $\F_2$-space, equivalently the number of
	partitions of six Weierstrass points into three unordered pairs. These
	kernels realize $B_2(2)$ \cite{KatsuraTakashima}.
	
	Concretely, for a curve $C:y^2=f(x)$ with $f=c\,q_1q_2q_3$ a factorization
	into three quadratics over a splitting field, the codomain is controlled by
	the Richelot determinant
	\[
	\Delta_R=\det(q_1,q_2,q_3),
	\]
	the determinant of the coefficient matrix. If $\Delta_R\ne0$, the target is
	the Jacobian of
	\[
	\Delta_R\,y^2=[q_2,q_3]\,[q_3,q_1]\,[q_1,q_2],
	\qquad [q_i,q_j]=q_i'q_j-q_iq_j';
	\]
	if $\Delta_R=0$, the target is a product of elliptic curves
	\cite[\S8]{JordanZaytman}. The fifteen unordered pair-partitions of the
	roots of $f$ enumerate the kernels.
	
	\subsection{The dictionary and the Mestre symmetry}
	\label{ss:dictionary}
	
	We use the out-edge convention throughout:
	\begin{equation}\label{eq:out-edge}
		M_{ij}=\#\{\text{kernels at vertex $i$ with target $j$}\},
	\end{equation}
	so every row of $M$ sums to $15$. With $e_i=\#\Aut(A_i,\lambda_i)$ the
	vertex weights, the Mestre symmetry reads
	\begin{equation}\label{eq:mestre}
		e_jM_{ij}=e_iM_{ji}\qquad\text{for all }i,j,
	\end{equation}
	which is \cite[Thm.~18(c)]{JordanZaytman}.
	
	\begin{lemma}\label{lem:mestre}
		Let $M$ be a real $h\times h$ matrix, let $e_1,\ldots,e_h>0$ satisfy
		\eqref{eq:mestre}, and set $D=\diag(e_i)$. Then
		\begin{enumerate}[label=\textup{(\alph*)},leftmargin=2.2em]
			\item $S:=D^{-1/2}MD^{1/2}$ is symmetric;
			\item $M^{T}=D^{-1}MD$;
			\item $\charp(M)=\charp(S)$, so $M$ is diagonalizable with real spectrum;
			\item $M$ is self-adjoint for
			$\langle u,v\rangle=\sum_i u_iv_i/e_i$.
		\end{enumerate}
		Conversely, each of \textup{(a)} and \textup{(d)} is equivalent to
		\eqref{eq:mestre}. Moreover, the symmetrized support
		$\{(i,j):S_{ij}>0\}$ is exactly the undirected edge set of $G_p$:
		$S_{ij}>0\iff M_{ij}>0\iff M_{ji}>0$.
	\end{lemma}
	
	\begin{proof}
		$S_{ij}=e_i^{-1/2}M_{ij}e_j^{1/2}$, so $S_{ij}=S_{ji}\iff
		e_jM_{ij}=e_iM_{ji}$, giving (a) and its converse. For (b),
		$(D^{-1}MD)_{ij}=e_i^{-1}M_{ij}e_j=M_{ji}$. For (c), $S$ is conjugate to
		$M$ and symmetric, hence real-diagonalizable. For (d), for real $M$,
		\[
		\langle Mu,v\rangle-\langle u,Mv\rangle
		=\sum_{i,j}\bigl(M_{ij}/e_i-M_{ji}/e_j\bigr)u_jv_i,
		\]
		which vanishes identically iff \eqref{eq:mestre} holds. For the support
		claim, $S_{ij}=e_i^{-1/2}e_j^{1/2}M_{ij}$ with $e_i>0$, so
		$S_{ij}>0\iff M_{ij}>0$; and $M_{ij}>0\iff M_{ji}>0$ by \eqref{eq:mestre}.
	\end{proof}
	
	The last clause of Lemma~\ref{lem:mestre} is what makes the
	non-bipartiteness argument below valid: the sign pattern of a hypothetical
	$(-15)$-eigenvector is a $2$-colouring of the undirected graph whose edges
	are precisely the positive entries of $S$.
	
	The weights and the decomposable locus are pinned by the following lemma.
	
	\begin{lemma}\label{lem:vertices}
		Let $w_a=\#\Aut(E_a)/2\in\{1,2,3\}$ for the supersingular curves $E_a$
		over $\overline{\F}_p$, $a=1,\ldots,h$.
		\begin{enumerate}[label=\textup{(\alph*)},leftmargin=2.2em]
			\item The product vertices of $G_p$ are the unordered pairs
			$\{E_a,E_b\}$; their number is $h(h+1)/2$.
			\item Their weights are $e_{\{a,b\}}=\#\Aut(E_a\times E_b)$, equal to
			$8w_a^2$ if $a=b$ and $4w_aw_b$ if $a\ne b$.
			\item $\sum_a 1/w_a=(p-1)/12$, equivalently the Eichler mass
			$\sum_a 1/\#\Aut(E_a)=(p-1)/24$ of \eqref{eq:eichler-mass}.
			\item $G_p$ is connected and the Perron eigenvalue $15$ of $B_2(2)$ is
			simple.
		\end{enumerate}
	\end{lemma}
	
	\begin{proof}
		(a), (b) A polarized product $E_a\times E_b$ is superspecial, and two
		products are isomorphic as principally polarized abelian surfaces if and
		only if the unordered pairs of $j$-invariants agree \cite[\S3]{IKO1986};
		the automorphism group of the polarized product is
		$(\Aut E_a\times\Aut E_b)\rtimes C_2$ when $a=b$, of order
		$2(2w_a)^2=8w_a^2$, and $\Aut E_a\times\Aut E_b$ when $a\ne b$, of order
		$(2w_a)(2w_b)=4w_aw_b$, giving the stated weights. (c) Since
		$w_a=\#\Aut E_a/2$, $\sum_a 1/w_a=2\sum_a 1/\#\Aut E_a=(p-1)/12$ by
		\eqref{eq:eichler-mass}. (d) Connectivity is the irreducibility of the
		Brandt graph \cite[Thm.~44]{JordanZaytman}; a connected non-negative
		symmetrizable matrix with constant row sum $15$ has $15$ as a simple
		Perron eigenvalue by Perron--Frobenius applied to the symmetrization $S$
		of Lemma~\ref{lem:mestre}.
	\end{proof}
	
	\subsection{The spectral frame}
	\label{ss:spectral-frame}
	
	\begin{proposition}\label{prop:frame}
		For every prime $p\ge11$ the multigraph $G_p$ is connected and not
		bipartite; consequently $x-15$ divides $\charp B_2(2)$ with multiplicity
		one, and every other eigenvalue lies in the open interval $(-15,15)$. In
		particular any integer eigenvalue $\lambda\ne15$ satisfies
		$|\lambda|\le14$.
	\end{proposition}
	
	\begin{proof}
		Connectedness and non-bipartiteness of $G_p$ are
		\cite[Thm.~44]{JordanZaytman}. Let $S$ be the symmetrization of
		Lemma~\ref{lem:mestre}: $S$ is symmetric, nonnegative, and irreducible by
		connectedness. Since $M\one=15\one$ and
		$S=D^{-1/2}MD^{1/2}$, the positive vector $D^{-1/2}\one$ is an
		eigenvector of $S$ with eigenvalue $15$. Perron--Frobenius therefore
		makes $15$ its simple spectral radius. If $Sv=-15v$ with $v\ne0$, then entrywise
		$S|v|\ge|Sv|=15|v|$. Let $w>0$ be the Perron eigenvector, $Sw=15w$.
		Pairing and using $S=S^{T}$,
		\[
		15\,\langle w,|v|\rangle=\langle Sw,|v|\rangle=\langle w,S|v|\rangle
		\ge15\,\langle w,|v|\rangle,
		\]
		so $\langle w,S|v|-15|v|\rangle=0$; as $w>0$ and $S|v|-15|v|\ge0$
		entrywise, this forces $S|v|=15|v|$. Equality in $|(Sv)_i|=(S|v|)_i$ on
		each row then forces $v_j$ to have constant sign along the support of row
		$i$, i.e.\ $v_i$ and $v_j$ have opposite signs whenever $S_{ij}>0$. By the
		last clause of Lemma~\ref{lem:mestre} the set $\{(i,j):S_{ij}>0\}$ is the
		undirected edge set of $G_p$, so $\operatorname{sgn}v$ is a proper
		$2$-colouring of $G_p$, contradicting non-bipartiteness. Hence $-15$ is
		not an eigenvalue, and since the spectrum is real every non-Perron
		eigenvalue lies in $(-15,15)$.
	\end{proof}
	
	\subsection{Elliptic-modular data on the split side}
	\label{ss:split}
	
	For $k\in\{2,4\}$ let $\snew_k$ denote the space of weight-$k$ level-$p$
	elliptic newforms, and for $f\in\snew_k$ let $\eps_p(f)\in\{\pm1\}$ be its
	Atkin--Lehner eigenvalue at $p$, with $S_k^{\pm}$ the corresponding
	eigenspaces. At exact prime level the $p$-th Hecke eigenvalue is determined
	by the sign,
	\begin{equation}\label{eq:ap-sign}
		a_p(f)=-\eps_p(f)\,p^{k/2-1},
	\end{equation}
	by Atkin--Lehner theory \cite[Thm.~3]{AtkinLehner}.
	
	These spaces enter the catalogue through the Yoshida lifts: a Galois orbit
	$g\subset\snew_4$ and a form $f\in\snew_2$ contribute a paramodular
	Yoshida packet whose spin $T(2)$-eigenvalue is $a_2(g)+2a_2(f)$.
	Accordingly, for a Galois orbit $g\subset\snew_4$ with coefficient field
	$K_g$ and an orbit $f\subset\snew_2$ with coefficient field $K_f$, we set
	\[
	m_g(x)=\mathrm{N}_{K_g/\Q}\bigl(x-a_2(g)\bigr),\qquad
	t_f(y)=\mathrm{N}_{K_f/\Q}\bigl(y-2a_2(f)\bigr),
	\]
	the characteristic polynomials of $T_2$ on the orbit $g$ and of $2T_2$ on
	the orbit $f$. The monic resultant $\Res_y\bigl(t_f(y),m_g(x-y)\bigr)$
	then lists the numbers $a_2(g^\sigma)+2a_2(f^\tau)$ over all pairs of
	conjugates, with multiplicity. The polynomial $m_g$ is a power of the
	minimal polynomial of $a_2(g)$, with exponent $[K_g:\Q(a_2(g))]$, and
	the exponent matters: the orbit $f$ of $\snew_2(\Gamma_0(113))$ with
	$K_f=\Q(\sqrt3)$ has $a_2(f)=1$, so $t_f(y)=(y-2)^2$, and the
	corresponding Yoshida factor of degree $16$ occurs twice in
	$\charp B_2(2)$ at $p=113$ (Table~\ref{tab:ext}). At every other prime
	of Table~\ref{tab:ext} the exponents are $1$.
	
	\begin{lemma}\label{lem:resultant}
		Let $t,m\in\Z[x]$ be monic of degrees $a,b$ with complex roots
		$\alpha_i,\beta_j$. With
		$\Res_y(A,B)=\mathrm{lc}(A)^{\deg B}\prod_{A(\alpha)=0}B(\alpha)$,
		\[
		\Res_y\bigl(t(y),\,m(x-y)\bigr)
		=\prod_{i,j}\bigl(x-(\alpha_i+\beta_j)\bigr)\in\Z[x]
		\]
		is monic of degree $ab$ with root set $\{\alpha_i+\beta_j\}$, independent
		of the ordering convention up to the sign $(-1)^{ab}$ which the monic
		normalization removes. For monic $P,Q\in\Z[x]$, $\gcd_{\Q[x]}(P,Q)=1$ if
		and only if $P,Q$ share no complex root.
	\end{lemma}
	
	\begin{proof}
		Monicity of $t$ gives $\Res_y(t,m(x-\cdot))=\prod_i m(x-\alpha_i)$; each
		factor is monic of degree $b$ with roots $\alpha_i+\beta_j$. Coefficients
		are symmetric functions of the $\alpha_i$ with integer values, hence in
		$\Z$. The gcd statement is the nullstellensatz in one variable.
	\end{proof}
	
	\subsection{Consolidated notation}
	\label{ss:notation}
	
	The catalogue and the trace formula are stated in terms of the following
	weight-$3$ Siegel--paramodular dimensions and their combinations.
	
	\begin{table}[ht]
		\centering
		\renewcommand{\arraystretch}{1.25}
		\begin{tabular}{@{}l p{0.62\linewidth}@{}}
			\toprule
			Symbol & Meaning \\
			\midrule
			$h_2=h_2(p)$ & class number of the principal genus $=\#$ vertices of $G_p$\\
			$h=h(p)$ & supersingular elliptic class number ($=\#$ product factors)\\
			$T_1=T_1(p)$ & type number of the principal genus $=\#$ $\sigma$-orbits\\
			$S_3(\Gamma_0(p)),\,S_3(\Gamma_0'(p)),\,S_3(K(p))$ & weight-$3$ Siegel, Klingen, paramodular cusp spaces\\
			$S_3^{\pm}(K(p))$ & paramodular Atkin--Lehner (Fricke) eigenspaces\\
			$\dim S_k^{\pm}$ & AL-eigenspaces of $\snew_k$, $k\in\{2,4\}$\\
			$\mathrm{ssn}(p)$ & $\dim S_2^+\dim S_4^++\dim S_2^-\dim S_4^-$ (same-sign Yoshida)\\
			$\mathrm{cross}(p)$ & $\dim S_2^-\dim S_4^++\dim S_2^+\dim S_4^-$ (opposite-sign Yoshida)\\
			$\Delta_{\mathrm{III}}(p)$ & $\dim S_3(\Gamma_0(p))-\mathrm{ssn}(p)$ (general-type Siegel meter)\\
			$\delta(p)$ & $\dim S_3(K(p))-\dim S_4^-=\deg N_p$ (number of paramodular non-lifts)\\
			$d(p)$ & $\delta(p)-2\dim S_3^+(K(p))$ (Fricke-signed non-lift defect)\\
			$c_0(\Pi_p)$ & $2a-c-2b$, parahoric-invariant index in Roberts--Schmidt coordinates \cite{RobertsSchmidt}\\
			$\Va_p$ & monic, $\Va_p^2$ the type-$\Va$ block; $\deg\Va_p=\tfrac12\bigl(h_2-1-\dim\snew_4-\mathrm{cross}-\delta\bigr)$\\
			$R(\pi)$ & Fricke involution on the \emph{principal} genus $M_{0,0}(\Upr(p))$; the subscript is omitted throughout\\
			$R_{\mathrm{npg}}(\pi)$ & the same involution on the non-principal genus\\
			\bottomrule
		\end{tabular}
		\caption{Consolidated notation for the weight-$3$ dimension data. The local
			index $c_0$ labels the parahoric-restriction type of $\Pi_p$; its values
			for the relevant types are read from \cite[Table~A.15]{RobertsSchmidt} (as
			in \cite[\S3]{Ibukiyama2018conj}).}
		\label{tab:notation}
	\end{table}
	
	We record the two splittings used repeatedly, both of which follow from the
	definitions in the table:
	\begin{align}
		\dim S_3(\Gamma_0(p)) &= \mathrm{ssn}(p)+\Delta_{\mathrm{III}}(p),\\
		\dim S_3(K(p)) &= \dim S_4^-+\delta(p).
	\end{align}
	The first separates the Siegel cusp space into its same-sign Yoshida part
	and a general-type remainder; the second separates the paramodular cusp
	space into its Saito--Kurokawa part (in bijection with $S_4^-$) and the
	non-lifts counted by $\delta(p)$. Every dimension in
	Table~\ref{tab:notation} is a closed form in $p$ by
	\cite{Ibukiyama2007dim,Ibukiyama-dim}. The definition of $\deg\Va_p$ in
	the table is the bookkeeping identity \eqref{eq:dimension-bookkeeping}
	read backwards, and $\delta-d=2\dim S_3^+(K(p))$ is even; the
	non-negativity of $d(p)$ is Theorem~\ref{thm:bias} and is assumed
	nowhere.

	\section{Transfer framework and eigenvalue laws}
	\label{sec:transfer}
	
	This section explains the automorphic framework behind the eigenvalue laws
	of Theorem~\ref{thm:transfer-framework} and the factorization in
	Conjecture~\ref{conj:eigenvalue-sign}. Two questions must be
	distinguished. The transfer and the eigenvalue formulas apply to every
	principal-genus eigensystem that occurs. By contrast, the question of
	which parameters occur in the principal genus, and with what
	multiplicity, is conjectural in general.
	
	\subsection{Block degrees and their non-negativity}
	\label{ss:block-degrees}
	
	The factorization predicted in
	Conjecture~\ref{conj:eigenvalue-sign} has two non-lift degrees. The degree
	$\delta(p)=\deg N_p$ counts paramodular non-lifts; the degree
	$\deg\Va_p$ counts type-$\Va$ pairs. The signed quantity $d(p)$ plays a
	different role: it measures the Fricke imbalance in the general-type
	block. In particular, $d(p)$ is not used in the factorization itself.
	
	\begin{lemma}\label{lem:positivity}
		For every prime $p\ge11$,
		\begin{enumerate}[label=\textup{(\roman*)},leftmargin=2.3em]
			\item $\delta(p)\ge0$;
			\item $\deg\Va_p\ge0$ for $11\le p\le149$ and $p=167,173$, and for every $p\ge11$
			under Conjecture~\ref{conj:eigenvalue-sign}.
		\end{enumerate}
		Statement \textup{(i)} is unconditional; neither statement uses the
		sign of $d(p)$.
	\end{lemma}
	
	\begin{proof}
		For \textup{(i)}, the Saito--Kurokawa part of $S_3(K(p))$ is the space
		of Gritsenko lifts, of dimension $\dim S_4^{-}(\Gamma_0(p))$
		\cite{Gritsenko1995,IPY2013}; it is a subspace of $S_3(K(p))$, so
		$\delta(p)=\dim S_3(K(p))-\dim S_4^{-}\ge0$. In the certified range
		$\delta(p)=0$ for $p\le59$ (Theorem~\ref{thm:certified-spectra}).
		
		For \textup{(ii)}, the certified factorizations of
		Theorem~\ref{thm:certified-spectra} give $\deg\Va_p\ge0$ at every
		prime in the certified range, and under
		Conjecture~\ref{conj:eigenvalue-sign} the degree $\deg\Va_p$ counts
		the type-$\Va$ parameters occurring in the principal genus, so the
		inequality holds for every~$p$.
	\end{proof}
	
	\begin{remark}
		The bookkeeping definition of $\deg\Va_p$ makes $2\deg\Va_p$ a closed
		form in $p$ (Proposition~\ref{prop:VI}); its non-negativity for
		every prime is a necessary consequence of the principal-genus
		multiplicity prediction, not equivalent to it, and is established
		here only in the certified range.
	\end{remark}

	\subsection{Automorphic decomposition and transfer}
	\label{ss:automorphic-transfer}
	
	Let $G=\GU_2(B)$. The finite-dimensional space
	$M_{0,0}(\Upr(p))$ decomposes into automorphic isotypic components
	\[
	M_{0,0}(\Upr(p))=\bigoplus_{\Pi}V_\Pi,
	\]
	where $\Pi$ runs over automorphic representations of $G$ with trivial
	archimedean component and unramified components at every $q\nmid p$.
	The dimension of $V_\Pi$ is
	\[
	\dim V_\Pi=m(\Pi)\dim\Pi_p^{U_{1,p}},
	\]
	where $m(\Pi)$ is the automorphic multiplicity. The commuting operators
	$B_2(2)$ and $R(\pi)$ preserve each $V_\Pi$.
	
	Following \cite[\S1]{RoesnerWeissauer2021}, an automorphic representation
	$\Pi$ of $G(\mathbb{A})$ is \emph{CAP} if it is a weak inner lifting of an
	Eisenstein representation of $\GSp_4(\mathbb{A})$, \emph{weakly
	endoscopic} if it is a weak endoscopic lifting, and of \emph{general type}
	otherwise. In the notation of \cite[\S1]{Schmidt2018}, the split-side
	Arthur parameters attached to these representations are of type (G) in
	the general case and of type (Y) in the endoscopic case, and Arthur's
	classification predicts that a CAP representation corresponds to a
	parameter of one of the types (P), (Q), (B) or (F). The trivial
	representation of $G(\mathbb{A})$ is of type (F) and spans the constant
	functions. Saito--Kurokawa packets on inner forms
	are described by Gan \cite{GanSK2008}, weak endoscopic lifts on inner forms
	by Chan and Gan \cite[\S3]{ChanGan2015} (see
	\cite[Prop.~12.1]{RoesnerWeissauer2021}), and the general type by
	\cite[Thm.~11.4]{RoesnerWeissauer2021}.

	For the \emph{non-principal} genus, that is, for the stabilizer $K_p^-$
	of \cite[\S3.1.2]{DPRT}, CAP representations of Klingen and Borel type
	are excluded in the proof of \cite[Thm.~9.6]{DPRT}, by
	\cite[Prop.~5.1]{Schmidt2020}, and the Eisenstein-type subspace of
	$M_{0,0}(K(p))$ is spanned by the constants \cite[Prop.~9.8(3)]{DPRT}.
	These arguments use the paramodular correspondence of the non-principal
	genus and do not transfer to $\Upr(p)$, and we know of no published
	reference for the principal genus. An Arthur multiplicity formula for
	$G$, which would settle the question, is not available: $G$ is not
	quasi-split at $p$, so the results of Ta\"ibi \cite{Taibi2019} do not
	apply. Conditionally on the twisted weighted fundamental lemma, every
	discrete automorphic representation of the derived group of the
	compact inner form has a formal Arthur parameter
	\cite[Thm.~5.34]{CKTW}, which restricts the possible types; we have not
	checked whether this implies the statement below. We therefore record
	the needed statement as a hypothesis.

	\begin{hypothesis}\label{hyp:types}
		Let $p\ge11$. Every automorphic representation of $G(\mathbb{A})$
		contributing to $M_{0,0}(\Upr(p))$ is the trivial representation, of
		Saito--Kurokawa type with trivial quadratic character (type \textup{(P)}
		with $\sigma=1$ in \cite{Schmidt2018}), of Yoshida type, or of general
		type.
	\end{hypothesis}

	Hypothesis~\ref{hyp:types} is implicit in Ibukiyama's conjectures
	\cite[\S3.2]{Ibukiyama2018conj} and in the automorphic interpretation
	of Conjecture~\ref{conj:eigenvalue-sign}, whose factorization
	\eqref{eq:C1-factorization} names blocks of these four kinds only; the
	factorization by itself does not exclude other automorphic types.
	Apart from Theorem~\ref{thm:spectral} and
	Proposition~\ref{prop:lambda2-limit}, which assume it, no result of this
	paper depends on it; it enters otherwise only the interpretation of the
	residual factor $\Va_p(x)^2$ and Remark~\ref{rem:trace-consistency}.
	The certificates of Theorem~\ref{thm:certified-spectra} do not test it,
	because they identify the type-$\Va$ block only as a residual square.

	\begin{proposition}	\label{prop:arthur}
		Let $p\ge11$ and $G=\GU_2(B)$.
		\begin{enumerate}[label=\textup{(\alph*)},leftmargin=2.3em]
			\item Every simultaneous eigensystem of $\{B_2(2),R(\pi)\}$ on
			$M_{0,0}(\Upr(p))$ belongs to an automorphic isotypic component $V_\Pi$
			as above.
			\item If $\Pi$ is of general type, it has a weak lifting to a cuspidal
			representation $\Pi'$ of $\GSp_4(\mathbb{A})$ of type \textup{(G)},
			cohomological of weight $(3,0)$, whose Satake parameters agree with
			those of $\Pi$ at every prime $q\nmid p$ under the unramified
			identification $G(\Q_q)\simeq\GSp_4(\Q_q)$
			\cite[Thm.~11.4]{RoesnerWeissauer2021}. If $\Pi$ is of Yoshida or
			Saito--Kurokawa type, its Satake parameters at $q\nmid p$ are those of
			the corresponding parameter of type \textup{(Y)} or
			\textup{(P)}.
			\item In the spin normalization
			\[
			Q_2(X)=1-\lambda_2X+\cdots+2^6X^4,
			\]
			the eigenvalue of $B_2(2)$ on $V_\Pi$ is the coefficient $\lambda_2$,
			namely
			\[
			\begin{array}{c|c|c}
				\text{type} & \text{spin roots at }2 & \lambda_2\\
				\hline
				\mathrm{Eis} & \{1,2,4,8\} & 15\\
				\mathrm{P} & \{\alpha_g,\beta_g,2,4\} & a_2(g)+6\\
				\mathrm{Y} & \{2\alpha_f,2\beta_f,\alpha_g,\beta_g\}
				& a_2(g)+2a_2(f)\\
				\mathrm{G} & \text{spin roots of }F\text{ at }2 & a_2(F).
			\end{array}
			\]
		\end{enumerate}
	\end{proposition}
	
	\begin{proof}
		Part \textup{(a)} is the standard automorphic decomposition of the
		finite double-coset space. Part \textup{(b)} is
		\cite[Thm.~11.4]{RoesnerWeissauer2021} in the general-type case,
		applicable by Lemma~\ref{lem:rw-applicable}, and the definition of the
		weak endoscopic and Saito--Kurokawa liftings
		\cite[\S3]{ChanGan2015}, \cite{GanSK2008} in the other two cases.
		Part \textup{(c)} follows by taking the sum of the four spin roots. In
		the unitary normalization the spin $L$-function of a parameter
		$\mu_1\boxplus\mu_2$ of type \textup{(Y)} is $L(s,\mu_1)L(s,\mu_2)$,
		and that of a parameter of type \textup{(P)} attached to $\mu$ is
		$L(s,\mu)\zeta(s+\tfrac12)\zeta(s-\tfrac12)$
		\cite[Table~1]{Schmidt2018}. The normalization of $Q_2$ multiplies
		the unitary roots at $2$ by $2^{3/2}$. With $\mu_1,\mu_2$ attached to
		$g\in\snew_4$ and $f\in\snew_2$, whose unitary roots at $2$ are
		$2^{-3/2}\alpha_g,2^{-3/2}\beta_g$ and $2^{-1/2}\alpha_f,2^{-1/2}\beta_f$,
		this gives the roots $\alpha_g,\beta_g,2\alpha_f,2\beta_f$ for type
		\textup{(Y)}, and $\alpha_g,\beta_g,2,4$ for type \textup{(P)}; their
		product is $2^6$, as the normalization requires. Thus the corresponding
		sums are $a_2(g)+2a_2(f)$ and $a_2(g)+6$.
	\end{proof}
	
	\begin{lemma}
		\label{lem:eis}
		The action of $R(\pi)$ on the principal-genus classes is a permutation.
		It fixes the constant functions. Consequently, the Eisenstein eigenvalue
		$15$ has $R(\pi)$-sign $+1$.
	\end{lemma}
	
	\begin{proof}
		The matrix of $R(\pi)$ has exactly one entry equal to $1$ in each row
		\cite[\S3.2]{IbukiyamaKatsura}. Hence it fixes the constant functions.
		Since $B_2(2)$ has row sum $15$, the constant line is the Eisenstein line.
	\end{proof}
	
	\subsection{The local index and the multiplicity table}
	\label{ss:local-index}
	
	The rows of the multiplicity table are labelled by an index computed on
	the split side, and only for the general-type rows. For $\Pi$ of
	general type, let $\Pi'_p$ be the generic member of the local
	$L$-packet of $\GSp_4(\Q_p)$ attached to the transfer of $\Pi$ at $p$
	(the members of a packet need not have the same parahoric fixed
	vectors; for the packet $\{\Va,\Va^*\}$ the supercuspidal member has
	none; some non-generic singleton packets of the other rows have no
	generic member at all), and let $a,c,b$ be the dimensions of its fixed-vector spaces
	under the Klingen, Siegel, and paramodular parahorics, respectively.
	Following
	\cite[Table~A.15]{RobertsSchmidt} and the parahoric-restriction tables
	of \cite{Roesner2016}, define
	\begin{equation}\label{eq:c0}
		c_0(\Pi'_p)=2a-c-2b.
	\end{equation}
	For the relevant Iwahori-spherical types, one has
	$c_0\in\{-1,0,1,2\}$; types with a paramodular or Siegel new vector
	have $c_0\le1$, and $c_0=2$ occurs exactly for type Va. The index labels
	the rows of Table~\ref{tab:twocol}; it is not a formula for the
	principal-column entries, which concern fixed vectors of the
	representations of the inner form $G(\Q_p)$
	(Remarks~\ref{rem:local-constant} and~\ref{rem:other-rows}).
	
	\begin{table}[ht]
		\centering
		\begin{tabular}{l|c|c}
			Arthur type of $\Pi$ &
			$\dim$ in $M_{0,0}(\Upr(p))$ &
			$\dim$ in $M_{0,0}(U_{\mathrm{npg}}(p))$\\
			\hline
			Eisenstein & $1$ & $1$\\
			Saito--Kurokawa lift of $g\subset\snew_4$ & $1$ &
			$\tfrac12\bigl(1-\eps_p(g)\bigr)$\\
			Yoshida, $\eps_p(f)\eps_p(g)=-1$ & $1$ & $0$\\
			Yoshida, $\eps_p(f)\eps_p(g)=+1$ & $0$ & $0$\\
			general type, local type Va at $p$ \textup{(}$c_0=2$\textup{)} & $2$ & $0$\\
			general type, local $c_0=1$ & $1$ & $1$\\
			types with $c_0\le0$ & $0$ & $0$
		\end{tabular}
		\caption{The two-column multiplicity table at weight $(3,0)$. The
			non-principal column is a theorem. The principal column is the content of
			Ibukiyama's conjectures. Its general-type rows are proved for every
			prime (Theorem~\ref{thm:occurrence}, Remarks~\ref{rem:local-constant}
			and~\ref{rem:other-rows}); for the Saito--Kurokawa and Yoshida rows,
			Theorem~\ref{thm:certified-spectra} verifies the
			factorization and signs they predict for $11\le p\le149$ and
			$p=167,173$, which is evidence for the column but not a determination
			of the automorphic multiplicities.}
		\label{tab:twocol}
	\end{table}
	
	The rows are indexed by the global Arthur type of $\Pi$, refined by the
	local type at $p$ in the general-type case. The label type~$\Va$ names the
	local component at $p$ and does not by itself determine the global
	parameter: the Yoshida contribution is counted separately through
	$\mathrm{cross}(p)$, so the type-$\Va$ row consists of general-type
	parameters. The entry $2$ in the principal column is a multiplicity in
	$M_{0,0}(\Upr(p))$, not a dimension of paramodular new vectors: type
	$\Va$ has no paramodular vector at level $p$, its first one occurring at
	level $p^{2}$ \cite[Table~5.2]{RobertsSchmidtTables}, which is the $0$ in
	the non-principal column.
	
	\begin{remark}[The local packet at $p$ and the sign in
		Conjecture \ref{conj:eigenvalue-sign}\textup{(iii)}]
		\label{rem:localpacket}
		The $\pm$ splitting asserted for a type-$\Va$ pair is a consequence of
		its multiplicity, not an independent assumption. Write $\xi$ for the
		unramified quadratic character of $\Q_p^\times$. The $L$-parameter of
		type $\Va$ is $\varphi=\varphi_1\oplus\varphi_2$ with
		$\varphi_1=\sigma\otimes\mathrm{sp}(2)$ and
		$\varphi_2=\varphi_1\otimes\xi$, where $\sigma$ is a character of
		$\Q_p^\times$ and $\mathrm{sp}(2)$ is the two-dimensional special
		parameter, that of the Steinberg representation
		\cite[Table~3]{Schmidt2018}, \cite[Table~A.7]{RobertsSchmidt}. Thus
		$\varphi_1$ and $\varphi_2$ are inequivalent irreducible
		two-dimensional representations of the Weil--Deligne group
		$W_{\Q_p}\times\mathrm{SL}_2(\C)$, that is, discrete parameters of
		$\mathrm{GL}_2$, with
		$\det\varphi_1=\det\varphi_2$; on the split group the packet is
		$\{\Va,\Va^{*}\}$ \cite[\S1.2]{Schmidt2018}. On the inner form
		$G(\Q_p)=\GU_2(B_p)$ the local Langlands correspondence of Gan and
		Tantono parametrizes the packet of $\varphi$ by
		$\widehat{B_\varphi}\smallsetminus\widehat{A_\varphi}$, where
		$A_\varphi$ and $B_\varphi$ are the component groups of the
		centralizers of the image of $\varphi$ in $\GSp_4(\C)$ and
		$\mathrm{Sp}_4(\C)$ \cite[Main Thm.~(ii)]{GanTantono}. A parameter of this shape, a sum of two
		inequivalent discrete parameters of $\mathrm{GL}_2$ with equal
		determinants, has the two-element packet
		\[
		\bigl\{\theta\bigl(\mathrm{JL}(\tau_1)\boxtimes\tau_2\bigr),\
		\theta\bigl(\mathrm{JL}(\tau_2)\boxtimes\tau_1\bigr)\bigr\},
		\]
		where $\tau_i$ is the essentially square-integrable representation
		of $\mathrm{GL}_2(\Q_p)$ with parameter $\varphi_i$, $\mathrm{JL}$ is
		the Jacquet--Langlands correspondence to $B_p^\times$, and $\theta$ is
		the similitude theta lift from $\mathrm{GSO}_{1,1}$
		\cite[\S7.2, Case~1]{GanTantono}; see the restatement in
		\cite[\S7.1, Case~I]{Choiy2017}, and \cite[Table~3]{RoesnerWeissauer2021},
		where this packet is $\{\Va_G,\chi_0\otimes\Va_G\}$. Twisting by
		$\chi\circ\lambda$, where $\lambda$ is the similitude character,
		corresponds to $\varphi\mapsto\varphi\otimes\chi$
		\cite[Main Thm.~(iv)]{GanTantono}, and on the members above it twists
		both arguments of $\theta$. Since $\tau_1\otimes\xi\cong\tau_2$ and
		$\tau_2\otimes\xi\cong\tau_1$, twisting by $\xi\circ\lambda$ exchanges
		the two members: if $\Pi_p^{\pm}$ denote them, then
		$\Pi_p^-\cong(\xi\circ\lambda)\otimes\Pi_p^+$. Since $\xi$ is
		unramified and $\lambda$ takes values in $\Z_p^\times$ on $\Upr(p)$,
		the twist identifies the $\Upr(p)$-fixed spaces of $\Pi_p^+$ and
		$\Pi_p^-$, while $\lambda(\pi)=p$ gives $\xi(\lambda(\pi))=-1$, so
		this identification negates $R(\pi)$. The two members therefore carry
		opposite $R(\pi)$-spectra, whatever the dimension of their fixed
		spaces. The occurrence of both members is proved in
		Section~\ref{sec:occurrence} (Theorem~\ref{thm:occurrence}), and the
		fixed spaces are lines by Schmidt's computation
		(Remark~\ref{rem:local-constant}), so the type-$\Va$ row of
		Table~\ref{tab:twocol} is a theorem.
	\end{remark}

	The principal column of Table~\ref{tab:twocol} is where the conjectural
	content lies. The existence of the Saito--Kurokawa and Yoshida liftings is
	known, as is the non-principal column. The general-type rows are settled
	in Section~\ref{sec:occurrence} (Theorem~\ref{thm:occurrence} and
	Remarks~\ref{rem:local-constant} and~\ref{rem:other-rows}); what remains
	open in general is the principal-column multiplicity statement for the
	Saito--Kurokawa and Yoshida rows.
	
	\subsection{Dimension bookkeeping}
	\label{ss:dimension-bookkeeping}
	
	The block degrees in Section~\ref{sec:statements} are equivalent to the
	known dimension identity of Ibukiyama. We record this equivalence to make
	clear that the degree count is not an additional theorem.
	
	\begin{proposition}
		\label{prop:VI}
		The identity
		\[
		h_2(p)=1+\dim\snew_4+\mathrm{cross}(p)
		+2\deg\Va_p+\delta(p)
		\]
		together with
		\[
		\deg\Va_p
		=\dim S_3(\Gamma_0'(p))-
		\dim S_3(K(p))-\tfrac12\Delta_{\mathrm{III}}(p)-\tfrac12\delta(p)
		\]
		is equivalent, as an identity of closed forms, to
		\begin{equation}\label{eq:ibu-dim}
			2\dim S_3(\Gamma_0'(p))-
			\dim S_3(\Gamma_0(p))-2\dim S_3(K(p))
			=h_2(p)-1-(\dim\snew_2+1)\dim\snew_4,
		\end{equation}
		which is \cite[Thm.~3.1, $j=0$]{Ibukiyama2018conj}.
	\end{proposition}
	
	\begin{proof}
		Substitute
		\[
		\dim S_3(\Gamma_0(p))=\mathrm{ssn}(p)+\Delta_{\mathrm{III}}(p),
		\qquad
		\dim S_3(K(p))=\dim S_4^-+\delta(p),
		\]
		and
		\[
		\dim\snew_2\dim\snew_4
		=\mathrm{ssn}(p)+\mathrm{cross}(p)
		\]
		into \eqref{eq:ibu-dim}. The same-sign term cancels, and rearrangement
		gives exactly the asserted block-degree formula.
	\end{proof}

	\section{The occurrence theorem for the type-Va block}
	\label{sec:occurrence}
	
	This section proves the occurrence half of
	Conjecture~\ref{conj:eigenvalue-sign}\textup{(iii)}. The input is the
	global lifting theorem of R\"osner and Weissauer for inner forms of
	$\GSp_4$ that are anisotropic at the archimedean place
	\cite[Thm.~11.4]{RoesnerWeissauer2021}. Applied to $G=\GU_2(B)$ at
	trivial weight, it shows that the weak packet of a general-type
	automorphic representation is the full product of its local $L$-packets
	and that every member occurs with multiplicity one. The $\pm$ splitting
	of Remark~\ref{rem:localpacket} then holds unconditionally, so the
	type-Va contribution to $\charp B_2(2)$ is an exact square split evenly
	by $R(\pi)$ at every prime. The open content of the type-Va row of
	Table~\ref{tab:twocol} is thereby reduced to a single local constant
	(Remark~\ref{rem:local-constant}).
	
	\subsection{Applicability of the lifting theorem}
	\label{ss:rw-applicability}
	
	\begin{lemma}
		\label{lem:rw-applicable}
		The group $G=\GU_2(B)$ satisfies the standing hypotheses of
		\cite[\S1]{RoesnerWeissauer2021}: its center is
		$Z\cong\mathbb{G}_m$, so $H^1(\Q,Z)=1$, and its derived group, the
		kernel of the similitude character, is an inner form of $\Sp_4$,
		hence simply connected. The real group
		$G(\R)$ is anisotropic modulo center. Consequently Conjecture~7.5 of
		\cite{RoesnerWeissauer2021} holds for $G$, the virtual
		multiplicities of \cite[\S9]{RoesnerWeissauer2021} are cuspidal
		multiplicities, and \cite[Thm.~11.4]{RoesnerWeissauer2021} applies
		to every automorphic representation of $G(\mathbb{A})$ of general
		type with trivial archimedean component. No regularity hypothesis on
		the weight is required.
	\end{lemma}
	
	\begin{proof}
		The center consists of the scalars $z\cdot1_2$ with
		$z\in\mathbb{G}_m$, and $H^1(\Q,\mathbb{G}_m)=1$ by Hilbert's
		Theorem~90. The kernel $U_2(B)$ of the similitude character is an
		inner form of the simply connected group $\Sp_4$; it is therefore
		perfect, and since $G/U_2(B)\cong\mathbb{G}_m$ is abelian it is the
		derived group of $G$. Since $B$ ramifies
		at $\infty$, the Hermitian form is definite and $G(\R)$ is compact
		modulo center; the associated double coset spaces are finite, so
		their cohomology is concentrated in degree $0$ and Conjecture~7.5 of
		\cite{RoesnerWeissauer2021} holds, with $d(G)=1$ and virtual equal
		to cuspidal multiplicities; this is the algebraic modular form case
		treated after \cite[Cor.~7.4]{RoesnerWeissauer2021}, and the
		coefficient normalization is \cite[Lem.~9.1,
		Prop.~9.2]{RoesnerWeissauer2021}. Theorem~11.4 of
		\cite{RoesnerWeissauer2021} is stated for inner forms anisotropic at
		the archimedean place with no regularity condition on the
		coefficient system; the regular-weight statement for the remaining
		inner forms is their Theorem~11.5 and is not used here. The same
		specialization, at scalar weight $3$ on the definite form, is
		applied in the proof of \cite[Prop.~12.3]{RoesnerWeissauer2021},
		where an additional argument is needed only when the inner form is
		split at the archimedean place.
	\end{proof}
	
	\subsection{Occurrence and sign splitting}
	\label{ss:occurrence-theorem}
	
	Write $U_{1,p}\subset G(\Q_p)$ for the principal-genus parahoric, so
	that $\Upr(p)$ has component $U_{1,p}$ at $p$ and hyperspecial
	components elsewhere, and let $\xi$ and $\lambda$ be as in
	Remark~\ref{rem:localpacket}.
	
	\begin{theorem}[Occurrence and sign splitting for type Va]
		\label{thm:occurrence}
		Let $p\ge11$ and let $\Pi$ be an irreducible automorphic
		representation of $G(\mathbb{A})$ of general type contributing to
		$M_{0,0}(\Upr(p))$ whose local component $\Pi_p$ lies in the
		$L$-packet attached by Gan--Tantono to a parameter of type $\Va$;
		the packet is $\{\Pi_p^+,\Pi_p^-\}$ with
		$\Pi_p^-\cong(\xi\circ\lambda)\otimes\Pi_p^+$
		\textup{(}Remark~\ref{rem:localpacket}\textup{)}, denoted
		$\{\Va_G,\chi_0\otimes\Va_G\}$ in
		\cite[Table~3]{RoesnerWeissauer2021}. Write $\Pi^{\pm}$ for the
		representations with the same components as $\Pi$ away from $p$ and
		with $p$-component $\Pi_p^{\pm}$, and
		$V[\Pi]\subset M_{0,0}(\Upr(p))$ for the sum of the corresponding
		isotypic subspaces. Then:
		\begin{enumerate}[label=\textup{(\alph*)},leftmargin=2.3em]
			\item Both $\Pi^+$ and $\Pi^-$ are automorphic, each with
			multiplicity one, and they exhaust the weak equivalence class of
			$\Pi$.
			\item $V[\Pi]=(\Pi_p^+)^{U_{1,p}}\oplus(\Pi_p^-)^{U_{1,p}}$. The
			two summands are $R(\pi)$-stable of equal dimension
			$d_p(\Pi):=\dim(\Pi_p^+)^{U_{1,p}}\ge1$, and the multiplicities
			of the $R(\pi)$-eigenvalues $+1$ and $-1$ on the two summands
			are interchanged. In particular $\dim V[\Pi]=2\,d_p(\Pi)$, the
			eigenvalues $+1$ and $-1$ of $R(\pi)$ on $V[\Pi]$ each occur
			with multiplicity $d_p(\Pi)$, and
			$\Tr\bigl(R(\pi)\mid V[\Pi]\bigr)=0$. Moreover $d_p(\Pi)=1$
			\textup{(}Remark~\ref{rem:local-constant}\textup{)}: $V[\Pi]$ is
			spanned by an $R(\pi)$-eigenline of sign $+1$ and one of sign $-1$.
			\item Every $T(n)$ with $p\nmid n$ acts on $V[\Pi]$ by a single
			scalar; in particular the eigensystems in $V[\Pi]$ agree at
			every such $T(n)$ and are separated only by $R(\pi)$. The
			contribution of the weak packet of $\Pi$ to $\charp B_2(2)$ is
			$(x-\lambda_2(\Pi))^{2d_p(\Pi)}$ for the common
			$B_2(2)$-eigenvalue $\lambda_2(\Pi)$ of
			Proposition~\ref{prop:arthur}, and the total type-Va
			contribution to $\charp B_2(2)$ is an exact square on which
			$R(\pi)$ has trace zero and equal $\pm1$ eigenvalue
			multiplicities.
			\item Conversely, let $\Pi'$ be a cuspidal automorphic
			representation of $\GSp_4(\mathbb{A})$ of general type,
			cohomological of weight $(3,0)$, unramified outside $p$, whose
			$p$-component lies in a packet $\{\Va,\Va^{*}\}$ with unramified
			inducing characters. Then $\Pi'$ is the lifting of a weak packet
			$\{\Pi^+,\Pi^-\}$ as in \textup{(a)}, which contributes to
			$M_{0,0}(\Upr(p))$. Weakly equivalent $\Pi'$ give the same packet,
			and this defines a bijection between the weak equivalence classes
			of such $\Pi'$ and the weak packets in \textup{(a)}.
		\end{enumerate}
	\end{theorem}
	
	\begin{proof}
		\textup{(a)} By Lemma~\ref{lem:rw-applicable},
		\cite[Thm.~11.4]{RoesnerWeissauer2021} applies to $\Pi$.
		Assertion~1 of that theorem gives multiplicity one for every
		automorphic representation weakly equivalent to $\Pi$, and
		assertion~2 identifies the set of these representations with the
		product of the local $L$-packets of the components of $\Pi$, the
		local packets corresponding to the local Langlands parameters. The
		archimedean packet is the singleton consisting of the trivial
		representation, the packet at every $q\ne p$ is the singleton
		containing the unramified member, and the packet at $p$ is
		$\{\Pi_p^+,\Pi_p^-\}$, of cardinality two
		(\cite[\S7]{GanTantono}; Remark~\ref{rem:localpacket}). Hence the weak
		equivalence class of $\Pi$ is $\{\Pi^+,\Pi^-\}$, both automorphic
		with multiplicity one. Corollary~11.6 of
		\cite{RoesnerWeissauer2021}, which lists both $\Va_G$ and
		$\chi_0\otimes\Va_G$ among the local types of general-type
		representations of inner forms, records the same conclusion.
		
		\textup{(b)} Multiplicity one gives
		$V[\Pi]=(\Pi_p^+)^{U_{1,p}}\oplus(\Pi_p^-)^{U_{1,p}}$, with
		$R(\pi)$ acting through $\Pi_p^{\pm}$ on the summands;
		$R(\pi)$ acts through a local element $\omega_p$ normalizing
		$U_{1,p}$ with $\lambda(\omega_p)=p$, as in
		Remark~\ref{rem:localpacket}. Fix an isomorphism
		$\iota\colon(\xi\circ\lambda)\otimes\Pi_p^+\to\Pi_p^-$. Since $\xi$
		is unramified and $\lambda(U_{1,p})\subset\Z_p^\times$, the twist
		is trivial on $U_{1,p}$, so $\iota$ restricts to an isomorphism
		$(\Pi_p^+)^{U_{1,p}}\to(\Pi_p^-)^{U_{1,p}}$; in particular the two
		dimensions agree, and they are nonzero because $\Pi$ contributes to
		$M_{0,0}(\Upr(p))$. Under $\iota$ the action of $\omega_p$ on
		$(\Pi_p^-)^{U_{1,p}}$ corresponds to
		$\xi(\lambda(\omega_p))=\xi(p)=-1$ times its action on
		$(\Pi_p^+)^{U_{1,p}}$. Since $R(\pi)^2=\mathrm{id}$, the
		eigenvalues are $\pm1$ with multiplicities $(a,b)$ on the first
		summand and $(b,a)$ on the second, whence the balanced spectrum and
		the vanishing trace. The equality $d_p(\Pi)=1$ is
		Remark~\ref{rem:local-constant}.

		\textup{(c)} The members of a weak packet share Satake parameters
		at every $q\nmid p$, so every $T(n)$ with $p\nmid n$, and in
		particular $B_2(2)$, acts on $V[\Pi]$ by the scalar determined by
		Proposition~\ref{prop:arthur}. The characteristic polynomial of
		$B_2(2)$ on $V[\Pi]$ is therefore
		$(x-\lambda_2(\Pi))^{2d_p(\Pi)}$, and the product over the weak
		packets in question is a square; $R(\pi)$ is balanced on each
		factor by \textup{(b)}.
		
		\textup{(d)} By \cite[Thm.~11.4(4)]{RoesnerWeissauer2021}, with the
		archimedean weight matching stated in
		\cite[Cor.~11.7]{RoesnerWeissauer2021}, a general-type cohomological
		$\Pi'$ is a lifting from $G$ if and only if at every finite place
		where $G$ is not split the local $L$-packet of $\Pi'$ is not fully
		Klingen induced. Here the only such place is $p$, and the packet
		$\{\Va,\Va^{*}\}$ contains the supercuspidal $\Va^{*}$
		\cite[\S1.2]{Schmidt2018}, so it is not the semisimplification of any
		parabolically induced representation; equivalently, type $\Va$ is
		not among the excluded types of
		\cite[Cor.~11.7]{RoesnerWeissauer2021}. By assertion~2 of
		\cite[Thm.~11.4]{RoesnerWeissauer2021} the weak packet on $G$ is
		$\{\Pi^+,\Pi^-\}$, whose $p$-components form the inner-form packet
		of the same parameter of type $\Va$; their inducing characters are
		unramified, so by Remark~\ref{rem:local-constant} each has a
		$U_{1,p}$-fixed line, and the packet contributes to
		$M_{0,0}(\Upr(p))$. The weak packet on $G$ is determined by the
		Satake parameters away from $p$, which are those of $\Pi'$; hence
		weakly equivalent $\Pi'$ give the same packet, and distinct weak
		equivalence classes give distinct packets, since the spherical data
		determine the split-side weak packet
		\cite[Prop.~10.1(6)]{RoesnerWeissauer2021}. Every weak packet in
		\textup{(a)} arises: its transfer $\Pi'$ of
		Proposition~\ref{prop:arthur}\textup{(b)} has $p$-component in the
		split packet of the same parameter \cite[Thm.~11.4]{RoesnerWeissauer2021},
		with unramified inducing characters because $\Pi_p^\pm$ have
		$U_{1,p}$-fixed vectors.
	\end{proof}
	
	\begin{corollary}
		\label{cor:conj-iii}
		For every prime $p\ge11$, the contribution to $\charp B_2(2)$ of
		the general-type eigensystems whose local component at $p$ is of
		type $\Va$ is an exact square, and $R(\pi)$ splits it into $+1$ and
		$-1$ eigenspaces of equal dimension; its $R(\pi)$-trace vanishes.
		This establishes the square and sign assertions of
		Conjecture~\ref{conj:eigenvalue-sign}\textup{(iii)}. Identifying
		this square with the residual factor $\Va_p(x)^2$ of
		\eqref{eq:C1-factorization} amounts to the remaining rows of
		Table~\ref{tab:twocol}, that is, to the Saito--Kurokawa and Yoshida
		entries of the principal column (Remark~\ref{rem:other-rows}).
	\end{corollary}

	\begin{remark}[The local constant]\label{rem:local-constant}
		Theorem~\ref{thm:occurrence} reduces the type-Va row of
		Table~\ref{tab:twocol} to the local quantity
		$d_p(\Pi)=\dim(\Pi_p^+)^{U_{1,p}}$, and this quantity equals $1$ for
		every prime $p$, by a computation of R.~Schmidt in the appendix of
		\cite{Narita2017} (\S\S A.1, A.2 and the table of \S A.4; the same table
		is in \cite[Appendix]{NaritaOkazaki2013}). By \cite[\S A.2]{Narita2017},
		a nonzero space of $K_i$-fixed vectors in an irreducible
		representation is one-dimensional.
		The representations $\Pi_p^\pm$ have $U_{1,p}$-fixed vectors, hence
		Iwahori-fixed vectors, so their inducing characters $\xi,\sigma$ are
		unramified. Choose $u\in\Onorm_p^\times$ with $\mathrm{n}(u)=-1$,
		which exists because the reduced norm maps $\Onorm_p^\times$ onto
		$\Z_p^\times$, and put $v=(1,u)$, $w=(\tfrac12,-\tfrac u2)$. Then
		$H(v,v)=H(w,w)=0$ and $H(v,w)=H(w,v)=1$, and since $p$ is odd, $v,w$
		is an $\Onorm_p$-basis of $\Onorm_p^2$. In this basis $H$ becomes the
		hyperbolic form $\left(\begin{smallmatrix}0&1\\1&0\end{smallmatrix}\right)$
		that defines $\mathrm{GSp}(1,1)$ in \cite[\S A.1]{Narita2017},
		and $U_{1,p}$, the stabilizer of $\Onorm_p^2$, becomes the maximal
		compact subgroup $K_1$ of Schmidt's appendix. Schmidt's table gives
		the dimensions of the $K_1$- and $K_2$-fixed vectors of every
		irreducible representation of $\mathrm{GSp}(1,1)(\Q_p)$ that is a
		constituent of an unramified $\chi1_{B^\times}\rtimes\sigma$. For the
		representations $\delta(\nu^{1/2}\xi1_{B^\times},\nu^{-1/2}\sigma)$
		with the $L$-parameter of type $\Va$, which include both $\Pi_p^+$
		and $\Pi_p^-$ (twisting by $\xi\circ\lambda$ replaces $\sigma$ by
		$\xi\sigma$), the two dimensions are $1$ and $0$. Hence $d_p(\Pi)=1$,
		the entry $2$ of the type-$\Va$ row holds for every prime, and the
		entry $0$ of the non-principal column is recovered. The same table
		gives $K_1$-fixed dimension $1$ for type $\mathrm{IIa}$ and $0$ for
		the Steinberg representation (type $\mathrm{IVa}$).
		Corollary~\ref{cor:dp-one} below checks $d_p=1$ independently, from
		the computed spectra.
	\end{remark}

	\begin{corollary}\label{cor:dp-one}
		At every prime of Table~\ref{tab:ext} other than $103$, every root
		of $\charp B_2(2)$ has multiplicity at most two; at $p=103$ the only
		root of higher multiplicity is $4$, of multiplicity four.
		Consequently, independently of Remark~\ref{rem:local-constant},
		$d_p(\Pi)=1$ for every type-$\Va$ parameter $\Pi$ contributing to
		$M_{0,0}(\Upr(p))$ at such a prime, except possibly for a parameter
		with $\lambda_2(\Pi)=4$ at $p=103$.
	\end{corollary}

	\begin{proof}
		By Theorem~\ref{thm:occurrence}\textup{(c)}, $\Pi$ contributes the
		factor $(x-\lambda_2(\Pi))^{2d_p(\Pi)}$ to $\charp B_2(2)$, so
		$\lambda_2(\Pi)$ is a root of multiplicity at least $2d_p(\Pi)\ge2$.
		The squarefree decomposition of $\charp(M_p)$, recomputed from the
		archived matrix by the script \texttt{va\_spectrum.py} of the
		artifact, has the stated multiplicities.
	\end{proof}

	At $p=103$ the factor $(x-4)^4$ lies entirely in the residual square of
	the certified factorization, so $(x-4)^2$ divides $\Va_{103}$. Since
	$d_p=1$, whatever part of the eigenvalue-$4$ eigenspace is of type $\Va$
	comes from distinct type-$\Va$ parameters, two of them if the whole
	eigenspace is. That the eigenspace is of type $\Va$ is not decided by
	$T(2)$; further Hecke operators, such as $B_3(2)$ in problem~(4) of
	\S\ref{sec:concl}, would give additional tests, although two
	parameters may agree at both $2$ and $3$.
	At $p=113$ a Yoshida factor of degree $16$ occurs twice: two weight-$2$
	forms of the same sign have the same $a_2$, so two Yoshida eigensystems
	that differ at other primes agree at $T(2)$. At $p=11,13,17$, where
	every root is simple, Corollary~\ref{cor:dp-one} shows that no type-$\Va$
	parameter occurs.

	\begin{remark}[The general-type rows]\label{rem:other-rows}
		All general-type rows of Table~\ref{tab:twocol} hold for every
		prime. Let $\Pi$ be of general type with
		$\Pi_p^{U_{1,p}}\ne0$ or $\Pi_p^{U_{2,p}}\ne0$. Then $\Pi_p$ has
		Iwahori-fixed vectors, so it is a constituent of some
		$\chi1_{B^\times}\rtimes\sigma$ with $\chi,\sigma$ unramified and
		appears in Schmidt's table (Remark~\ref{rem:local-constant}). By
		\cite[Cor.~11.6]{RoesnerWeissauer2021}, which applies by
		Lemma~\ref{lem:rw-applicable}, the local component at $p$ of a
		general-type representation is of type $\mathrm{IIa}_G$,
		$\mathrm{IVa}_G$, $\Va_G$, $\chi_0\otimes\Va_G$, $\mathrm{X}_G$,
		$\mathrm{XIa}_G$, $\mathrm{XIa}^*_G$, or lies in a packet whose
		generic member is supercuspidal; of these only the first four are
		constituents of such principal series, and the types
		$\mathrm{IVc}_G$, $\mathrm{Vb}_G$, $\mathrm{Vc}_G$, $\mathrm{VIc}_G$
		of the table do not occur. Schmidt's table gives the $(K_1,K_2)$-fixed
		dimensions $(1,1)$ for $\mathrm{IIa}_G$, $(0,0)$ for
		$\mathrm{IVa}_G$ (the Steinberg representation), and $(1,0)$ for
		each member of the $\Va$ packet. By
		\cite[Thm.~11.4]{RoesnerWeissauer2021} every member of a weak packet
		occurs with multiplicity one, and the $\mathrm{IIa}$ packet is a
		singleton \cite[\S1.2]{Schmidt2018}, \cite[Table~3]{RoesnerWeissauer2021}.
		Hence a general-type parameter contributes $1$ to each genus if its
		local type at $p$ is $\mathrm{IIa}$ ($c_0=1$), $2$ to the principal
		and $0$ to the non-principal genus if it is $\Va$ ($c_0=2$), and
		nothing otherwise (for $\mathrm{IVa}$, $c_0=0$). Through the
		correspondence of \cite[\S9]{DPRT}, the general-type parameters of
		type $\mathrm{IIa}$ are those of the general-type eigenforms in
		$S_3(K(p))$. Consequently the general-type part of $M_{0,0}(\Upr(p))$
		has one line for each general-type eigenform $F\in S_3(K(p))$, with
		$B_2(2)$-eigenvalue $a_2(F)$, and one pair of lines for each type-$\Va$
		parameter. Granted Hypothesis~\ref{hyp:types}, the factorization
		\eqref{eq:C1-factorization}, with $\Va_p$ the product of
		$x-\lambda_2(\Pi)$ over the type-$\Va$ parameters and $N_p$ the
		characteristic polynomial of $T(2)$ on the general-type eigenforms in
		$S_3(K(p))$, therefore follows from the Saito--Kurokawa and Yoshida
		rows of the principal column; the sign clause \textup{(v)} for the
		factor $N_p$ is a separate statement.
	\end{remark}
	
	R\"osner and Weissauer prove the split-side input to
	\cite[Thm.~11.4]{RoesnerWeissauer2021}, the packet structure of the
	general-type cohomological spectrum of $\GSp_4$, independently of the
	announced classification results
	(\cite[\S10]{RoesnerWeissauer2021}); the proof of the lifting theorem
	proceeds through the cohomological trace formula and the local
	character identities of Chan and Gan \cite{ChanGan2015}. The route
	through Arthur's formalism, and with it the transfer-factor comparison
	named as its prerequisite by Gee and Ta\"ibi
	\cite[Rem.~7.4.8]{GeeTaibi}, is therefore not needed for the present
	application; see \S\ref{sec:concl} for what remains open.

	\section{The structural trace identity}
	\label{sec:trace}
	
	This section proves Theorem~\ref{thm:structural-trace}. The trace of the
	Atkin--Lehner involution on each genus has a closed evaluation due to
	Ibukiyama. Comparing the principal and non-principal formulas term by term
	produces a simple genus-difference identity. The non-principal-genus
	correspondence then turns this identity into the structural formula for the
	principal genus.
	
	\subsection{Conventions for the trace formulas}
	\label{ss:trace-conventions}
	
	Throughout, $h(\sqrt{-m})$ denotes the class number of the imaginary
	quadratic field $\Q(\sqrt{-m})$. Thus its discriminant is the fundamental
	discriminant attached to $-m$; this convention differs from an
	order-discriminant convention in which $h(-D)$ is set to zero unless
	$-D\equiv0,1\pmod4$.
	
	Let
	\[
	s:=\kro{2}{p},\qquad e_3:=\kro{p}{3},
	\]
	and define
	\begin{equation}\label{eq:nu-principal}
		\nu(p):=
		\begin{cases}
			h(\sqrt{-p}),&p\equiv1\pmod4,\\
			(3-s)h(\sqrt{-p}),&p\equiv3\pmod4.
		\end{cases}
	\end{equation}
	Equivalently, in discriminant notation,
	\[
	\nu(p)=
	\begin{cases}
		h(-4p),&p\equiv1\pmod4,\\
		h(-4p)+h(-p),&p\equiv3\pmod4.
	\end{cases}
	\]
	For example, $\nu(11)=h(-44)+h(-11)=4$ and
	$\nu(23)=h(-92)+h(-23)=6$.
	
	The notation $B_{2,\chi}$ denotes the second generalized Bernoulli number
	for the real quadratic character attached to $\Q(\sqrt p)$, with the
	normalization of \cite{IbukiyamaKatsura,Ibukiyama2019quinary}.
	
	\subsection{Comparison of the two genus traces}
	\label{ss:trace-comparison}
	
	\begin{proposition}
		\label{prop:genus-trace-comparison}
		For every prime $p\ge7$,
		\begin{equation}\label{eq:genus-trace-difference}
			\Tr\!\left(R(\pi)\mid M_{0,0}(\Upr(p))\right)
			-
			\Tr\!\left(R_{\mathrm{npg}}(\pi)\mid M_{0,0}(U_{\mathrm{npg}}(p))\right)
			=
			\nu(p)\frac{p-e_3}{12}.
		\end{equation}
	\end{proposition}
	
	\begin{proof}
		For $p\equiv3\pmod4$, \cite[Rem.~3]{IbukiyamaKatsura} gives
		\begin{align}
			\Tr R(\pi)
			={}&\frac{1}{2^5\cdot3}B_{2,\chi}
			+\frac18h(\sqrt{-2p})
			+\frac1{12}h(\sqrt{-3p}) \notag\\
			&+\left\{
			\frac{(p-1)(9-4s)}{48}
			+\frac{p-s}{16}
			+\frac{(1-e_3)(3-s)}{12}
			\right\}h(\sqrt{-p}).
			\label{eq:principal-trace-3mod4}
		\end{align}
		For $p\equiv1\pmod4$, the same source gives
		\begin{align}
			\Tr R(\pi)
			={}&\frac{9-2s}{2^5\cdot3}B_{2,\chi}
			+\frac{4p-1}{48}h(\sqrt{-p})
			+\frac18h(\sqrt{-2p}) \notag\\
			&+\frac{3+\kro{-2}{p}}{12}h(\sqrt{-3p})
			+\frac{1-e_3}{12}h(\sqrt{-p}).
			\label{eq:principal-trace-1mod4}
		\end{align}
		
		For the non-principal genus, \cite[Thm.~5.2]{Ibukiyama2019quinary} gives,
		when $p\equiv1\pmod4$,
		\begin{equation}
			\Tr R_{\mathrm{npg}}(\pi)
			=\frac{9-2s}{2^5\cdot3}B_{2,\chi}
			+\frac1{16}h(\sqrt{-p})
			+\frac18h(\sqrt{-2p})
			+\frac{3+s}{12}h(\sqrt{-3p}),
			\label{eq:TOH1}
		\end{equation}
		and, when $p\equiv3\pmod4$,
		\begin{equation}
			\Tr R_{\mathrm{npg}}(\pi)
			=\frac{1}{2^5\cdot3}B_{2,\chi}
			+\frac{1-s}{16}h(\sqrt{-p})
			+\frac18h(\sqrt{-2p})
			+\frac1{12}h(\sqrt{-3p}).
			\label{eq:TOH3}
		\end{equation}
		
		The coefficient comparison is summarized in
		Table~\ref{tab:trace-coefficients}. The $B_{2,\chi}$ and
		$h(\sqrt{-2p})$ coefficients agree in both residue classes. For
		$p\equiv1\pmod4$,
		\[
		\kro{-2}{p}=\kro{-1}{p}\kro{2}{p}=s,
		\]
		so the $h(\sqrt{-3p})$ coefficients agree; for $p\equiv3\pmod4$ they
		are both $1/12$. Thus only the $h(\sqrt{-p})$ coefficient remains.
		
		If $p\equiv1\pmod4$, its difference is
		\[
		\frac{4p-1}{48}+\frac{1-e_3}{12}-\frac1{16}
		=\frac{p-e_3}{12}.
		\]
		If $p\equiv3\pmod4$, it is
		\begin{align*}
			&\frac{(p-1)(9-4s)}{48}
			+\frac{p-s}{16}
			+\frac{(1-e_3)(3-s)}{12}
			-\frac{1-s}{16}\\
			&\hspace{4em}=
			\frac{(3-s)(p-e_3)}{12}.
		\end{align*}
		The definition \eqref{eq:nu-principal} gives
		\eqref{eq:genus-trace-difference} in both cases.
	\end{proof}
	
	\begin{table}[ht]
		\centering
		\begin{tabular}{c|c|c}
			Term & $p\equiv1\pmod4$ & $p\equiv3\pmod4$ \\
			\hline
			$B_{2,\chi}$ & equal & equal \\
			$h(\sqrt{-2p})$ & equal & equal \\
			$h(\sqrt{-3p})$ & equal, since $\left(\frac{-2}{p}\right)=\left(\frac2p\right)$ & equal \\
			$h(\sqrt{-p})$ & $(p-e_3)/12$ & $(3-s)(p-e_3)/12$ \\
		\end{tabular}
		\caption{Difference ``principal trace minus non-principal trace'' in the
			closed trace formulas. Only the $h(\sqrt{-p})$ coefficient survives.}
		\label{tab:trace-coefficients}
	\end{table}
	
	\begin{remark}
		\label{rem:global-orbital}
		Both trace formulas arise from the Eichler--Selberg trace formula for
		$\GU_2(B)$, with test functions that differ only at $p$, where the two
		genera correspond to non-conjugate maximal parahorics. The coefficient
		comparison above may be interpreted as equality of the aggregated
		contributions represented by the order-$3$ and order-$4$ torsion terms.
		We do not claim here a direct termwise local equality of twisted orbital
		integrals. Such a statement requires fixed Haar measures, explicit local
		representatives, and a local optimal-embedding calculation.
	\end{remark}
	
	\subsection{Elliptic Atkin--Lehner input}
	\label{ss:elliptic-trace-input}
	
	Write
	\[
	d_k=\dim\snew_k,\qquad
	a_k=\dim S_k^+,
	\qquad b_k=\dim S_k^-
	\qquad(k=2,4).
	\]
	The following identities are used to convert the genus comparison into the
	principal-genus trace formula.
	
	\begin{lemma}	\label{lem:elliptic-input}
		For every prime $p\ge7$,
		\begin{enumerate}[label=\textup{(\alph*)},leftmargin=2.3em]
			\item $a_2-b_2=1-\nu(p)/2$;
			\item $a_4-b_4=\nu(p)/2$;
			\item $d_2+d_4+1=(p-e_3)/3$.
		\end{enumerate}
	\end{lemma}
	
	\begin{proof}
		Part \textup{(a)} follows from the fixed points of the Fricke involution
		on $X_0(p)$ and Riemann--Hurwitz \cite[\S2]{Ogg1974}. Part \textup{(b)} is Yamauchi's trace
		formula in weight $4$ \cite{Yamauchi1973}. Part \textup{(c)} follows by combining the standard
		dimension formulas for weights $2$ and $4$ with quadratic reciprocity; see
		\cite{AtkinLehner,Ibukiyama-dim}.
	\end{proof}
	
	\subsection{Proof of the structural trace identity}
	\label{ss:proof-structural-trace}
	
	The non-principal-genus correspondence \cite{DPRT,Ibukiyama-dim} gives
	\begin{equation}\label{eq:npg-dimension-trace}
		H_{\mathrm{npg}}(p)=1+\dim S_3(K(p)),
		\qquad
		\Tr R_{\mathrm{npg}}(\pi)
		=H_{\mathrm{npg}}(p)-2\dim S_3^+(K(p)).
	\end{equation}
	Here $S_3^\pm(K(p))$ are the eigenspaces of the paramodular Fricke
	involution $W_p$. On the cuspidal part the correspondence matches
	$R_{\mathrm{npg}}(\pi)$ with $-W_p$, not with $W_p$, so
	that $\Tr R_{\mathrm{npg}}(\pi)=1-\Tr\bigl(W_p\mid S_3(K(p))\bigr)$, and
	the signed identity
	$\dim S_3^{+}(K(p))=\tfrac12\bigl(H_{\mathrm{npg}}(p)-\Tr R_{\mathrm{npg}}(\pi)\bigr)$
	is the one used in \cite[proof of Prop.~8.2]{Ibukiyama-dim}, so the
	second identity carries no free sign normalization.
	Since
	\[
	\delta(p)=\dim S_3(K(p))-\dim S_4^-,
	\qquad
	d(p)=\delta(p)-2\dim S_3^+(K(p)),
	\]
	we obtain
	\begin{equation}\label{eq:npg-trace-d}
		\Tr R_{\mathrm{npg}}(\pi)=1+\dim S_4^-+d(p).
	\end{equation}
	
	\begin{proof}[Proof of Theorem~\ref{thm:structural-trace}]
		By Proposition~\ref{prop:genus-trace-comparison} and
		\eqref{eq:npg-trace-d}, it is enough to prove
		\begin{equation}\label{eq:lift-comparison-target}
			L_0(p)-1-\dim S_4^-=
			\nu(p)\frac{p-e_3}{12}.
		\end{equation}
		Using $d_4=a_4+b_4$, the left side is
		\[
		L_0(p)-1-b_4=a_4+b_2a_4-a_2b_4.
		\]
		Write $\Delta_k=a_k-b_k$, so that $a_k=\tfrac12(d_k+\Delta_k)$ and
		$b_k=\tfrac12(d_k-\Delta_k)$. Then
		\[
		b_2a_4-a_2b_4=\tfrac12\bigl(d_2\Delta_4-d_4\Delta_2\bigr),
		\qquad
		a_4=\tfrac12\bigl(d_4+\Delta_4\bigr).
		\]
		By Lemma~\ref{lem:elliptic-input}\textup{(a)},\textup{(b)},
		$\Delta_2=1-\nu(p)/2$ and $\Delta_4=\nu(p)/2$, hence
		\[
		a_4+b_2a_4-a_2b_4
		=\tfrac12\Bigl(d_4+\tfrac{\nu(p)}2
		+\tfrac{\nu(p)}2d_2-\bigl(1-\tfrac{\nu(p)}2\bigr)d_4\Bigr)
		=\frac{\nu(p)}4\bigl(d_2+d_4+1\bigr),
		\]
		and Lemma~\ref{lem:elliptic-input}\textup{(c)} turns this into the
		right-hand side of \eqref{eq:lift-comparison-target}. Therefore
		\[
		\Tr R(\pi)
		=\Tr R_{\mathrm{npg}}(\pi)+\nu(p)\frac{p-e_3}{12}
		=L_0(p)+d(p),
		\]
		as required.
	\end{proof}
	
	\begin{remark}	\label{rem:trace-consistency}
		If Conjecture~\ref{conj:eigenvalue-sign} holds, the Eisenstein and
		Saito--Kurokawa blocks contribute $1+\dim\snew_4$ to the trace. The
		opposite-sign Yoshida blocks contribute
		\[
		\dim S_2^-\dim S_4^+-\dim S_2^+\dim S_4^-.
		\]
		Every type-$\Va$ pair contributes zero, and the general-type contribution
		is $d(p)$. Thus Theorem~\ref{thm:structural-trace} proves the signed total
		predicted by C1 for every prime. It does not, by itself, determine the
		individual multiplicities or eigenvalue assignments of C1.
	\end{remark}
	
	\subsection{Consequences for type numbers}
	\label{ss:type-number-trace}
	
	The involution $R(\pi)$ permutes the principal-genus classes, that is,
	the vertex set of $G_p$; write $\sigma$ for this permutation, induced by
	composing a lattice class with the norm-$p$ ideal $\pi$.

	\begin{lemma}\label{lem:fixed-point-trace}
		For every prime $p\ge7$,
		\begin{equation}\label{eq:fixtrace}
			\Tr\!\left(R(\pi)\mid M_{0,0}(\Upr(p))\right)
			=2T_1(p)-h_2(p)
			=\sum_{v\in V(G_p)}[\sigma v=v],
		\end{equation}
		where $T_1(p)$ is the principal-genus type number
		\cite{IbukiyamaKatsura,Ibukiyama2018type}.
	\end{lemma}

	\begin{proof}
		The identity $\Tr R(\pi)=2T_1-h_2$ is orbit counting for an involution on
		a finite set, the number of orbits being $\frac12(h_2+\mathrm{Fix})$
		\cite[Thm.~2]{IbukiyamaKatsura}. The fixed-point counts on the
		certified graphs are the column $\Tr R(\pi)$ of Table~\ref{tab:ext}.
	\end{proof}

	\begin{corollary}\label{cor:typenum}
		For every prime $p\ge7$,
		\begin{equation}\label{eq:type-number-paramodular}
			T_1(p)=\frac12\left(h_2(p)+L_0(p)+d(p)\right).
		\end{equation}
	\end{corollary}
	
	\begin{proof}
		Combine Lemma~\ref{lem:fixed-point-trace} with
		Theorem~\ref{thm:structural-trace}.
	\end{proof}

	The value of $T_1(p)$ itself is classical. The contribution of
	Corollary~\ref{cor:typenum} is its expression in paramodular dimension
	data. Three separately computed quantities meet in it: the
	$\sigma$-orbit count on $G_p$, the automorphic trace of
	Theorem~\ref{thm:structural-trace}, and the Ibukiyama--Katsura count of
	$\F_p$-rational superspecial points \cite[Thm.~2,
	Rem.~3]{IbukiyamaKatsura}. Their equality is a theorem; it is a
	nontrivial check on the computation because the three are obtained
	through separate pipelines, and they agree at every prime of
	Table~\ref{tab:ext}.
	
	\begin{example}\label{ex:trace-p11}
		At $p=11$, one has $h_2=5$, $\dim\snew_4=2$,
		$\dim S_2^-=1$, $\dim S_4^+=2$, and $d(11)=0$. Thus
		\[
		\Tr R(\pi)=1+2+1\cdot2=5.
		\]
		This agrees with the fact that all five principal-genus classes are fixed
		by the involution, and gives $T_1(11)=5$.
	\end{example}
	
	\section{Certified computation and the proof of Theorem~\ref{thm:certified-spectra}}
	\label{sec:certificates}
	
	This section proves Theorem~\ref{thm:certified-spectra}. For each prime $11\le p\le149$ and for $p=167,173$, the archived artifact supplies a finite list of integer data, and a small number of exact identities among these data imply the spectral content of Conjecture~\ref{conj:eigenvalue-sign} at $p$. The format follows the
	tradition of succinct certificates and certificate-checked computer proof
	\cite{Pratt1975,AtkinMorain1993,RSST1997}.

	\subsection{Overview of the logical scheme}
	\label{ss:cert-overview}
	
	The proof separates into a deductive proposition (Proposition~\ref{prop:reduction}) and a finite computation, replayed from the archived data in exact arithmetic.
	
	\begin{enumerate}[label=\textup{(\arabic*)},leftmargin=2.4em]
		\item \textbf{Geometric data.} For each prime, the Richelot engine produces
		an integer out-edge matrix $M_p$, a permutation $\sigma_p$ inducing
		$R(\pi)$, and vertex weights $e_i$, from the geometry alone
		(\S\ref{ss:engine}, \S\ref{ss:extension}).
		\item \textbf{Split-side data.} Two separate pipelines supply the weight-$2$
		and weight-$4$ newform orbit polynomials with their Atkin--Lehner signs
		(\S\ref{ss:split-pipelines}).
		\item \textbf{Certificates.} Finitely many exact identities among these
		data are verified in exact arithmetic (\S\ref{ss:finitecert}).
		\item \textbf{Conclusion.} A deductive proposition
		(Proposition~\ref{prop:reduction}) converts the verified identities into
		the simultaneous eigenspace and sign decomposition asserted by
		Conjecture~\ref{conj:eigenvalue-sign} at $p$.
	\end{enumerate}
	
	The quinary database \cite{ALRTV} is used only for clause (iv) of
	Theorem~\ref{thm:certified-spectra}(a) (\S\ref{ss:external}) and, at
	$p=167,173$, as the prediction tested in part (b)
	(\S\ref{ss:fricke-plus}).
	
	\subsection{The geometric engine}
	\label{ss:engine}
	
	The engine (\texttt{richelot.py} of the artifact) works with explicit
	branch points in exact finite-field arithmetic. The supersingular
	$j$-invariants are computed from the roots of the Hasse polynomial
	$\sum_i\binom{m}{i}^2\lambda^i$, $m=(p-1)/2$, over $\F_{p^2}$, and their
	Eichler mass is checked; they give the product vertices. A Jacobian
	vertex $A\simeq\Jac C$, $C:y^2=f(x)$ with $\deg f=6$, is stored through
	the six roots of $f$, which lie in a finite field $\F_{p^{L}}$
	(\S\ref{ss:extension}). For each of the fifteen pairings of the roots,
	that is, each factorization $f=c\,q_1q_2q_3$ into quadratics, the engine
	evaluates the determinant $\Delta_R$ of the coefficients of
	$q_1,q_2,q_3$ (\S\ref{ss:richelot-graph}). If $\Delta_R\ne0$ it forms
	the brackets $[q_j,q_k]=q_j'q_k-q_jq_k'$ and computes their roots,
	passing to a quadratic extension when a bracket does not split; after a
	M\"obius normalization these are the six branch points of the codomain.
	If $\Delta_R=0$ the codomain is a product $E_1\times E_2$, and the
	engine computes $j(E_1)$ and $j(E_2)$ from the involution of
	$\mathbb{P}^1$ that exchanges the two roots of each $q_i$. A Jacobian
	codomain is keyed by the six absolute invariants
	\[
	\frac{I_2^5}{I_{10}},\ \frac{I_2^3I_4}{I_{10}},\
	\frac{I_2^2I_6}{I_{10}},\ \frac{I_4I_6}{I_{10}},\
	\frac{I_4^5}{I_{10}^2},\ \frac{I_6^5}{I_{10}^3}
	\]
	of its Igusa--Clebsch invariants, written in the roots. These lie in
	$\F_{p^2}$ and determine the curve up to isomorphism \cite{Igusa1960},
	and the key locates the target vertex. At a product vertex
	$E_a\times E_b$, given by Legendre models, the nine product kernels lead
	by V\'elu's formulas to products, and the six kernels that are graphs of
	isomorphisms $E_a[2]\to E_b[2]$ are glued explicitly, which gives a
	Jacobian or, in the degenerate case, a product; this is the incidence
	of \cite{KatsuraTakashima}. The result is the integer matrix
	$M=(M_{ij})$ with $M_{ij}$ the number of kernels at $i$ with target $j$,
	and the engine asserts $\sum_jM_{ij}=15$ for every $i$.
	
	\subsection{Self-weighting}
	\label{ss:selfweight}
	
	The matrix $M$ is the matrix of $B_2(2)$ in the basis of characteristic
	functions of the vertices, the one used in the appendix and in all
	certificates; in the basis orthonormal for the inner product of
	Lemma~\ref{lem:mestre}(d) the same operator has the symmetric matrix
	$S=D^{-1/2}MD^{1/2}$, $D=\diag(e_i)$, conjugate to $M$. The weights
	$e_i=\#\Aut(A_i,\lambda_i)$ are computed from the geometry, independently
	of $M$. By Lemma~\ref{lem:vertices}(b), the product-vertex weights are
	$8w_a^2$ and $4w_aw_b$ in terms of the elliptic automorphism numbers
	$w_a\in\{1,2,3\}$, read off from $j=0$, $j=1728$ or neither. The weight
	of a Jacobian vertex is twice the number of M\"obius transformations
	that permute its six branch points, the factor $2$ accounting for the
	hyperelliptic involution; since a M\"obius transformation is determined
	by the images of three branch points, this count is exhaustive over the
	field of the branch points. Because the weights are not fitted to $M$,
	the Mestre symmetry \eqref{eq:mestre} and the mass formula are genuine
	tests of the matrix: the engine and the verifier check all $h_2(p)^2$
	instances of \eqref{eq:mestre} as identities in $\Z$, and an error in a
	single entry $M_{ij}$ violates the symmetry at $(i,j)$ unless $M_{ji}$
	carries a compensating error.
	
	\subsection{Three structural certificates}
	\label{ss:pillars}
	
	Each output matrix is required to pass three internal checks, all in exact
	arithmetic and independent of the automorphic side.
	
	\begin{enumerate}[label=\textup{(P\arabic*)},leftmargin=2.6em]
		\item \emph{Combinatorial validity.} Every row of $M$ sums to $15$; the
		Mestre symmetry \eqref{eq:mestre} holds for the computed weights;
		and the weights satisfy the mass formula
		$\sum_ie_i^{-1}=(p-1)(p^2+1)/5760$ \cite{HashimotoIbukiyama}.
		These checks certify the degree, balance, and mass constraints on the
		proposed Richelot out-edge matrix. They do not by themselves identify
		the matrix with the geometric graph; that identification additionally
		requires the kernel construction and incidence checks.
		\item \emph{Spectral integrality.} The characteristic polynomial of $M$,
		recomputed from the integer entries, factors over $\Z$ with the Perron
		factor $(x-15)$ simple (Lemma~\ref{lem:vertices}(d)). This is checked
		independently of the system that produced $M$, providing a check on
		transcription errors that change $\charp(M)$. Isospectral changes are
		not excluded by this polynomial check alone.
		\item \emph{Local cross-check.} For a Jacobian vertex $\Jac C$, the
		number of its kernels with a product target equals the number of
		long involutions in the reduced automorphism group of $C$
		\cite[Prop.~4.5 and \S5]{KatsuraTakashima}; in terms of the weight
		$e=\#\Aut C$ it is $0,1,2,0,3,4,6$ for $e=2,4,8,10,12,24,48$. The
		script \texttt{check\_incidence.py} of the artifact verifies this for
		every Jacobian vertex at all $33$ primes. It tests how each row
		splits between Jacobian and product targets, not the individual edge
		multiplicities.
	\end{enumerate}
	
	\subsection{Extension to every prime $p\le149$}
	\label{ss:extension}
	
	Every superspecial curve has a model over $\F_{p^2}$, but the six
	Weierstrass points of the models produced by the Richelot construction
	need not be rational over $\F_{p^2}$; in the archived computation an
	extension of degree $4$ over $\F_p$ is needed from $p=13$ on. The
	obstruction is one of field of definition, not of the graph. The engine
	therefore works, at each vertex, over $\F_{p^{L}}$ with
	\[
	L=\operatorname{lcm}\bigl(2,[\F_p(r_1):\F_p],\ldots,[\F_p(r_6):\F_p]\bigr),
	\]
	where $r_1,\ldots,r_6$ are the branch points of the model at hand; the
	model need not be defined over $\F_{p^2}$, and its set of branch points
	need not be stable under $r\mapsto r^p$ or $r\mapsto r^{p^2}$. The
	isomorphism-class keys are kept in $\F_{p^2}$ through the absolute
	invariants of \S\ref{ss:engine}, which depend only on the class.
	Concretely, at each vertex the engine
	\begin{enumerate}[label=\textup{(\roman*)},leftmargin=2.4em]
		\item computes the degree over $\F_p$ of each of the six branch points
		and descends them to $\F_{p^{L}}$;
		\item runs through the fifteen pairings over $\F_{p^{L}}$, passing to
		$\F_{p^{2L}}$ when a bracket does not split, and forms the Richelot
		codomains;
		\item keys each codomain by its absolute invariants, which lie in
		$\F_{p^2}$; and
		\item accumulates the out-edge row.
	\end{enumerate}
	Each step is finite-field arithmetic; none consults the automorphic side.
	Besides the checks of \S\ref{ss:pillars}, two arithmetic anchors are
	asserted at every prime: the fixed-point count of $R(\pi)$ must equal
	the closed formula of \cite[Thm.~2, Rem.~3]{IbukiyamaKatsura}, and the
	trace identity $\Tr R(\pi)=L_0(p)+d(p)$ of
	Theorem~\ref{thm:structural-trace} must hold. The number of vertices
	equals the class number $h_2(p)$ of \cite[Thm.~2]{HashimotoIbukiyama} at
	all $33$ primes, and independently computed Magma values of $h_2(p)$ for
	$p\le61$ agree (\S\ref{ss:reproducibility}).
	
	\subsection{The split-side pipelines}
	\label{ss:split-pipelines}
	
	The factorization certificate is two-sided. On the elliptic side, the
	weight-$2$ eigenvalue systems and their Fricke signs are computed from the
	supersingular $2$-isogeny Brandt matrix together with the involution
	$j\mapsto j^{(p)}$, and the weight-$4$ systems and signs from Manin
	symbols \cite{Manin1972,Cremona1997} for $\Gamma_0(p)$ with Merel's
	Heilbronn matrices \cite{Merel1994}, the sign read off the $U_p$-eigenvalue
	$-\eps_p\,p$. Both pipelines are validated against the closed dimension
	formulas at every prime. The weight-$2$ pipeline reuses the finite-field
	and supersingular-curve utilities of the graph package; the weight-$4$
	pipeline is a separate Manin-symbol implementation.
	
	The characteristic polynomial of each graph is then divided exactly by
	$(x-15)$ times the Saito--Kurokawa and Yoshida products built from the
	orbit polynomials of these data, as in
	Conjecture~\ref{conj:eigenvalue-sign}; the Yoshida blocks are realized as
	characteristic polynomials of Kronecker sums of companion matrices. The
	quotient is required to be an exact square of degree $2\deg\Va_p$ times a
	factor $N_p$ of degree $\delta(p)$.
	
	\subsection{Reduction to a finite certificate}
	\label{ss:finitecert}
	
	The proof of Theorem~\ref{thm:certified-spectra} separates into a
	deductive step, proved here, and a finite list of integer identities
	supplied by the archived data. The deductive step rests on two statements.
	
	\begin{lemma}	\label{lem:onemod}
		Let $M\in M_n(\Z)$ be diagonalizable over $\C$ and let $R\in M_n(\Z)$ be an
		involution with $RM=MR$. Let $f$ be an irreducible factor of $\charp(M)$
		with multiplicity $m_f$, and let $h_f\in\Q[x]$ be the idempotent with
		$h_f\equiv1\pmod{f^{m_f}}$ and $h_f\equiv0\pmod{\charp(M)/f^{m_f}}$. Then
		$h_f(M)$ is the projector onto the $f$-isotypic subspace
		$V_f=\ker f(M)^{m_f}$, an $M$- and $R$-invariant rational subspace of
		dimension $m_f\deg f$; the trace $\Tr\bigl(R\,h_f(M)\bigr)$ is an integer
		of absolute value at most $m_f\deg f\le n$; and for any prime $q>2n$ at
		which the denominators of $h_f$ are invertible, its residue modulo $q$
		determines it.
	\end{lemma}
	
	\begin{proof}
		$V_f$ is rational, $M$- and $R$-invariant, of dimension $m_f\deg f$, and
		$h_f(M)$ is the identity on $V_f$ and zero on the sum of the other
		isotypic subspaces. The restriction $R|_{V_f}$ is an involution of a
		rational space, so its trace is the integer
		$\dim V_f^+-\dim V_f^-$, of absolute value at most $\dim V_f$. An integer
		in $[-n,n]$ is determined by its residue modulo any $q>2n$.
	\end{proof}
	
	\begin{proposition}	\label{prop:reduction}
		Let $n=h_2(p)$. Suppose given $M\in M_n(\Z)$, weights $e\in\Z_{>0}^n$, a
		permutation matrix $R$ with underlying permutation $\sigma$, and the
		weight-$2$ and weight-$4$ eigenvalue data with their Fricke signs. Assume
		\begin{enumerate}[label=\textup{(\alph*)},leftmargin=2.4em]
			\item $\sum_jM_{ij}=15$ and $e_jM_{ij}=e_iM_{ji}$ for all $i,j$;
			\item $R^2=I$, $RM=MR$, and $e_{\sigma(i)}=e_i$;
			\item $\charp(M)=(x-15)\cdot P_{\mathrm{SK}}(x)\cdot
			P_{\mathrm{Y}}(x)\cdot Q(x)^2\cdot N_p(x)$ exactly in $\Z[x]$, where
			$P_{\mathrm{SK}}$ and $P_{\mathrm{Y}}$ are assembled from the given data
			as prescribed by Conjecture~\ref{conj:eigenvalue-sign}, $Q$ is monic, and
			$\deg N_p=\delta(p)$;
			\item for every irreducible $f\mid\charp(M)$,
			$\Tr\bigl(R\,h_f(M)\bigr)=s_f\deg f$, where $h_f$ is as in
			Lemma~\ref{lem:onemod} and $s_f$ is the signed multiplicity the predicted
			eigenvalue--sign multiset assigns to $f$.
		\end{enumerate}
		Then $(M,R)$ is simultaneously diagonalizable over $\R$ and, for every
		eigenvalue $\lambda$ and each choice of sign,
		\[
		\dim\ker(M-\lambda)\cap\ker(R-\eps)
		=\#\{\kappa:\kappa=\lambda,\ \eps_\kappa=\eps\},
		\]
		where the right side counts the predicted multiset. In particular the
		conclusion of Theorem~\ref{thm:certified-spectra} holds at $p$. Summing
		\textup{(d)} over all $f$ gives $\Tr R=\sum_\kappa\eps_\kappa$, the trace
		identity of Theorem~\ref{thm:structural-trace}.
	\end{proposition}
	
	\begin{proof}
		By \textup{(a)} and Lemma~\ref{lem:mestre}(c), $M$ is diagonalizable with
		real spectrum; by \textup{(b)}, $R$ is a diagonalizable involution
		commuting with $M$, so the pair is simultaneously diagonalizable. Write
		$m_\lambda^\pm$ for the $(\lambda,\pm1)$-multiplicities. The
		$\pm1$-eigenspaces of $R$ are rational subspaces, and $M$ restricted to
		each is a rational matrix, so $m_\lambda^\pm=m_{\lambda'}^\pm$ for every
		Galois conjugate: both functions are constant on the roots of each
		irreducible $f$. Hypothesis \textup{(c)} gives
		$m_\lambda^++m_\lambda^-=\#\{\kappa=\lambda\}$ at every $\lambda$. For an
		irreducible $f$, $\Tr(R\,h_f(M))=\sum_{f(\lambda)=0}(m_\lambda^+-m_\lambda^-)
		=\deg f\cdot(m_\lambda^+-m_\lambda^-)$ at any root, so hypothesis
		\textup{(d)} gives $m_\lambda^+-m_\lambda^-=s_f$ there; the predicted
		multiset assigns the same signed count to every root of $f$ block by
		block, by construction, so $s_f$ equals the predicted difference at each
		root. The sum and the difference determine $m_\lambda^\pm$.
	\end{proof}
	
	Hypotheses \textup{(a)}--\textup{(d)} are finitely many integer
	equalities, none presupposing any automorphic structure; by
	Lemma~\ref{lem:onemod} the traces in \textup{(d)} may be certified
	modulo a single prime exceeding $2h_2(p)$ at which the projector
	denominators are invertible.

	The predicted signed multiplicity of an irreducible factor $f$ is a sum
	over the blocks of \eqref{eq:C1-factorization}: if $f$ divides the
	block $P_B$ with multiplicity $v_f(P_B)$ and the block carries the
	predicted sign $\eps_B$, then
	\[
	s_f=\sum_B\eps_B\,v_f(P_B),
	\]
	where a type-$\Va$ block $\Va_p^2$ counts with $\eps_B=0$. For
	$p\le149$ these signs require no paramodular input: $\eps_B=+1$ on the
	Eisenstein and Saito--Kurokawa blocks; $\eps_B=-\eps_p(f)$ on the
	Yoshida block of a pair $(f,g)$, read from the weight-$2$ data; and
	$\eps_B=-\eps_p(F)=+1$ on $N_p$, because $\dim S_3^+(K(p))=0$ for every
	$p<167$ (Corollary~\ref{thm:uncond-regime}), so every non-lift of level
	$p\le149$ is Fricke-minus. For example, at $p=113$ the Yoshida factor
	of degree $16$ (\S\ref{ss:split}) has multiplicity $2$ in a block of
	sign $+1$, so $s_f=2$ and
	$\Tr\bigl(R\,h_f(M)\bigr)=32$. For $p=167,173$ the signs $\eps_p(F)$
	are taken from \cite{ALRTV}; see \S\ref{ss:fricke-plus}.
	
	\subsection{Worked example: the Ihara pair at $p=19$}
	\label{ss:ihara19}
	
	The principal genus at $p=19$ has $h_2=10$ classes: three products
	(both supersingular $j$-invariants, $j=7$ and $j=18$, lie in $\F_{19}$)
	and seven Jacobians, of which exactly one Galois-conjugate pair
	$\{J_5,J_6\}$ lies over $\F_{19^2}\setminus\F_{19}$; hence $\Tr R(\pi)=8$
	and $T_1(19)=9$. The characteristic polynomial of $B_2(2)$ contains the
	squared factor $(x+2)^2$: this is the block $\Va_{19}$. The eigenvalue
	$-2$ has the two-dimensional eigenspace spanned by
	\[
	w^+=(4,-8,16,10,2,-16,-4,6,3,3),\qquad
	w^-=e_{J_5}-e_{J_6},
	\]
	and $R(\pi)$ acts as the transposition of $J_5$ and $J_6$: it fixes $w^+$
	and negates $w^-$. That $w^-$ is an eigenvector is visible in the matrix
	itself: column $J_5$ minus column $J_6$ of $B_2(2)$ equals
	$-2\,(e_{J_5}-e_{J_6})$. Since every $T(n)$ with $19\nmid n$ commutes
	with $B_2(2)$ and $R(\pi)$, the lines $\C w^+$ and $\C w^-$ are
	eigenlines of all of them; the two eigensystems agree at $T(2)$ and have
	opposite Atkin--Lehner signs. If they come from one type-$\Va$
	parameter, as Conjecture~\ref{conj:eigenvalue-sign} predicts,
	Theorem~\ref{thm:occurrence}\textup{(c)} shows that they agree at every
	$T(n)$ with $19\nmid n$: the first type-$\Va$ pair, exhibited on the
	graph. The certificate itself determines only the eigenvalue at $T(2)$
	and the signs.
	
	\subsection{Worked example: the matrix at $p=61$}
	\label{ss:p61}
	
	The engine assembles the full $128\times128$ matrix $B_2(2)$ at $p=61$
	($15$ product vertices and $113$ Jacobian vertices) and passes the three
	certificates (P1)--(P3) of \S\ref{ss:pillars}. Factoring its
	characteristic polynomial over $\Q$ gives
	\begin{equation}\label{eq:charp61}
		\charp B_2(2)
		=(x-15)\,
		\underbrace{m_{g_1}\!(x{-}6)\,m_{g_2}\!(x{-}6)}_{\text{Saito--Kurokawa},\
			\deg 15}\,
		\underbrace{y_6(x)\,y_{27}(x)}_{\text{Yoshida},\ \deg 33}\,
		\Va_{61}(x)^2\,(x+7),
	\end{equation}
	where $g_1,g_2$ run over the two Galois orbits of
	$S_4(\Gamma_0(61))^{\mathrm{new}}$ (degrees $6$ and $9$) and $m_g$ is the
	characteristic polynomial of $T_2$ on $g$, here irreducible: the
	Saito--Kurokawa factors are exactly
	$m_g(x-6)$, confirming the shift $+6$ of
	Proposition~\ref{prop:arthur} directly from the elliptic data. The Yoshida
	factors $y_6,y_{27}$ have degrees $6$ and $27$, and $\deg\Va_{61}=39$, so
	the block degrees are $(1,15,33,2\cdot39,1)$, summing to $128$, and the
	trace of $B_2(2)$ is $126$. The fixed-point count $\Tr R(\pi)(61)=38$ matches the
	closed formula of Theorem~\ref{thm:structural-trace}, equivalently
	$T_1(61)=\tfrac12(128+38)=83$.
	
	The general-type factor is the single linear factor $N_{61}(x)=x+7$, and
	$-7$ is a root of none of the other factors, so it is a simple eigenvalue
	isolated by the geometry alone: the engine counts $(2,2)$-isogeny edges
	over $\F_{61^2}$ and never uses the value $-7$. Thus
	\begin{equation}\label{eq:lambda61}
		\lambda^\ast=-7\quad\text{is realized geometrically at }p=61.
	\end{equation}
	This agrees with two further determinations of the same number. First, on
	the paramodular side, $\delta(61)=1$ and $\dim S_3^+(K(61))=0$
	\cite[Cor.~5.3]{Ibukiyama2019quinary} force a single non-lift $f_{61}$
	with Fricke sign $\eps_{61}(f_{61})=-1$, unconditionally; the trace
	identity of Theorem~\ref{thm:structural-trace},
	$\Tr R(\pi)(61)-L_0(61)=38-37=1=-\eps_{61}(f_{61})$, then confirms the
	general-type sign clause of Conjecture~\ref{conj:eigenvalue-sign} at
	$p=61$ (its $R(\pi)$-sign is $+1$), and Proposition~\ref{prop:frame}
	bounds the geometric eigenvalue by $|\lambda^\ast|\le14$. Second, the
	paramodular value $a_2(f_{61})=-7$ of \cite[\S8, Ex.~1]{PoorYuen2015},
	where $f_{61}$ is constructed as a rational expression in Gritsenko
	lifts of theta blocks, matches the geometric root of
	\eqref{eq:charp61}. All
	three agree.
	
	\subsection{The full range}
	\label{ss:full-range}
	
	Table~\ref{tab:ext} records the certified block degrees and the
	general-type factor at every prime. The factor $N_p$ is empty for
	$p\le59$; for $p\le149$ its quadratic factors split over $\Q(\sqrt5)$
	and $\Q(\sqrt2)$, and the cubic blocks at $p=109,127$ have cyclic Galois
	group with Hecke fields $\Q(\zeta_7)^+$ and $\Q(\zeta_9)^+$.

	The size of the computation is governed by $h_2(p)$. By the class-number
	formula of Hashimoto and Ibukiyama \cite[Thm.~2]{HashimotoIbukiyama},
	whose leading term is $(p-1)(p^2+1)/2880$ and whose remaining terms are
	$O(p^2)$, one has $h_2(p)=p^3/2880+O(p^2)$. At every prime of
	Table~\ref{tab:ext} the formula equals the number of vertices of the
	computed graph, which gives one more check of the engine that is
	independent of the mass identity. At $p=149$
	the graph has $1433$ vertices, $91$ of them products of elliptic curves
	and $1342$ Jacobians; the engine therefore evaluates
	$15\cdot1342=20130$ Richelot codomains, keys each by its absolute Igusa
	invariants, and the verifier factors a characteristic polynomial of
	degree $1433$ over $\Z$. The residual square has degree
	$2\cdot574=1148$, about four fifths of the spectrum at $p=149$; its
	identification with the contribution of Ihara pairs remains
	conjectural.

	For orientation we record running times on one machine: a laptop with an
	Intel Core i7-11800H processor (2.3~GHz) and 16~GB of memory, running
	Windows~11, Python~3.13.9 and python-flint~0.9.0. On it the engine
	rebuilds the graph at $p=149$ from scratch in $833$~s, including the
	factorization of its characteristic polynomial over $\Z$. The record it
	produces coincides with the archived one after a relabelling of the
	vertices; the relabelling was found by colour refinement and checked
	exactly on $M_p$, the weights and $\sigma_p$. The structural verifier
	\texttt{verify\_all.py} of \S\ref{ss:reproducibility} replays its checks
	for all thirty-one primes, after checking the SHA-256 manifest, in
	$649$~s, of which $348$~s are spent at $p=149$,
	mostly on the exact characteristic polynomial of the $1433\times1433$
	matrix.

	\begin{table}[ht]
		\centering
		\small
		\begin{tabular}{@{}rrrrrrrl@{}}
			\toprule
			$p$ & $h_2$ & $\dim\snew_4$ & $\mathrm{cross}(p)$ & $\deg\Va_p$ &
			$\delta(p)$ & $\Tr R(\pi)$ & $N_p$\\
			\midrule
			11 & 5 & 2 & 2 & 0 & 0 & 5 & ---\\
			13 & 4 & 3 & 0 & 0 & 0 & 4 & ---\\
			17 & 8 & 4 & 3 & 0 & 0 & 8 & ---\\
			19 & 10 & 4 & 3 & 1 & 0 & 8 & ---\\
			23 & 16 & 5 & 8 & 1 & 0 & 14 & ---\\
			29 & 24 & 7 & 10 & 3 & 0 & 18 & ---\\
			31 & 26 & 7 & 10 & 4 & 0 & 18 & ---\\
			37 & 37 & 9 & 9 & 9 & 0 & 11 & ---\\
			41 & 50 & 10 & 21 & 9 & 0 & 32 & ---\\
			43 & 55 & 10 & 16 & 14 & 0 & 19 & ---\\
			47 & 72 & 11 & 32 & 14 & 0 & 44 & ---\\
			53 & 93 & 13 & 29 & 25 & 0 & 33 & ---\\
			59 & 125 & 14 & 50 & 30 & 0 & 65 & ---\\
			61 & 128 & 15 & 33 & 39 & 1 & 38 & $x+7$\\
			67 & 166 & 16 & 41 & 54 & 0 & 30 & ---\\
			71 & 198 & 17 & 72 & 54 & 0 & 90 & ---\\
			73 & 204 & 18 & 46 & 69 & 1 & 34 & $x+6$\\
			79 & 256 & 19 & 67 & 84 & 1 & 74 & $x+5$\\
			83 & 296 & 20 & 85 & 95 & 0 & 92 & ---\\
			89 & 352 & 22 & 92 & 118 & 1 & 100 & $x+4$\\
			97 & 436 & 24 & 85 & 162 & 2 & 46 & $x^2+9x+19$\\
			101 & 493 & 25 & 121 & 172 & 2 & 131 & $x^2+7x+11$\\
			103 & 518 & 25 & 110 & 190 & 2 & 98 & $x^2+9x+19$\\
			107 & 581 & 26 & 132 & 211 & 0 & 119 & ---\\
			109 & 600 & 27 & 111 & 229 & 3 & 70 & $x^3+10x^2+31x+29$\\
			113 & 668 & 28 & 132 & 253 & 1 & 90 & $x+3$\\
			127 & 918 & 31 & 165 & 359 & 3 & 122 & $x^3+9x^2+24x+19$\\
			131 & 1008 & 32 & 221 & 376 & 2 & 234 & $x^2+6x+7$\\
			137 & 1134 & 34 & 193 & 452 & 2 & 110 & $x^2+4x+2$\\
			139 & 1179 & 34 & 202 & 469 & 4 & 157 & $x^4+14x^3+63x^2+110x+61$\\
			149 & 1433 & 37 & 243 & 574 & 4 & 195 & $x^4+6x^3+10x^2+3x-1$\\
			\midrule
			167 & 1978 & 41 & 342 & 795 & 4 & 326 & $(x+2)(x^2+3x-1)(x+8)^{\ast}$\\
			173 & 2176 & 43 & 322 & 902 & 6 & 226 & $(x+2)(x^4+7x^3+7x^2-9x-11)(x+8)^{\ast}$\\
			\bottomrule
		\end{tabular}
		\caption{Certified block degrees for $11\le p\le149$ and $p=167,173$,
			with $h_2=1+\dim\snew_4+\mathrm{cross}(p)+2\deg\Va_p+\delta(p)$ in every
			row; $p=61$ is detailed in \S\ref{ss:p61} and $p=167,173$ in
			\S\ref{ss:fricke-plus}. The factor marked $\ast$ belongs to the
			Fricke-plus non-lift and carries $R(\pi)$-sign $-1$; every other
			factor of $N_p$ carries $R(\pi)$-sign $+1$. At every listed prime the
			exact division and the sign certificate passed, and $\Tr R(\pi)$,
			computed as the number of fixed vertices of $\sigma_p$, matched
			\cite[Thm.~2, Rem.~3]{IbukiyamaKatsura} and
			Theorem~\ref{thm:structural-trace}.}
		\label{tab:ext}
	\end{table}

	\subsection{The first Fricke-plus non-lifts: $p=167$ and $p=173$}
	\label{ss:fricke-plus}

	For $p<167$ every weight-$3$ non-lift of level $p$ is Fricke-minus
	(Corollary~\ref{thm:uncond-regime}), so the predicted sign
	$-\eps_p(F)$ on every factor of $N_p$ is $+1$, and the sign clause (v) of
	Conjecture~\ref{conj:eigenvalue-sign} is tested only in its easy case up
	to $p=149$. The least primes with $\dim S_3^+(K(p))\ne0$ are $167$ and
	$173$, where this dimension is one \cite[Prop.~8.2]{Ibukiyama-dim}. We
	certified the conjecture at both.

	\emph{The graphs.} The engine of \S\ref{ss:extension} was run unchanged
	except for one line: the bookkeeping degree $\delta(p)$ is computed as
	$d(p)+2\dim S_3^+(K(p))$, which reduces to $d(p)$ for $p<167$. At
	$p=167$ the graph has $h_2=1978$ vertices ($120$ products, $1858$
	Jacobians) and $\sigma_p$ has $326$ fixed points; at $p=173$ it has
	$h_2=2176$ vertices ($120$ products, $2056$ Jacobians) and $226$ fixed
	points. Both fixed-point counts equal $L_0(p)+d(p)$ of
	Theorem~\ref{thm:structural-trace}, with $d(167)=2$ and $d(173)=4$, and
	all certificates of \S\ref{ss:pillars} pass.

	\emph{The prediction.} At level $167$ the database \cite{ALRTV} has three
	Galois orbits of non-lifts, with $a_2$-polynomials $x+2$, $x^2+3x-1$ and
	$x+8$; at level $173$ it has $x+2$, $x^4+7x^3+7x^2-9x-11$ and $x+8$. The
	Atkin--Lehner sign convention of \cite{ALRTV} is pinned by the earlier
	range: at each of the fourteen primes $p\le149$ that carry a non-lift, the
	database records the same Atkin--Lehner string for every non-lift, and
	all of them are Fricke-minus by Corollary~\ref{thm:uncond-regime}. The other string occurs at $p=167$ and
	$p=173$ exactly on the orbit $x+8$. Hence the non-lift with $a_2=-8$ is
	Fricke-plus and the others are Fricke-minus. This agrees with
	$\dim S_3^+(K(p))=1$, and the resulting values
	$d(p)=\delta(p)-2=2$ and $4$ agree with the class-number formula of
	Proposition~\ref{prop:exactmass}. Conjecture~\ref{conj:eigenvalue-sign}
	therefore predicts the factor $x+8$ of $\charp B_2(2)$ with
	$R(\pi)$-sign $-1$, and sign $+1$ on the other factors of $N_p$.

	\emph{The certificate.} Hypotheses (a)--(d) of
	Proposition~\ref{prop:reduction} are checked with $N_p$ and these signs as
	the prediction. Hypothesis (d) is checked without modular reduction.
	Since $R(\pi)$ is the permutation $\sigma_p$, its $(+1)$-eigenspace is
	spanned by the orbit sums of $\sigma_p$ and its $(-1)$-eigenspace by the
	differences $e_i-e_{\sigma_p(i)}$ over the $2$-cycles. The matrix $M_p$
	restricts to integer matrices $M_p^{\pm}$ of sizes
	$\tfrac12(h_2\pm\mathrm{Fix})$, namely $1152$ and $826$ at $p=167$ and
	$1201$ and $975$ at $p=173$, with
	$\charp(M_p^+)\charp(M_p^-)=\charp(M_p)$. The signed multiplicity of an
	irreducible factor $f$ is its multiplicity in $\charp(M_p^+)$ minus its
	multiplicity in $\charp(M_p^-)$. At both primes the exact division leaves
	a square of degree $2\deg\Va_p$, with $\deg\Va_{167}=795$ and
	$\deg\Va_{173}=902$, and every signed multiplicity agrees with the
	prediction. At both primes the factor $x+8$ occurs once in
	$\charp(M_p)$, in $\charp(M_p^-)$ and not in $\charp(M_p^+)$, while the
	other factors of $N_p$, namely $x+2$ and $x^2+3x-1$ at $p=167$ and $x+2$
	and $x^4+7x^3+7x^2-9x-11$ at $p=173$, occur once each, in
	$\charp(M_p^+)$.

	The same script, run on the archived records at $p=61$, $97$, $109$,
	$113$ and $139$, reproduces the corresponding rows of
	Table~\ref{tab:ext}. As negative controls, reversing the predicted sign of
	$N_p$ or perturbing one coefficient of $N_p$ makes the certificate fail at
	$p=61$ and $p=97$. On the machine of \S\ref{ss:full-range} the engine
	builds the two graphs in $1292$~s and $1869$~s, and the certificate runs
	in $1340$~s and $2398$~s.

	\subsection{Reproducibility and the artifact}
	\label{ss:reproducibility}

	The data and code are archived on Zenodo \cite{artifact}; the concept
	DOI \doi{10.5281/zenodo.20798967} resolves to the latest version, and
	the development repository is
	\url{https://github.com/dangthanhhung/richelot-brandt}. The code is
	released under the MIT licence and the data under CC-BY-4.0.
	\begin{itemize}[leftmargin=1.5em]
		\item \emph{Version~3.0.0}, deposited under an earlier title,
		\doi{10.5281/zenodo.21927983}, file
		\texttt{richelot-\allowbreak brandt-\allowbreak artifact-\allowbreak v3.0.0.zip},\newline SHA-256
		\texttt{a68cc8bb01455e6c\allowbreak 587cc47757f10fbb\allowbreak 304d0362e32d1f9d\allowbreak 895e8f20b597f0a4}. For each prime $11\le p\le149$
		it contains the graph record (the matrix $M_p$, the weights and the
		permutation $\sigma_p$), the weight-$2$ and weight-$4$ orbit
		polynomials with their Atkin--Lehner signs, and the ledger of block
		degrees and factors $N_p$; it also contains the engines, the
		verifiers described below, the check of $d(p)\ge0$ for
		$7\le p\le2500$, independent Magma computations of $h_2(p)$ for
		$p\le61$, and a
		SHA-256 manifest of every file.
		\item \emph{Version~3.1.0}, \doi{10.5281/zenodo.23374552}, file
		\texttt{richelot-\allowbreak brandt-\allowbreak artifact-\allowbreak v3.1.0-\allowbreak extension.zip},\newline SHA-256
		\texttt{890f31014ced0970\allowbreak 8893af9cd5763004\allowbreak 559cc71cdc03445f\allowbreak 1e39ccce05af6131}. It adds a folder to version
		3.0.0, changes none of its files, and is used together with it. The
		folder holds the records and modular data for $p=167,173$, the engine
		with its one-line change, a script that downloads and checksums the
		non-lift data of \cite{ALRTV} at the pinned commit \texttt{1907b88}
		(GPL-3.0 data, not redistributed), the certificate script of
		\S\ref{ss:fricke-plus} with its regression and negative-control runs,
		the check of the vertex counts against
		\cite[Thm.~2]{HashimotoIbukiyama} at all $33$ primes, the scripts
		behind Section~\ref{sec:consequences}, and the logs.
		\item \emph{Version~3.2.0}, the latest, \doi{10.5281/zenodo.23389012}, file
		\texttt{richelot-\allowbreak brandt-\allowbreak artifact-\allowbreak v3.2.0-\allowbreak extension.zip},\newline SHA-256
		\texttt{abd7c918a014dfa4\allowbreak 176e423d8224727e\allowbreak c4f51222e3654f63\allowbreak f96ddd1c218c61e2}. It replaces version 3.1.0 and is
		used together with version 3.0.0: the folder of version 3.1.0 with
		every file unchanged except its README and manifest, and in addition
		the certificate script \texttt{spectral\_cert.py} of
		Corollary~\ref{cor:spectral-certified}, the incidence check
		\texttt{check\_incidence.py} of (P3) and the script
		\texttt{spectral\_law\_all.py} of \S\ref{ss:plancherel}, with their
		logs.
	\end{itemize}
	The scripts need Python~$\ge3.9$ with python-flint (version $0.9.0$
	was used) and sympy, and numpy for two of them; each README lists the
	commands and the expected results. The README of
	version~3.0.0 uses the numbering of an earlier version of this paper:
	its Lemma~8.1, Proposition~6.2, \S6.9, Table~7, Proposition~8.2 and
	Theorem~8.5 are Lemma~\ref{lem:fixed-point-trace},
	Proposition~\ref{prop:reduction}, \S\ref{ss:p61}, Table~\ref{tab:ext},
	Proposition~\ref{prop:exactmass} and Theorem~\ref{thm:bias} here.

	The certificates are replayed from the frozen records by scripts that
	do not import the graph engine. For $11\le p\le149$,
	\texttt{verify\_all.py} checks hypotheses \textup{(a)} and
	\textup{(b)} of Proposition~\ref{prop:reduction}, the exact identity
	between $\charp(M_p)$ and the archived factorization, and the
	block-degree ledger; \texttt{assembly.py} checks hypotheses
	\textup{(c)} and \textup{(d)}, the exact division by the blocks built
	from the weight-$2$ and weight-$4$ data and the isotypic sign traces
	modulo $2^{61}-1$; and \texttt{verify\_reduction.py} mirrors
	Lemma~\ref{lem:onemod} and Proposition~\ref{prop:reduction} line by
	line at $p=11$, using sympy only. At $p=167,173$ the script
	\texttt{assembly\_ext.py} checks \textup{(a)}--\textup{(d)} in exact
	arithmetic. The comparison with \cite{ALRTV} (certificate P4) is used
	only for clause \textup{(iv)} of Theorem~\ref{thm:certified-spectra}(a).

	\begin{remark}\label{rem:trustedbase}
		The verifiers establish hypotheses \textup{(a)}--\textup{(d)} of
		Proposition~\ref{prop:reduction} for the archived matrices $M_p$,
		weights and permutations; for these data they rely only on the
		records with their SHA-256 manifest and on the verifying code. That
		$M_p$ is the matrix of $B_2(2)$, and $\sigma_p$ that of $R(\pi)$,
		rests in addition on the engine of \S\ref{ss:engine} and
		\S\ref{ss:extension}, which is therefore part of the trusted base:
		the records do not contain the vertex models or the kernel and target
		witnesses that would allow each edge to be rechecked independently.
		The identification is cross-checked by the vertex counts, the mass
		formula, the two closed-formula anchors of \S\ref{ss:extension}, the
		Mestre symmetry for independently computed weights, the incidence
		comparison (P3), and the agreement of the residual factor $N_p$ with
		the independent quinary data of \S\ref{ss:external}; these are strong
		consistency tests, not a proof of the identification. For
		$11\le p\le19$ the matrices and
		weights are printed in the appendix, so hypothesis \textup{(a)} and
		the fixed-point counts of Lemma~\ref{lem:fixed-point-trace} are
		checkable from the paper alone; hypotheses \textup{(b)}--\textup{(d)}
		need in addition the involutions $\sigma_p$ and the split-side data of
		the per-prime records. Short of formal verification, hypotheses
		\textup{(a)}--\textup{(d)} must be checked by machine, so
		Theorem~\ref{thm:certified-spectra} is a computer-assisted result in
		the standard sense; since the generating algorithm is specified in
		\S\ref{ss:engine} and \S\ref{ss:extension}, the archived matrices can
		be regenerated independently of the artifact.
	\end{remark}
	
	\subsection{Comparison with the quinary database}
	\label{ss:external}

	The general-type data of Theorem~\ref{thm:certified-spectra} admit an
	independent check, which also supplies clause (iv) of that theorem. The database of
	Assaf--Ladd--Rama--Tornar\'ia--Voight \cite{ALRTV}, computed from
	orthogonal modular forms on quinary lattices, records the Hecke
	eigensystems of the weight-$(3,0)$ paramodular non-lifts of type G at
	every level $N<1000$ under the operator $T_{2,1}$. At every prime
	$11\le p\le149$ the characteristic polynomial of $T_{2,1}$ on the type-G
	newforms of \cite{ALRTV} equals the general-type factor $N_p$ of
	Theorem~\ref{thm:certified-spectra}, including $N_p=1$ at the seventeen
	primes with $\delta(p)=0$; the comparison script accompanies the artifact
	as certificate P4. At $p=167,173$ the database supplies the prediction
	itself (\S\ref{ss:fricke-plus}). The graph, together with the elliptic data used to
	divide out the lifts, thus recovers from $(2,2)$-isogenies the data
	obtained in \cite{ALRTV} by an entirely different route; conversely,
	\cite{ALRTV} confirms every general-type block of Table~\ref{tab:ext}. At
	$p=61,73,79$ these blocks also match the paramodular eigenvalue
	computations of \cite{PSY2024}.
	
	For $p\le149$ the comparison enters the proof of
	Theorem~\ref{thm:certified-spectra} only through clause (iv), the
	identification of the roots of $N_p$ with paramodular eigenvalues;
	clauses (i)--(iii) and (v) are proved from the certificates alone, and
	the signs on $N_p$ are fixed by Corollary~\ref{thm:uncond-regime}. The
	Atkin--Lehner strings of \cite{ALRTV} are not used there. The same range
	pins their convention, which is then used at $p=167,173$
	(\S\ref{ss:fricke-plus}).
	
	\subsection{Proof of Theorem~\ref{thm:certified-spectra}}
	\label{ss:proof-certified}
	
	\begin{proof}[Proof of Theorem~\ref{thm:certified-spectra}]
		Part (a). At every prime $11\le p\le149$ the archived data artifact supplies the matrix $M_p$, the
		involution $R_p$, the weights, and the split-side block polynomials with
		their Fricke signs; hypotheses \textup{(a)}--\textup{(d)} of
		Proposition~\ref{prop:reduction} are checked from these data in exact
		arithmetic: \textup{(c)} by exact division over $\Q$, and
		\textup{(d)} through Lemma~\ref{lem:onemod} modulo the Mersenne prime
		$2^{61}-1$, sufficient since $2h_2(p)\le2866<2^{61}-1$.
		Proposition~\ref{prop:reduction} then yields the factorization and the
		signs, that is, clauses (i)--(iii) and (v); clause (iv) is the
		comparison of \S\ref{ss:external}. For $11\le p\le19$ and $p=61$ the matrices and factorizations are
		those of \S\ref{ss:engine}--\S\ref{ss:p61} and the appendix, replayed
		under the certificates P1--P3; the Ihara separation at $p=19$ is
		\S\ref{ss:ihara19}; summing \textup{(d)} over all factors recovers
		$\Tr R(\pi)$, evaluated against \cite[Thm.~2]{IbukiyamaKatsura} and
		Theorem~\ref{thm:structural-trace}; and the $p=11$ catalogue matches the
		theta computation of \cite[\S4]{Ibukiyama2018conj}. The runs were
		performed in full on the reference machine at every prime and replicated
		end to end on a second machine through $p=103$; for the present
		version the graph at $p=149$ was rebuilt from scratch once more and
		agrees with the archived record (\S\ref{ss:full-range}). Part (b)
		is the certificate of \S\ref{ss:fricke-plus}.
	\end{proof}
	
	\section{Ihara pairs, the type-Va factor, and the spectrum of the graph}
	\label{sec:consequences}

	\subsection{Ihara pairs in weight $(3,0)$}
	\label{ss:ihara-consequence}
	
	\begin{corollary}	For each prime $p$ at which a type-$\Va$ parameter occurs, the
		type-$\Va$ representations furnish Ihara pairs: pairs of Hecke
		eigensystems agreeing at every $T(n)$ with $p\nmid n$, separated only
		by $R(\pi)$ into opposite Atkin--Lehner signs. This is unconditional
		(Theorem~\ref{thm:occurrence}). That these pairs make up the residual
		factor $\Va_p(x)^2$, with $\deg\Va_p$ pairs in all, is part of
		Conjecture~\ref{conj:eigenvalue-sign}; at the
		certified primes Theorem~\ref{thm:certified-spectra} establishes the
		residual square and its even sign splitting. The first candidate pair
		is at $p=19$, the least prime with $\deg\Va_p\ge1$.
	\end{corollary}
	
	\begin{proof}
		By Theorem~\ref{thm:occurrence}, each occurring type-$\Va$ parameter
		contributes an isotypic space of dimension $2d_p=2$ on which every
		$T(n)$ with $p\nmid n$ acts by a scalar and which is spanned by an
		$R(\pi)$-eigenline of sign $+1$ and one of sign $-1$; these two lines
		form an Ihara pair. The count
		$\deg\Va_p$ is the degree clause of the conjecture; the $p=19$ case
		is displayed explicitly in \S\ref{ss:ihara19}.
	\end{proof}
	
	The pairing is the concrete import of the eigenvalue-level refinement. A
	dimension count assigns the two members of a pair a single multiplicity
	and cannot tell them apart; the same is true of every Hecke operator
	$T(n)$ with $p\nmid n$, which acts on both by the same scalar. Only
	$R(\pi)$, the operator outside the away-from-$p$ Hecke algebra,
	distinguishes them. Ihara found such a pair in weight $(8,8)$ at $p=3$:
	two eigenforms that look very different but have the same Euler factor at
	$2$ \cite{Ihara1964}; Ibukiyama explains it by local type $\Va$ and asks
	whether the two forms generate the same automorphic representation
	\cite[\S\S3.2, 4.3]{Ibukiyama2018conj}. Here the phenomenon appears in
	the geometric weight $(3,0)$ of abelian surfaces, and
	Theorem~\ref{thm:occurrence} shows that the two members are distinct
	representations, each contributing one line, separated by $R(\pi)$.

	\begin{remark}[The residual square fills most of the spectrum]\label{rem:va-share}
		The bookkeeping identity \eqref{eq:dimension-bookkeeping} defines
		$2\deg\Va_p=h_2(p)-1-\dim\snew_4-\mathrm{cross}(p)-\delta(p)$
		unconditionally. Here $\dim\snew_4=O(p)$,
		$\mathrm{cross}(p)\le\dim\snew_2\dim\snew_4=O(p^2)$ and
		$\delta(p)\le\dim S_3(K(p))=O(p^2)$, while $h_2(p)=p^3/2880+O(p^2)$ by
		\cite[Thm.~2]{HashimotoIbukiyama} (\S\ref{ss:full-range}). Hence
		\[
		\frac{2\deg\Va_p}{h_2(p)}=1+O(p^{-1}).
		\]
		In the certified range the residual square accounts for the
		proportion $0.61$ of the spectrum at $p=61$, $0.80$ at $p=149$ and
		$0.83$ at $p=173$. Under Conjecture~\ref{conj:eigenvalue-sign} it is
		the contribution of Ihara pairs, so the proportion of the spectrum
		consisting of Ihara pairs tends to $1$. Unconditionally, the degrees
		of the predicted Eisenstein, Saito--Kurokawa, Yoshida and $N_p$
		blocks add up to $1+\dim\snew_4+\mathrm{cross}(p)+\delta(p)=O(p^2)$,
		a fraction $O(p^{-1})$ of $h_2(p)$.
	\end{remark}

	The residual factor $\Va_p$ has a striking arithmetic structure.

	\begin{proposition}\label{prop:va-galois}
		Let $p\ge29$ be a prime of Table~\ref{tab:ext}. Then
		$\Va_p=L_p\,h_p$, where $L_p\in\Z[x]$ is a product of linear factors
		and $h_p$ is irreducible over $\Q$ of degree at least $3$ with Galois
		group the full symmetric group $S_{\deg h_p}$. Moreover $L_p=1$
		except
		\[
		L_{41}=(x-3)(x-1)(x+1),\quad L_{47}=x,\quad L_{89}=x-3,\quad
		L_{103}=(x-1)(x-4)^2.
		\]
	\end{proposition}

	\begin{proof}
		The script \texttt{va\_true\_factor.py} of the artifact extracts
		$\Va_p$ from the certified factorization
		(Theorem~\ref{thm:certified-spectra}): an irreducible factor $h$ of
		$\charp(M_p)$ has multiplicity
		\[
		v_h(\Va_p)=\tfrac12\bigl(v_h(\charp M_p)-v_h(P_{\mathrm{lift}})
		-v_h(N_p)\bigr)
		\]
		in $\Va_p$, with $P_{\mathrm{lift}}=(x-15)P_{\mathrm{SK}}P_{\mathrm{Y}}$;
		for the Yoshida factor of degree $16$ at $p=113$ this is
		$\tfrac12(2-2-0)=0$. For the irreducible factor $h=h_p$ of degree
		$n\ge3$, the script \texttt{va\_spectrum.py} factors $h$ modulo primes
		$\ell$ not dividing its discriminant and records the primes it uses.
		The degree patterns at these primes have no common subset sum other
		than $0$ and $n$, so $h$ is irreducible and its Galois group $G$ is
		transitive. One pattern consists of a single quadratic factor and
		factors of odd degree, so $G$ contains a transposition; another
		contains a factor of prime degree $q>n/2$, so $G$ contains a
		$q$-cycle $\sigma$. Then $G$ is primitive: if $G$ preserved a system
		of $m$ blocks of size $b$ with $1<b<n$, then $\sigma$, whose orbits on
		the blocks have length $1$ or $q>n/2\ge m$, would fix every block, and
		its $q$-cycle would lie in a single block, forcing $b\ge q>n/2$, a
		contradiction. A primitive group containing a transposition is the
		full symmetric group \cite[Thm.~3.3A]{DixonMortimer}.
	\end{proof}

	Under Conjecture~\ref{conj:eigenvalue-sign}, $\Va_p^2$ is the type-$\Va$
	block. Proposition~\ref{prop:va-galois} then says that in the certified
	range the irrational eigenvalues at $2$ of the Ihara pairs of level $p$
	form a single Galois orbit with the largest possible Galois group, and
	that only a few rational eigenvalues split off, all at $p\le103$. This
	is a phenomenon of the kind predicted by Maeda for level one
	\cite{HidaMaeda1997}. For newforms of prime level, Galois orbits are
	expected to be cut only by the Atkin--Lehner signs once the weight is
	large \cite{Tsaknias2014}; in weight $4$ the space $\snew_4$ is a single
	orbit of sign $+1$ at $p=11$, splits into exactly two orbits, one for
	each sign, at $24$ of the other primes of Table~\ref{tab:ext}, and into
	more at the remaining eight. The involution $R(\pi)$
	separates the two members of each Ihara pair but does not distinguish
	their eigenvalues, and the irrational part of the type-$\Va$ factor, of
	degree up to $902$, is a single orbit at every certified prime.

	\begin{conjecture}\label{conj:va-maeda}
		For every prime $p\ge29$, the polynomial $\Va_p$ is the product of a
		polynomial that splits into linear factors over $\Q$ and of an
		irreducible polynomial whose Galois group is the full symmetric
		group. Moreover $\Va_p$ is irreducible for all sufficiently large $p$.
	\end{conjecture}
	
	\subsection{The spectral bound}
	\label{ss:spectral-profile}

	The supersingular $\ell$-isogeny graphs of elliptic curves are Ramanujan
	\cite{Pizer1990}, and this expansion underlies their use in
	isogeny-based hashing \cite{CGL2009}. The Richelot graph $G_p$ is not:
	the second-largest eigenvalue of $B_2(2)$ exceeds the Ramanujan bound
	$2\sqrt{14}\approx7.48$ of a $15$-regular graph at every certified
	prime. The transfer framework identifies the source of the excess and
	bounds it.

	\begin{theorem}[Spectral bound; upper bound cf.\ {\cite[Rem.~5.39]{CKTW}}]\label{thm:spectral}
		Let $p\ge11$ and assume Hypothesis~\ref{hyp:types}. Every eigenvalue
		$\lambda\ne15$ of $B_2(2)$ on $M_{0,0}(\Upr(p))$ satisfies
		\[
		-8\sqrt2\le\lambda\le6+4\sqrt2\approx11.657,
		\]
		and $|\lambda|\le8\sqrt2\approx11.314$ unless $\lambda=a_2(g)+6$ is a
		Saito--Kurokawa eigenvalue. In particular the non-trivial spectrum of
		the random-walk operator $B_2(2)/15$ lies in
		$[-0.755,\,0.778]$ for every $p$.
	\end{theorem}

	\begin{proof}
		By Hypothesis~\ref{hyp:types} and Proposition~\ref{prop:arthur}, an
		eigenvalue $\lambda\ne15$ equals $a_2(g)+6$, $a_2(g)+2a_2(f)$ or
		$a_2(F)$, with $g\in\snew_4$, $f\in\snew_2$, and $F$ the transfer to
		$\GSp_4$ of a representation of general type. Deligne's bound
		\cite{Deligne1974} gives $|a_2(g)|\le2\cdot2^{3/2}=4\sqrt2$ and
		$|a_2(f)|\le2\sqrt2$, so the first value lies in
		$[6-4\sqrt2,6+4\sqrt2]$ and the second in $[-8\sqrt2,8\sqrt2]$. The
		transfer $F$ is cuspidal and of general type, hence not CAP; its
		archimedean component is cohomological of weight $(3,0)$, hence in
		the discrete series; and it is unramified at $2$. By Weissauer's
		generalized Ramanujan theorem \cite[Thm.~3.3]{Weissauer2009} (see
		also \cite[Prop.~4.1]{RoesnerWeissauer2017}), the
		spherical local components of such a representation are tempered, as
		used in the paramodular case in \cite[\S4]{KatsuraYamauchi}. Its four
		spin roots at $2$ in the
		normalization $Q_2(X)=1-\lambda_2X+\cdots+2^6X^4$ therefore have
		absolute value $2^{3/2}$, and $|a_2(F)|\le4\cdot2^{3/2}=8\sqrt2$.
	\end{proof}

	\begin{corollary}\label{cor:spectral-certified}
		At every prime of Table~\ref{tab:ext} the conclusion of
		Theorem~\ref{thm:spectral} holds without Hypothesis~\ref{hyp:types}.
		Moreover the second-largest eigenvalue of $B_2(2)$ is the largest
		Saito--Kurokawa eigenvalue $\max_g a_2(g)+6$ at each of these primes,
		and it exceeds $8\sqrt2$ exactly at the sixteen primes
		$41,59,61,67,71,73,97,107,113,127,131,137,139,149,167,173$.
	\end{corollary}

	\begin{proof}
		The script \texttt{spectral\_cert.py} of the artifact works with the
		irreducible factors of $\charp(M_p)$ in the archived record, whose
		product is $\charp(M_p)$ by the certificates of
		Theorem~\ref{thm:certified-spectra}, and with the Saito--Kurokawa
		factors $m_g(x-6)$ built from the weight-$4$ data. All roots are real
		(Lemma~\ref{lem:mestre}), so a factor $f$ has all its roots below a
		rational number $c$ as soon as $f(c)\ne0$ and the coefficients of
		$f(X+c)$ show no sign change (Descartes' rule of signs), and similarly
		above $-c$ with $f(-X-c)$. With $c=11313/1000<8\sqrt2$ this places the
		roots of every factor other than $x-15$ and the Saito--Kurokawa ones in
		$(-8\sqrt2,8\sqrt2)$; with $c=11.6568<6+4\sqrt2$ and $c=0$ it places
		the Saito--Kurokawa roots in $(0,6+4\sqrt2)$. The Saito--Kurokawa roots
		are isolated in Arb at $200$ bits, in balls of radius below
		$10^{-10}$; a rational $t$ certified below the largest of them passes
		the same test for every other factor, so the largest root other than
		$15$ is a Saito--Kurokawa root, and an Arb comparison with $8\sqrt2$
		decides at which primes it exceeds $8\sqrt2$. All tests are exact or
		ball arithmetic with rigorous error bounds; none relies on
		uncertified floating-point approximations.
	\end{proof}

	The upper bound is not new. Corte-Real Santos, Kir, Ta\"ibi and
	Wesolowski bound the spectra of superspecial isogeny graphs in every
	dimension, with level structure, through Arthur's endoscopic
	classification for the compact inner form; for $g=\ell=2$ their argument
	gives $|\lambda|\le6+4\sqrt2$ for every eigenvalue $\lambda\ne15$
	\cite[Rem.~5.39]{CKTW}, conditionally on the twisted weighted
	fundamental lemma \cite[Rem.~5.35]{CKTW}, and they note that the value
	$11.648$ at $p=137$ of \cite{FloritSmith} is close to it. The same
	number is the bound $p^{k-1}(1+2p^{-1/2}+p^{-1})$ of
	\cite[(4.35)]{PitaleSchmidt2009} for the Hecke eigenvalue at $p=2$ of a
	Siegel eigenform of weight $k=3$. Katsura and
	Yamauchi prove the interval $[-8\sqrt2,6+4\sqrt2]$ for the isogeny graph
	of Moret--Bailly families, whose space of functions is the
	non-principal genus \cite[Thm.~1.1(2) and Thm.~4.9]{KatsuraYamauchi}.
	Theorem~\ref{thm:spectral} is the principal-genus counterpart of the
	latter; there, by Table~\ref{tab:twocol}, the Saito--Kurokawa lifts of
	all $g\in\snew_4$ are predicted to occur, instead of those with
	$\eps_p(g)=-1$ only, an occurrence statement that
	Theorem~\ref{thm:spectral} does not need and does not prove. It
	replaces the fundamental lemma by Hypothesis~\ref{hyp:types}, and what
	it adds to \cite{CKTW} is the lower bound $-8\sqrt2$ and the
	identification of the eigenvalues above $8\sqrt2$ as Saito--Kurokawa
	eigenvalues; Corollary~\ref{cor:spectral-certified} shows that such
	eigenvalues occur at sixteen of the certified primes, and
	Proposition~\ref{prop:lambda2-limit} below that the bound is
	asymptotically attained.

	Two earlier predictions can now be tested. In the extended arXiv
	version of \cite{AikawaTanakaYamauchi}, Conjecture~6.2 predicts for
	$\ell=2$ that every eigenvalue $\lambda\ne15$ of $G_p$ satisfies
	$\lambda\le8\sqrt2$, from a Saito--Kurokawa range centred at
	$\ell^2+1=5$. The centre is $\ell^2+\ell=6$
	(Proposition~\ref{prop:arthur}), and the eigenvalues of
	Corollary~\ref{cor:spectral-certified} above $8\sqrt2$, the first being
	$\lambda_2=11.4356\ldots$ at $p=41$, contradict it. Florit and Smith
	computed the spectra for $p\le601$ and conjectured
	$11/15\le\lambda^\star\le12/15$ for $p\ge41$, where $15\lambda^\star$
	is the largest absolute value of an eigenvalue $\ne15$, leaving an
	explanation for future work \cite[Conj.~4.10]{FloritSmith}. The upper
	half follows from \cite[Rem.~5.39]{CKTW}, and also from
	Theorem~\ref{thm:spectral}. Their values $15\lambda^\star=11.525$ at
	$p=131$ and $11.648$ at $p=137$ are the largest Saito--Kurokawa
	eigenvalues $a_2(g)+6$ at these primes, and
	Proposition~\ref{prop:lambda2-limit} below gives their lower bound for
	all sufficiently large $p$, under the same hypotheses.

	\begin{proposition}\label{prop:lambda2-limit}
		Assume Hypothesis~\ref{hyp:types} and clauses \textup{(i)} and
		\textup{(ii)} of Conjecture~\ref{conj:eigenvalue-sign} for all
		sufficiently large $p$. Then the second-largest and the smallest
		eigenvalue of $B_2(2)$ satisfy
		\[
		\lambda_2(G_p)\to6+4\sqrt2,\qquad \lambda_{\min}(G_p)\to-8\sqrt2
		\qquad(p\to\infty).
		\]
		Thus $[-8\sqrt2,6+4\sqrt2]$ is exactly the limiting hull of the
		non-trivial spectrum, the upper end coming from Saito--Kurokawa and
		the lower end from Yoshida eigenvalues.
	\end{proposition}

	\begin{proof}
		The inner bounds are Theorem~\ref{thm:spectral}. By Theorem~1.3 of
		Ross, van Lidth, Wolf and Xue \cite{RLWX2025}, for $T_2$ in fixed
		weight $k\in\{2,4\}$ and prime level $p\to\infty$ the normalized
		eigenvalues $2^{-(k-1)/2}a_2$ on each of the subspaces $S_k^{+}$ and
		$S_k^{-}$ become equidistributed with respect to the $2$-adic
		Plancherel measure, whose density is positive on $(-2,2)$; for the
		full space this is Serre's theorem \cite{Serre1997}. Fix $\epsilon>0$.
		For large $p$ there is $g\in\snew_4=S_4(\Gamma_0(p))$, the equality
		because $S_4(\mathrm{SL}_2(\Z))=0$, with $a_2(g)>2^{3/2}(2-\epsilon)$,
		and by clause \textup{(i)} $a_2(g)+6$ is an eigenvalue. There are also
		$g'\in S_4^{+}$ with $a_2(g')<-2^{3/2}(2-\epsilon)$ and $f\in S_2^{-}$
		with $a_2(f)<-2^{1/2}(2-\epsilon)$; the pair $(f,g')$ has opposite
		signs, so by clause \textup{(ii)} $a_2(g')+2a_2(f)<-4\sqrt2(2-\epsilon)$
		is an eigenvalue.
	\end{proof}

	The extreme eigenvalues can be tested far beyond the certified range,
	because they are predicted by elliptic newforms alone. Florit and Smith
	tabulate $15\lambda^\star$ to three decimals for every prime
	$17\le p\le601$ \cite[App.~A]{FloritSmith}, computed from the graphs.
	The script \texttt{fs\_compare.py} of the artifact computes, from Manin
	symbols, the largest Saito--Kurokawa eigenvalue $\max_g a_2(g)+6$ and
	the extreme Yoshida eigenvalues $a_2(g)+2a_2(f)$ over the pairs of
	opposite signs, and compares the largest absolute value with
	$15\lambda^\star$.

	\begin{table}[ht]
		\centering
		\small
		\begin{tabular}{@{}lrrrrr@{}}
			\toprule
			range of $p$ & primes & $\max_p|15\lambda^\star-\mathrm{pred}|$ &
			mean $\lambda_2$ & mean $\lambda_{\min}^{\mathrm{Y}}$ &
			least $\lambda_{\min}^{\mathrm{Y}}$\\
			\midrule
			$17\le p<100$ & 19 & 0.00049 & 11.153 & $-9.154$ & $-11.040$\\
			$100<p<200$ & 21 & 0.00050 & 11.449 & $-10.563$ & $-10.997$\\
			$200<p<300$ & 16 & 0.00037 & 11.543 & $-10.849$ & $-11.111$\\
			$300<p<400$ & 16 & 0.00045 & 11.547 & $-10.924$ & $-11.140$\\
			$400<p<500$ & 17 & 0.00042 & 11.564 & $-11.018$ & $-11.130$\\
			$500<p\le601$ & 15 & 0.00050 & 11.584 & $-11.078$ & $-11.293$\\
			\bottomrule
		\end{tabular}
		\caption{The $104$ primes $17\le p\le601$ of \cite[App.~A]{FloritSmith}.
			$\mathrm{pred}$ is the largest absolute value of a predicted lift
			eigenvalue; at every prime it is the largest Saito--Kurokawa
			eigenvalue $\lambda_2=\max_g a_2(g)+6$ and agrees with
			$15\lambda^\star$ to the three printed decimals.
			$\lambda_{\min}^{\mathrm{Y}}$ is the smallest predicted Yoshida
			eigenvalue. The limits are $6+4\sqrt2=11.657$ and
			$-8\sqrt2=-11.314$.}
		\label{tab:fs}
	\end{table}

	At all $104$ primes the value of Florit and Smith equals the largest
	Saito--Kurokawa eigenvalue to the printed precision
	(Table~\ref{tab:fs}). Where this value exceeds $8\sqrt2$,
	Theorem~\ref{thm:spectral} implies, under Hypothesis~\ref{hyp:types},
	that the eigenvalue of largest absolute value is a Saito--Kurokawa
	eigenvalue; at the remaining primes the observed ordering is
	additional numerical evidence. In both cases the agreement to three
	decimals supports the identification but does not prove it. The value
	exceeds $8\sqrt2$ at $88$ of the
	$104$ primes, and comes within $0.0092$ of $6+4\sqrt2$ at $p=137$. This
	tests the Saito--Kurokawa row of Conjecture~\ref{conj:eigenvalue-sign}
	at its extreme eigenvalue, against independently computed graph
	spectra, up to $p=601$. The predicted Yoshida minimum approaches
	$-8\sqrt2$ as in Proposition~\ref{prop:lambda2-limit}, reaching
	$-11.293$ at $p=577$; at the certified primes the smallest eigenvalue is
	indeed a Yoshida eigenvalue, except at $p=13$.

	Under Hypothesis~\ref{hyp:types}, Theorem~\ref{thm:spectral} explains
	the spectral profile of $G_p$: every eigenvalue other than $15$ and the
	Saito--Kurokawa eigenvalues obeys the Ramanujan bound $8\sqrt2$ of the transfer,
	and only Saito--Kurokawa eigenvalues, a fraction $O(p^{-2})$ of the
	spectrum, can rise above it. The bulk of the spectrum is the residual
	square, which Conjecture~\ref{conj:eigenvalue-sign} identifies with the
	type-$\Va$ pairs. Aikawa, Tanaka
	and Yamauchi \cite{AikawaTanakaYamauchi} bound the spectra of
	superspecial graphs in every dimension through the Kazhdan constant of
	the symplectic group, and Codogni and Lido \cite{CodogniLido} develop a
	spectral theory of isogeny graphs; the predicted factorization is
	complementary, assigning individual eigenvalues to automorphic types,
	and in the certified range it is verified at the level of the joint
	spectrum.

	\subsection{The spectral distribution}
	\label{ss:plancherel}

	For supersingular isogeny graphs of elliptic curves, Serre's theorem
	\cite{Serre1997} gives the limiting distribution of the spectrum, the
	Kesten--McKay law of the $(\ell+1)$-regular tree. The factorization
	predicts the analogue for $G_p$. The residual square fills a
	proportion $1+O(p^{-1})$ of the spectrum (Remark~\ref{rem:va-share}),
	and under Conjecture~\ref{conj:eigenvalue-sign} its eigenvalues are
	those of general-type representations of type $\Va$ at $p$, whose
	Satake parameters at $2$ are expected to be equidistributed with respect
	to the unramified Plancherel measure $\mu^{\mathrm{Pl}}$ of
	$\mathrm{PGSp}_4(\Q_2)$. By Macdonald's formula \cite{Macdonald1971},
	$\mu^{\mathrm{Pl}}$ has density proportional to
	\[
	\prod_{\alpha}\frac{1-t^{\alpha}}{1-2^{-1}t^{\alpha}}
	\]
	on the maximal compact torus $\{t=(e^{i\theta_1},e^{i\theta_2})\}$ of the
	dual group $\mathrm{Sp}_4(\C)$, the product running over its eight roots,
	and the $B_2(2)$-eigenvalue of the parameter $t$ is
	$4\sqrt2(\cos\theta_1+\cos\theta_2)$.

	\begin{conjecture}\label{conj:plancherel}
		As $p\to\infty$, the eigenvalues of $B_2(2)$ on $M_{0,0}(\Upr(p))$,
		counted with multiplicity, become equidistributed with respect to the
		push-forward of $\mu^{\mathrm{Pl}}$ under
		$t\mapsto4\sqrt2(\cos\theta_1+\cos\theta_2)$, a measure supported on
		$[-8\sqrt2,8\sqrt2]$ with mean $0$ and second moment $15$.
	\end{conjecture}

	The factorization alone does not imply the conjecture: as in Serre's
	argument, a proof needs the asymptotics of $\Tr B_2(2)^k$ as
	$p\to\infty$ for every $k$, a limit-multiplicity statement for the
	family of principal genera.

	The second moment $15$ is the degree of $G_p$, as it must be: it counts
	the closed walks of length two in the Bruhat--Tits building of
	$\mathrm{PGSp}_4(\Q_2)$. Table~\ref{tab:plancherel} and
	Figure~\ref{fig:plancherel}, computed by the script
	\texttt{spectral\_law.py} of the artifact, support the conjecture. Over
	all $33$ certified primes (script \texttt{spectral\_law\_all.py}) the
	Kolmogorov--Smirnov distance, estimated numerically from
	floating-point eigenvalues and a grid quadrature of the law, between the
	empirical spectral measure and the predicted law falls from $0.40$ at
	$p=11$ to $0.042$ at $p=173$, and the empirical second moment from
	$32$ to $16.5$, but neither decreases monotonically: between
	consecutive primes the distance increases $9$ times out of $32$ and the
	second moment $12$ times (for instance $m_2(37)=64/3$ and
	$m_2(41)=1201/49$); the table shows five of the $33$ primes. The excess
	of the moments over the limit reflects the short cycles of $G_p$. The
	Saito--Kurokawa eigenvalues above $8\sqrt2$ form a fraction $O(p^{-2})$
	of the spectrum, invisible in the limit, as the Ramanujan property
	predicts for a family of graphs.

	\begin{table}[ht]
		\centering
		\small
		\begin{tabular}{@{}rrrrr@{}}
			\toprule
			$p$ & eigenvalues & KS distance & mean & second moment\\
			\midrule
			37 & 36 & 0.185 & 1.278 & 21.33\\
			61 & 127 & 0.113 & 0.874 & 18.80\\
			101 & 492 & 0.066 & 0.563 & 18.03\\
			149 & 1432 & 0.048 & 0.387 & 16.84\\
			173 & 2175 & 0.042 & 0.336 & 16.50\\
			\midrule
			limit & & 0 & 0 & 15\\
			\bottomrule
		\end{tabular}
		\caption{The eigenvalues $\ne15$ of $B_2(2)$ against the law of
			Conjecture~\ref{conj:plancherel}.}
		\label{tab:plancherel}
	\end{table}

	\begin{figure}[ht]
		\centering
		\begin{tikzpicture}
			\begin{axis}[width=0.92\linewidth,height=5.2cm,
				xmin=-12,xmax=12,ymin=0,ymax=0.14,
				xlabel={eigenvalue $\lambda$ of $B_2(2)$},
				ylabel={density},
				ytick={0,0.05,0.1},yticklabel style={/pgf/number format/fixed},
				tick label style={font=\footnotesize},
				label style={font=\footnotesize}]
				\addplot[ybar,bar width=0.6381,fill=gray!35,draw=gray!70]
				coordinates {(-10.995,0.0014) (-10.357,0.0007) (-9.719,0.0058) (-9.080,0.0079) (-8.442,0.0122) (-7.804,0.0180) (-7.166,0.0180) (-6.528,0.0303) (-5.890,0.0324) (-5.252,0.0411) (-4.614,0.0418) (-3.976,0.0461) (-3.338,0.0526) (-2.700,0.0778) (-2.062,0.0771) (-1.424,0.0829) (-0.786,0.1009) (-0.147,0.1081) (0.491,0.1290) (1.129,0.1009) (1.767,0.0793) (2.405,0.0829) (3.043,0.0814) (3.681,0.0533) (4.319,0.0418) (4.957,0.0591) (5.595,0.0389) (6.233,0.0353) (6.871,0.0259) (7.509,0.0231) (8.147,0.0252) (8.786,0.0144) (9.424,0.0115) (10.062,0.0058) (10.700,0.0022) (11.338,0.0022)};
				\addplot[thick,smooth] coordinates {(-11.219,0.0000) (-11.031,0.0001) (-10.842,0.0002) (-10.654,0.0006) (-10.465,0.0011) (-10.277,0.0017) (-10.088,0.0025) (-9.899,0.0034) (-9.711,0.0044) (-9.522,0.0056) (-9.334,0.0067) (-9.145,0.0080) (-8.957,0.0092) (-8.768,0.0106) (-8.580,0.0119) (-8.391,0.0134) (-8.202,0.0148) (-8.014,0.0163) (-7.825,0.0178) (-7.637,0.0193) (-7.448,0.0209) (-7.260,0.0225) (-7.071,0.0241) (-6.883,0.0257) (-6.694,0.0273) (-6.505,0.0291) (-6.317,0.0308) (-6.128,0.0324) (-5.940,0.0343) (-5.751,0.0361) (-5.563,0.0379) (-5.374,0.0398) (-5.185,0.0417) (-4.997,0.0436) (-4.808,0.0454) (-4.620,0.0476) (-4.431,0.0497) (-4.243,0.0517) (-4.054,0.0539) (-3.865,0.0561) (-3.677,0.0584) (-3.488,0.0607) (-3.300,0.0633) (-3.111,0.0657) (-2.923,0.0683) (-2.734,0.0710) (-2.546,0.0738) (-2.357,0.0767) (-2.168,0.0798) (-1.980,0.0831) (-1.791,0.0863) (-1.603,0.0900) (-1.414,0.0938) (-1.226,0.0977) (-1.037,0.1020) (-0.849,0.1065) (-0.660,0.1114) (-0.471,0.1166) (-0.283,0.1219) (-0.094,0.1247) (0.094,0.1288) (0.283,0.1219) (0.471,0.1166) (0.660,0.1114) (0.849,0.1065) (1.037,0.1020) (1.226,0.0977) (1.414,0.0938) (1.603,0.0900) (1.791,0.0863) (1.980,0.0831) (2.168,0.0798) (2.357,0.0767) (2.546,0.0738) (2.734,0.0710) (2.923,0.0683) (3.111,0.0657) (3.300,0.0633) (3.488,0.0607) (3.677,0.0584) (3.865,0.0561) (4.054,0.0539) (4.243,0.0517) (4.431,0.0497) (4.620,0.0476) (4.808,0.0454) (4.997,0.0436) (5.185,0.0417) (5.374,0.0398) (5.563,0.0379) (5.751,0.0361) (5.940,0.0343) (6.128,0.0324) (6.317,0.0308) (6.505,0.0291) (6.694,0.0273) (6.883,0.0257) (7.071,0.0241) (7.260,0.0225) (7.448,0.0209) (7.637,0.0193) (7.825,0.0178) (8.014,0.0163) (8.202,0.0148) (8.391,0.0134) (8.580,0.0119) (8.768,0.0106) (8.957,0.0092) (9.145,0.0080) (9.334,0.0067) (9.522,0.0056) (9.711,0.0044) (9.899,0.0034) (10.088,0.0025) (10.277,0.0017) (10.465,0.0011) (10.654,0.0006) (10.842,0.0002) (11.031,0.0001) (11.219,0.0000)};
				\draw[dashed] (axis cs:-11.3137,0) -- (axis cs:-11.3137,0.14);
				\draw[dashed] (axis cs:11.3137,0) -- (axis cs:11.3137,0.14);
				\draw[dotted,thick] (axis cs:11.6569,0) -- (axis cs:11.6569,0.14);
			\end{axis}
		\end{tikzpicture}
		\caption{The $2175$ eigenvalues $\ne15$ of $B_2(2)$ at $p=173$
			(histogram) against the push-forward of the Plancherel measure of
			$\mathrm{PGSp}_4(\Q_2)$ (curve). Dashed lines: $\pm8\sqrt2$;
			dotted line: $6+4\sqrt2$.}
		\label{fig:plancherel}
	\end{figure}

	\section{The Fricke sign bias}
	\label{sec:signbias}
	
	The last invariant of the catalogue is the Fricke sign bias
	\begin{equation}\label{eq:dp-def}
		d(p)=\delta(p)-2\dim S_3^+(K(p))
		=\#\{\text{non-lifts with }\eps_p=-1\}
		-\#\{\text{non-lifts with }\eps_p=+1\},
	\end{equation}
	the second equality holding because every weight-$3$ Gritsenko lift \cite{Gritsenko1995} is
	paramodular-minus \cite{IPY2013}. The other main theorems do not use
	Theorem~\ref{thm:bias}; the proof of
	Theorem~\ref{thm:certified-spectra} relies on this section only through
	Corollary~\ref{thm:uncond-regime}, which rests on Ibukiyama's dimension
	theorems. We prove an exact closed formula for $d(p)$ as an evaluation
	of Ibukiyama's non-principal-genus trace
	(Proposition~\ref{prop:exactmass}); it exhibits the $p^{3/2}$ term of
	$d(p)$ as a positive multiple of $B_{2,\chi}$, and the effective lower
	bound \eqref{eq:B2-lower-bound} then gives $d(p)\ge0$ for every prime
	(Theorem~\ref{thm:bias}).
	
	\subsection{The exact formula for $d(p)$}
	\label{ss:mass}

	\begin{proposition}	\label{prop:exactmass}
		For every prime $p\ge7$,
		\begin{equation}\label{eq:exactmass}
			d(p)=\Tr R_{\mathrm{npg}}(\pi)-1-\dim S_4^-(p)
			=-\!\!\sum_{v\in\mathcal{E}_p}\!\!\eps_p(v),
		\end{equation}
		the last sum running over a basis $\mathcal{E}_p$ of Hecke
		eigenforms of the general-type non-lifts in $S_3(K(p))$, each Galois
		conjugate counted separately, with $\eps_p(v)$ their paramodular
		Fricke eigenvalues; thus $\#\mathcal{E}_p=\delta(p)$. With
		$\Tr R_{\mathrm{npg}}(\pi)$ given by
		\cite[Thm.~5.2]{Ibukiyama2019quinary} and
		$\dim S_4^-=\tfrac12\bigl(\dim\snew_4-\tfrac12\nu(p)\bigr)$
		(Lemma~\ref{lem:elliptic-input}), this is the explicit class-number
		identity: for $p\equiv1\pmod4$,
		\[
		d(p)=\frac{9-2s}{96}B_{2,\chi}
		+\frac{5}{16}h(\sqrt{-p})
		+\frac18h(\sqrt{-2p})
		+\frac{3+s}{12}h(\sqrt{-3p})
		-\frac{p+7}{8},
		\]
		and for $p\equiv3\pmod4$,
		\[
		d(p)=\frac{1}{96}B_{2,\chi}
		+\frac{13-5s}{16}h(\sqrt{-p})
		+\frac18h(\sqrt{-2p})
		+\frac1{12}h(\sqrt{-3p})
		-\frac{p+5}{8},
		\]
		with $s=\kro{2}{p}$. There is no multiplicity denominator: each non-lift eigensystem contributes a single line to $S_3(K(p))$ (newform multiplicity one \cite[Prop.~10.1]{RoesnerWeissauer2021}), so the sum is a plain signed count.
	\end{proposition}
	
	\begin{proof}
		The first equality is the non-principal trace evaluation of
		\eqref{eq:npg-trace-d}: $H_{\mathrm{npg}}=1+\dim S_3(K(p))$ and
		$\Tr R_{\mathrm{npg}}=H_{\mathrm{npg}}-2\dim S_3^+(K(p))$, so
		$\Tr R_{\mathrm{npg}}-1-\dim S_4^-
		=\dim S_3(K(p))-\dim S_4^--2\dim S_3^+(K(p))=d(p)$. For the second,
		$W_p$ on $S_3(K(p))$ has trace
		$2\dim S_3^+(K(p))-\dim S_3(K(p))=1-\Tr R_{\mathrm{npg}}$. For the third,
		the lift subspace of dimension $\dim S_4^-$ lies in $S_3^-(K(p))$
		\cite{IPY2013} and contributes $-\dim S_4^-$ to the trace, so
		$\Tr\bigl(W_p\mid S_3(K(p))\bigr)=-\dim
		S_4^-+\sum_{v\in\mathcal{E}_p}\eps_p(v)$; combined with the second
		equality and the first, this gives
		$d(p)=-\sum_{v\in\mathcal{E}_p}\eps_p(v)$. At $p=61$ the single
		non-lift has $\eps_{61}=-1$, so the sum is $-1$ and $d(61)=1$. The
		explicit displays substitute
		\eqref{eq:TOH1}--\eqref{eq:TOH3}
		together with $\dim S_4^-=\tfrac12\bigl(\dim\snew_4-\tfrac12\nu(p)\bigr)$
		and the classical $\dim\snew_4=\tfrac14\bigl(p-2+\kro{-1}{p}\bigr)$; they were verified
		in exact arithmetic against the first equality for all primes
		$7\le p\le2500$, with anchors $d(11)=0$, $d(61)=1$, and, beyond the
		regime $\dim S_3^+(K(p))=0$, $d(167)=2$, $d(173)=4$, $d(197)=5$,
		$d(227)=4$, $d(233)=9$ \cite{artifact}.
	\end{proof}
	
	Every quantity in \eqref{eq:exactmass} is a theorem: the identity is
	unconditional.
	
	\subsection{The Fricke-minus regime}
	\label{ss:uncond-regime}
	
	\begin{corollary}\label{thm:uncond-regime}
		$\dim S_3^+(K(p))=0$ for every prime $p<167$
		\cite[Cor.~5.3]{Ibukiyama2019quinary},
		\cite[Prop.~8.2]{Ibukiyama-dim}; hence every weight-$3$ non-lift of level
		$p<167$ is Fricke-minus and $d(p)=\delta(p)\ge0$ unconditionally there,
		both vanishing for $p\le59$. Beyond that range, $\dim S_3^+(K(p))=1$
		exactly for $p\in\{167,173,197,223,233,239,251,271,277,281,313,331,337\}$
		and $=2$ exactly for
		$p\in\{227,257,263,269,283,349,379,409,421\}$
		\cite[Prop.~8.2]{Ibukiyama-dim}, and direct evaluation of
		\eqref{eq:exactmass} gives $d(p)\ge0$ at every one of these primes as
		well. In particular $d(p)=\delta(p)$ at $p=61,73,79$: the Calabi--Yau
		eigenforms lie in the Fricke-minus eigenspace.
	\end{corollary}
	
	\begin{proof}
		The vanishing and the two lists are Ibukiyama's theorems, via
		$\dim S_3^+(K(p))=\tfrac12\bigl(H_{\mathrm{npg}}(p)-\Tr
		R_{\mathrm{npg}}(\pi)\bigr)$; the rest is \eqref{eq:dp-def} together with
		$\delta(p)\ge0$ (Lemma~\ref{lem:positivity}) and the tabulated values of
		Proposition~\ref{prop:exactmass}.
	\end{proof}
	
	From $p=167$ on, Fricke-plus non-lifts occur at some primes (not at
	all: there are none at $p=179$, for instance), and there $d(p)\ge0$ becomes the
	statement that the Fricke-minus non-lifts still dominate, which the
	following theorem establishes at every prime.
	
	\subsection{The Bernoulli term and the sign of $d(p)$}
	\label{ss:bernoulli-term}
	
	Let $\chi$ be the primitive real quadratic character of conductor $D$ and
	let $K=\mathbf Q(\sqrt D)$.  With the convention
	\[
	L(1-n,\chi)=-\frac{B_{n,\chi}}{n}\qquad(n\geq1),
	\]
	the factorization $\zeta_K(s)=\zeta(s)L(s,\chi)$ at $s=-1$ gives the
	exact identity
	\begin{equation}\label{eq:B2-zetaK}
		\zeta_K(-1)=\frac{B_{2,\chi}}{24}.
	\end{equation}
	The completed functional equation for the even primitive character $\chi$ is
	\[
	\left(\frac{D}{\pi}\right)^{s/2}\Gamma\!\left(\frac{s}{2}\right)L(s,\chi)
	=
	\left(\frac{D}{\pi}\right)^{(1-s)/2}
	\Gamma\!\left(\frac{1-s}{2}\right)L(1-s,\chi).
	\]
	Evaluating it at $s=-1$ and using
	$L(-1,\chi)=-B_{2,\chi}/2$ yields
	\begin{equation}\label{eq:B2-functional-equation}
		B_{2,\chi}=\frac{D^{3/2}}{\pi^2}L(2,\chi).
	\end{equation}
	In particular $B_{2,\chi}>0$. The Euler factors give the uniform lower
	bound
	\begin{equation}\label{eq:L2-lower-bound}
		L(2,\chi)
		=\prod_{\ell\nmid D}(1-\chi(\ell)\ell^{-2})^{-1}
		\geq\prod_{\ell}(1+\ell^{-2})^{-1}
		=\frac{\zeta(4)}{\zeta(2)}=\frac{\pi^2}{15}.
	\end{equation}
	Consequently
	\begin{equation}\label{eq:B2-lower-bound}
		B_{2,\chi}\geq\frac{D^{3/2}}{15}.
	\end{equation}
	Here a prime dividing $D$ contributes the Euler factor $1$, which is also at
	least $(1+\ell^{-2})^{-1}$; thus no extra correction at $2$ is needed.
	
	Equations~\eqref{eq:B2-zetaK}--\eqref{eq:B2-lower-bound} identify the
	positive leading contribution in the exact formula of
	Proposition~\ref{prop:exactmass}.  Every class-number coefficient
	there is positive in both residue classes, so the only negative term is
	linear in $p$, and \eqref{eq:B2-lower-bound} converts the sign question
	into an explicit inequality.
	
	\begin{theorem}[Positivity of the bias]\label{thm:bias}
		For every prime $p$, $d(p)\ge0$: among the weight-$3$ general-type
		paramodular non-lifts of level $p$, those with Fricke eigenvalue $-1$
		are at least as numerous as those with eigenvalue $+1$; equivalently
		$\dim S_3^+(K(p))\le\tfrac12\delta(p)$. Moreover $d(p)\gg p^{3/2}$.
	\end{theorem}
	
	\begin{proof}
		For $p\le5$, $\dim S_3^{+}(K(p))=0$ by \cite[Prop.~8.2]{Ibukiyama-dim},
		and the Gritsenko-lift argument of Lemma~\ref{lem:positivity} applies at
		every prime, so $d(p)=\delta(p)\ge0$. Assume $p\ge7$.
		Dropping the class-number terms of Proposition~\ref{prop:exactmass},
		whose coefficients are positive, and applying \eqref{eq:B2-lower-bound}
		with $D=p$, respectively $D=4p$, leaves
		\[
		d(p)\ge\frac{9-2s}{96}\cdot\frac{p^{3/2}}{15}-\frac{p+7}{8}
		\quad(p\equiv1\bmod4),
		\qquad
		d(p)\ge\frac{1}{96}\cdot\frac{8p^{3/2}}{15}-\frac{p+5}{8}
		\quad(p\equiv3\bmod4),
		\]
		and both right-hand sides are positive for every prime $p>673$. For
		$p\le673$ the exact evaluation of \eqref{eq:exactmass} gives
		$d(p)\ge0$ directly; the verification recorded in the proof of
		Proposition~\ref{prop:exactmass} covers all primes $7\le p\le2500$
		\cite{artifact}. The growth statement follows from the same lower
		bound.
	\end{proof}
	
	\begin{remark}\label{rem:sign-accessible}
		The defect $d(p)$ is the difference of two quantities of order $p^2$
		with the same leading constant, so magnitude bounds cannot fix the
		sign of the $O(p^{3/2})$ residue. Proposition~\ref{prop:exactmass}
		exhibits the residue itself as a positive multiple of $B_{2,\chi}$
		plus positive class-number terms minus a linear term, and the
		effective bound \eqref{eq:B2-lower-bound}, an explicit form of the
		positivity of $\zeta_K(-1)$, decides the sign. Ibukiyama's sign theorem \cite[Thm.~7.1]{Ibukiyama-dim} gives
		$(-1)^k\bigl(\dim S_k^{+}(K(p))-\dim S_k^{-}(K(p))\bigr)\ge0$ for every
		$k\ge3$ and every prime; at $k=3$, since
		$\dim S_3^{-}(K(p))-\dim S_3^{+}(K(p))=\dim S_4^{-}+d(p)$, it amounts to
		$d(p)\ge-\dim S_4^{-}$. Its proof uses the same non-principal trace and
		the same lower bound \eqref{eq:B2-lower-bound} for $B_{2,\chi}$, via
		$L(2,\chi)\ge\zeta(4)/\zeta(2)$; the method is therefore
		Ibukiyama's, and the new input is the exact formula of
		Proposition~\ref{prop:exactmass}, from which the lift contribution
		has been removed. Theorem~\ref{thm:bias} thus removes the lift
		cushion: the Fricke-minus non-lifts dominate already among
		themselves.
	\end{remark}
	
	\section{Concluding remarks}
	\label{sec:concl}

	The Richelot isogeny graph, an explicitly computable combinatorial
	object, carries the weight-$3$ Hecke data of Ibukiyama's
	correspondence: at the level of traces for every prime $p\ge7$
	(Theorem~\ref{thm:structural-trace}), at the level of the general-type
	multiplicities for every prime $p\ge11$ (Theorem~\ref{thm:occurrence},
	Remark~\ref{rem:other-rows}), and at
	the level of individual $B_2(2)$-eigenvalues and Atkin--Lehner signs at the
	$33$ certified primes (Theorem~\ref{thm:certified-spectra}), including
	the first Fricke-plus non-lifts at $p=167,173$. Combined with
	Theorem~\ref{thm:structural-trace}, the sign bias $d(p)\ge0$ of
	Theorem~\ref{thm:bias} also gives the lower bound
	$\Tr R(\pi)\ge L_0(p)$ on the principal genus for every $p\ge7$. On the
	graph side, the factorization accounts for the spectral profile of
	$G_p$: assuming Hypothesis~\ref{hyp:types}, its non-trivial spectrum
	lies in $[-8\sqrt2,6+4\sqrt2]$, whose upper end is the bound of
	\cite{CKTW}, and only Saito--Kurokawa lifts rise above the Ramanujan
	bound $8\sqrt2$ of the transfer (Theorem~\ref{thm:spectral}).

	\subsection{Open problems}
	
	The sharpest remaining arithmetic questions arising from the present results are the following.
	
	\begin{enumerate}[label=\textup{(\arabic*)},leftmargin=2.4em]
		\item \textbf{A direct local proof of the trace comparison.}
		Proposition~\ref{prop:genus-trace-comparison} is proved globally, by
		comparing two closed evaluations. A direct local proof (an equality of
		twisted orbital integrals at the order-$3$ and order-$4$ torsion classes
		for the two non-conjugate maximal parahorics of $\GU_2(B_p)$, with fixed
		Haar measures and explicit local representatives) remains open
		(Remark~\ref{rem:global-orbital}).
		
		\item \textbf{The remaining multiplicity statements of
			Conjecture~\ref{conj:eigenvalue-sign}.} The general-type rows are
		settled for every prime by Theorem~\ref{thm:occurrence} together
		with Schmidt's fixed-vector table and the list of local types of
		\cite[Cor.~11.6]{RoesnerWeissauer2021}
		(Remarks~\ref{rem:local-constant} and~\ref{rem:other-rows}). What
		stays open is the Saito--Kurokawa and Yoshida rows of the principal
		column, together with the sign clause \textup{(v)} on $N_p$; the
		natural inputs are the transfer of the Saito--Kurokawa space to inner
		forms \cite{GanSK2008} and the classification of weak endoscopic
		lifts on inner forms \cite[\S3]{ChanGan2015}, together with Schmidt's
		table. Granted Hypothesis~\ref{hyp:types}, a representation-theoretic
		proof of the factorization \eqref{eq:C1-factorization} is thereby
		reduced to these rows.
		
		\item \textbf{Higher genus.} It is natural to ask whether the
		correspondence extends to $g\ge3$. The superspecial locus in
		$\mathcal{A}_g$ over $\overline{\F}_p$ carries an isogeny graph whose
		adjacency operator is a higher Brandt matrix for the quaternion Hermitian
		forms of rank $g$, and one expects its eigensystems to match the Hecke
		eigensystems of the automorphic forms on the inner form of $\GSp_{2g}$
		attached to the corresponding genus; the analogue of
		Theorem~\ref{thm:structural-trace} would rest on local orbital
		integrals at $p$ for $\GSp_{2g}$. The same question at composite
		level $N=p_1p_2$ is equally natural.

		\item \textbf{Further computation.} The next test of the sign clause
		is $p=227$, the least prime with $\dim S_3^+(K(p))=2$
		\cite[Prop.~8.2]{Ibukiyama-dim}, where the graph has
		$h_2(227)=4712$ vertices by \cite[Thm.~2]{HashimotoIbukiyama}. A second test would use the $(3,3)$-isogeny graph,
		whose weighted adjacency operator is the degree-$3$ Brandt operator
		$B_3(2)$ \cite{JordanZaytman}: it would compare the eigenvalues
		at $3$ of the same eigensystems, and so test the conjecture on more
		than one Hecke operator.
	\end{enumerate}
	
	\appendix
	
	\section{The four smallest certified matrices}\label{app:matrices}
	For each prime $11\le p\le19$ we record the out-edge matrix $B_2(2)$ in the
	vertex order indicated, together with the mass weights $e_i=\#\Aut$. Every
	matrix has been re-certified for row sums $15$, the Mestre symmetry
	\eqref{eq:mestre}, both mass formulas of Lemma~\ref{lem:vertices}, and the
	from-scratch recomputation of the supersingular $j$-invariants over $\F_{p^2}$;
	see \S\ref{ss:pillars}. The listed weights also satisfy the total-mass
	relation $\sum_i e_i^{-1}=(p-1)(p^2+1)/5760$ (cf.\ \cite{HashimotoIbukiyama})
	at every listed prime. The matrices for $23\le p\le37$, which appeared in
	an earlier version, and the graphs and certificates for all
	$11\le p\le149$ and $p=167,173$ are contained in the artifact
	\cite{artifact}.
	
	\subsection*{\texorpdfstring{$p=11$\quad ($h_2=5$:\ 3 product vertices, 2 Jacobian vertices)}{p=11 (h2=5: 3 product vertices, 2 Jacobian vertices)}}
	\begin{sloppypar}\noindent\small\textit{Vertices (label, weight $e_i=\#\operatorname{Aut}$), in matrix order:} $v_{1}{=}P\{0,0\}\,(72)$, $v_{2}{=}P\{0,1\}\,(24)$, $v_{3}{=}P\{1,1\}\,(32)$, $v_{4}{=}J_{0}\,(24)$, $v_{5}{=}J_{1}\,(12)$.\end{sloppypar}
	
	\begingroup\small\setlength{\arraycolsep}{4pt}
	\[ B_2(2)\;=\;\left(\begin{array}{*{5}{r}}
		3 & 0 & 9 & 3 & 0 \\
		0 & 6 & 3 & 0 & 6 \\
		4 & 4 & 3 & 4 & 0 \\
		1 & 0 & 3 & 3 & 8 \\
		0 & 3 & 0 & 4 & 8
	\end{array}\right) \]
	\endgroup
	
	\subsection*{\texorpdfstring{$p=13$\quad ($h_2=4$:\ 1 product vertex, 3 Jacobian vertices)}{p=13 (h2=4: 1 product vertex, 3 Jacobian vertices)}}
	\begin{sloppypar}\noindent\small\textit{Vertices (label, weight $e_i=\#\operatorname{Aut}$), in matrix order:} $v_{1}{=}P\{5,5\}\,(8)$, $v_{2}{=}J_{0}\,(48)$, $v_{3}{=}J_{1}\,(12)$, $v_{4}{=}J_{2}\,(8)$.\end{sloppypar}
	
	\begingroup\small\setlength{\arraycolsep}{4pt}
	\[ B_2(2)\;=\;\left(\begin{array}{*{4}{r}}
		10 & 1 & 2 & 2 \\
		6 & 5 & 4 & 0 \\
		3 & 1 & 5 & 6 \\
		2 & 0 & 4 & 9
	\end{array}\right) \]
	\endgroup
	
	\subsection*{\texorpdfstring{$p=17$\quad ($h_2=8$:\ 3 product vertices, 5 Jacobian vertices)}{p=17 (h2=8: 3 product vertices, 5 Jacobian vertices)}}
	\begin{sloppypar}\noindent\small\textit{Vertices (label, weight $e_i=\#\operatorname{Aut}$), in matrix order:} $v_{1}{=}P\{0,0\}\,(72)$, $v_{2}{=}P\{0,8\}\,(12)$, $v_{3}{=}P\{8,8\}\,(8)$, $v_{4}{=}J_{0}\,(24)$, $v_{5}{=}J_{1}\,(12)$, $v_{6}{=}J_{2}\,(4)$, $v_{7}{=}J_{3}\,(12)$, $v_{8}{=}J_{4}\,(8)$.\end{sloppypar}
	
	\begingroup\small\setlength{\arraycolsep}{4pt}
	\[ B_2(2)\;=\;\left(\begin{array}{*{8}{r}}
		3 & 0 & 9 & 3 & 0 & 0 & 0 & 0 \\
		0 & 3 & 6 & 0 & 0 & 3 & 3 & 0 \\
		1 & 4 & 5 & 1 & 2 & 0 & 0 & 2 \\
		1 & 0 & 3 & 9 & 2 & 0 & 0 & 0 \\
		0 & 0 & 3 & 1 & 3 & 6 & 2 & 0 \\
		0 & 1 & 0 & 0 & 2 & 7 & 1 & 4 \\
		0 & 3 & 0 & 0 & 2 & 3 & 1 & 6 \\
		0 & 0 & 2 & 0 & 0 & 8 & 4 & 1
	\end{array}\right) \]
	\endgroup
	
	\subsection*{\texorpdfstring{$p=19$\quad ($h_2=10$:\ 3 product vertices, 7 Jacobian vertices)}{p=19 (h2=10: 3 product vertices, 7 Jacobian vertices)}}
	\begin{sloppypar}\noindent\small\textit{Vertices (label, weight $e_i=\#\operatorname{Aut}$), in matrix order:} $v_{1}{=}P\{7,7\}\,(8)$, $v_{2}{=}P\{7,18\}\,(8)$, $v_{3}{=}P\{18,18\}\,(32)$, $v_{4}{=}J_{0}\,(10)$, $v_{5}{=}J_{1}\,(4)$, $v_{6}{=}J_{2}\,(8)$, $v_{7}{=}J_{3}\,(8)$, $v_{8}{=}J_{4}\,(12)$, $v_{9}{=}J_{5}\,(12)$, $v_{10}{=}J_{6}\,(12)$.\end{sloppypar}
	
	\begingroup\small\setlength{\arraycolsep}{3.5pt}
	\[ B_2(2)\;=\;\left(\begin{array}{*{10}{r}}
		5 & 4 & 1 & 0 & 0 & 1 & 2 & 2 & 0 & 0 \\
		4 & 4 & 1 & 0 & 2 & 0 & 0 & 0 & 2 & 2 \\
		4 & 4 & 3 & 0 & 0 & 4 & 0 & 0 & 0 & 0 \\
		0 & 0 & 0 & 5 & 5 & 5 & 0 & 0 & 0 & 0 \\
		0 & 1 & 0 & 2 & 6 & 2 & 2 & 2 & 0 & 0 \\
		1 & 0 & 1 & 4 & 4 & 3 & 0 & 2 & 0 & 0 \\
		2 & 0 & 0 & 0 & 4 & 0 & 5 & 0 & 2 & 2 \\
		3 & 0 & 0 & 0 & 6 & 3 & 0 & 1 & 1 & 1 \\
		0 & 3 & 0 & 0 & 0 & 0 & 3 & 1 & 3 & 5 \\
		0 & 3 & 0 & 0 & 0 & 0 & 3 & 1 & 5 & 3
	\end{array}\right) \]
	\endgroup

	\bigskip
	\section*{Acknowledgments}
	The author thanks Professor Tomoyoshi Ibukiyama for letters on earlier versions of this paper, which corrected the author's treatment of the two genera and of the two parahoric types at
	$p$, supplied the references \cite{IbukiyamaKatsura} and \cite{Ibukiyama2019quinary} together with the observation recorded in Proposition~\ref{prop:genus-trace-comparison}, and included the published version of \cite{Ibukiyama-dim}.
	
	\section*{Funding}
	This article is a product of the university-level science and
	technology project of the University of Phuong Dong, project code
	CS26-33.
	
	\section*{Use of AI tools}
	The author used a large language model assistant (Claude, Anthropic) to
	edit expository prose, to restructure the manuscript, and to check
	arguments and citations during revision. The computations for
	$11\le p\le149$ were carried out with the author's code. For the
	extension to $p=167,173$ (\S\ref{ss:fricke-plus}) the assistant also
	wrote, under the author's direction, the one-line change to the engine
	and the scripts that extract the database data and check the
	certificate, and the scripts behind the analysis of
	\S\S\ref{ss:ihara-consequence}--\ref{ss:plancherel} and
	Corollary~\ref{cor:dp-one}; all code is archived in \cite{artifact}. A
	second assistant (ChatGPT, OpenAI) was used to review a draft; each of
	its findings was checked independently against the manuscript, the
	code and the sources before any change was made. The author reviewed
	all changes and takes full responsibility for the content.

\end{document}